\documentclass[a4paper,10pt,reqno]{article}

\usepackage[utf8]{inputenc}
\usepackage[T1]{fontenc}
\usepackage{lmodern}
\usepackage[english]{babel}
\usepackage{microtype}

\usepackage{amsmath,amssymb,amsfonts,amsthm}
\usepackage{mathtools,accents}
\usepackage{mathrsfs}
\usepackage{aliascnt}
\usepackage{braket}
\usepackage{bm}
\usepackage{dsfont}

\usepackage[a4paper,margin=3cm]{geometry}
\usepackage[citecolor=blue,colorlinks]{hyperref}

\usepackage{enumerate}
\usepackage{xcolor}

\makeatletter
\g@addto@macro\@floatboxreset\centering
\makeatother

\makeatletter
\def\newaliasedtheorem#1[#2]#3{
	\newaliascnt{#1@alt}{#2}
	\newtheorem{#1}[#1@alt]{#3}
	\expandafter\newcommand\csname #1@altname\endcsname{#3}
}
\makeatother

\numberwithin{equation}{section}

\newtheoremstyle{slanted}{\topsep}{\topsep}{\slshape}{}{\bfseries}{.}{.5em}{}

\theoremstyle{plain}
\newtheorem{theorem}{Theorem}[section]
\newaliasedtheorem{proposition}[theorem]{Proposition}
\newaliasedtheorem{lemma}[theorem]{Lemma}
\newaliasedtheorem{corollary}[theorem]{Corollary}
\newaliasedtheorem{counterexample}[theorem]{Counterexample}

\theoremstyle{definition}
\newaliasedtheorem{definition}[theorem]{Definition}
\newaliasedtheorem{question}[theorem]{Question}
\newaliasedtheorem{conjecture}[theorem]{Conjecture}

\theoremstyle{remark}
\newaliasedtheorem{remark}[theorem]{Remark}
\newaliasedtheorem{example}[theorem]{Example}

\newcounter{Counter}

\newtheorem{assumption}[Counter]{Assumption}

\newcommand{\setN}{\mathbb{N}}

\newcommand{\setR}{\mathbb{R}}

\newcommand{\eps}{\varepsilon}

\let\altphi\phi
\let\phi\varphi
\let\varphi\altphi
\let\altphi\undefined

\newcommand{\abs}[1]{\left\lvert#1\right\rvert}
\newcommand{\norm}[1]{\left\lVert#1\right\rVert}

\newcommand{\id}{\ensuremath{\mathds{1}}}

\DeclareMathOperator{\tr}{tr}
\let\div\undefined
\DeclareMathOperator{\div}{div}
\DeclareMathOperator{\Hess}{Hess}

\newcommand{\di}{\mathop{}\!\mathrm{d}}

\newcommand{\bs}{{\rm bs}}
\newcommand{\loc}{{\rm loc}}

\newcommand{\res}{\mathop{\hbox{\vrule height 7pt width .5pt depth 0pt
			\vrule height .5pt width 6pt depth 0pt}}\nolimits}

\newcommand{\restr}{\raisebox{-.1618ex}{$\bigr\rvert$}}
\DeclareMathOperator{\supp}{supp}

\newcommand{\Ch}{{\sf Ch}}
\newcommand{\HS}{\rm HS} 
\newcommand{\sym}{\rm sym}

\DeclareMathOperator{\Lip}{Lip}
\DeclareMathOperator{\Lipb}{Lip_b}
\DeclareMathOperator{\Lipbs}{Lip_\bs}

\DeclareMathOperator{\lip}{lip} 

\DeclareMathOperator{\esssup}{ess\,sup}

\newcommand{\leb}{\mathscr{L}}

\newcommand{\Prob}{\mathscr{P}}
\newcommand{\Probtwo}{\mathscr{P}_{2}}
\newcommand{\Borel}{\mathscr{B}}

\newcommand{\XX}{{\boldsymbol{X}}}

\newcommand{\dist}{\mathsf{d}}

\newcommand{\meas}{\mathfrak{m}}

\newcommand{\Liploc}{{\Lip_{\rm loc}}}

\newcommand{\Geo}{{\rm Geo}}
\newcommand{\Test}{{\rm Test}}

\newcommand{\Tan}{{\rm Tan}}
\DeclareMathOperator{\CD}{CD}
\DeclareMathOperator{\RCD}{RCD}
\DeclareMathOperator{\Ric}{Ric}
\DeclareMathOperator{\BV}{BV}
\newfont{\tmpf}{cmsy10 scaled 2500}

\def\Xint#1{\mathchoice
	{\XXint\displaystyle\textstyle{#1}}%
	{\XXint\textstyle\scriptstyle{#1}}%
	{\XXint\scriptstyle\scriptscriptstyle{#1}}%
	{\XXint\scriptscriptstyle\scriptscriptstyle{#1}}%
	\!\int}
\def\XXint#1#2#3{{\setbox0=\hbox{$#1{#2#3}{\int}$ }
		\vcenter{\hbox{$#2#3$ }}\kern-.6\wd0}}

\def\dashint{\Xint-}

\begin{document}
	
	\title{Constancy of the dimension for $\RCD(K,N)$ spaces via regularity of Lagrangian flows}
	\author{Elia Bru\'e 
		\thanks{Scuola Normale Superiore, \url{elia.brue@sns.it}.}\and Daniele Semola
		\thanks{Scuola Normale Superiore,
			\url{daniele.semola@sns.it}.}}  
	\maketitle
	
	\begin{abstract}
		We prove a regularity result for Lagrangian flows of Sobolev vector fields over $\RCD(K,N)$ metric measure spaces, regularity is understood with respect to a newly defined quasi-metric built from the Green function of the Laplacian.
		Its main application is that $\RCD(K,N)$ spaces have constant dimension. In this way we generalize to such abstract framework a result proved by Colding-Naber for Ricci limit spaces, introducing ingredients that are new even in the smooth setting.
	\end{abstract}

	\tableofcontents
	
	\section*{Introduction}

	It is well known that many analytical and geometrical properties of Riemannian manifolds are deeply related to lower bounds on the Ricci curvature. Moreover, the class of $n$-dimensional Riemannian manifolds with Ricci curvature uniformly bounded from below and diameter bounded from above being precompact with respect to Gromov-Hausdorff convergence (see \cite{Gromov81}) was the starting point for the study of the so-called Ricci-limit spaces, initiated by Cheeger-Colding in the series of works \cite{CheegerColding96,CheegerColding97,CheegerColding2000a,CheegerColding2000b}.\newline 
	Their deep analysis motivated the interest on finding a way to talk about Ricci curvature lower bounds without having a smooth structure at disposal, in analogy with the theory of Alexandrov spaces with sectional curvature bounded from below (see \cite{BuragoGromovPerelman92} and \cite[Appendix 2]{CheegerColding97}). Meanwhile, it became soon clear that Ricci curvature lower bounds should be seen as a property coupling the measure and the distance, in contrast with sectional curvature bounds, depending solely on the metric structure.
	
	The investigation around a synthetic treatment of lower Ricci bounds began with the seminal and   independent works by Lott-Villani \cite{LottVillani} and Sturm \cite{Sturm06a,Sturm06b} in which the class of $\CD(K,N)$ metric measure spaces was introduced with the aim to provide a synthetic notion of having Ricci curvature bounded from below by $K\in\setR$ and dimension bounded from above by $1\le N\le+\infty$. 
	The $\CD(K,N)$ condition was therein formulated in terms of convexity-type properties of suitable entropies over the Wasserstein space. Crucial properties of such a notion are the compatibility with the smooth Riemannian case and the stability with respect to measured Gromov-Hausdorff convergence.\newline 
	However the class of $\CD(K,N)$ metric measure spaces is still too large to some extent. For instance, it includes smooth Finsler manifolds (see the last theorem in \cite{Villani09}) which are known not to appear as Ricci limit spaces after the above mentioned works by Cheeger-Colding.\newline
	To single out spaces with a Riemannian-like behaviour from this broader class, Ambrosio-Gigli-Savaré introduced in \cite{AmbrosioGigliSavare14} the notion of metric measure space with Riemannian Ricci curvature bounded from below ($\RCD(K,\infty)$ m.m.s. for short), adding the request of linearity of the heat flow to the $\CD(K,\infty)$ condition. 
	Building upon this, the definition of $\RCD(K,N)$ metric measure spaces, which will be the main object of our study in this paper, was proposed by Gigli in \cite{Gigli15} as a finite-dimensional counterpart to the $\RCD(K,\infty)$ condition, coupling the $\CD(K,N)$ condition with the linearity of the heat flow.
	
	In the last years, this has been a considerably increasing research area, with several contributions by many authors (see for instance \cite{AmbGigliMondRaj12,AmbrosioGigliSavare14calc,AmbrosioMondinoSavare15,BacherSturm10,CavallettiMilman16,ErbarKuwadaSturm15,Gigli14,JangLiZhang,VonRenesse08}) that have often given new insights over more classical questions, both of analytical and geometric nature.
	\paragraph*{}
	
	With the aim of better introducing the reader to the statement of the main result of this note, let us briefly describe which was the state of the art around the so-called structure theory of $\RCD(K,N)$ metric measure spaces.
	
	Given an arbitrary metric measure space, there is a well defined notion of measured tangent space at a fixed point, as pointed measured Gromov-Hausdorff limit of a sequence of rescalings of the starting space. In particular, in the case of an $\RCD(K,N)$ metric measure space $(X,\dist,\meas)$, we can define, for any $1\le k\le N$, the $k$-dimensional regular set $\mathcal{R}_k$ to be the set of those $x\in X$ such that $x$ belongs to the support of $\meas$ and the tangent space of $(X,\dist,\meas)$ at $x$ is the $k$-dimensional Euclidean space. Better said, $x\in\mathcal{R}_k$ if $(X,r^{-1}\dist,\meas_r^x,x)\to(\setR^k,\dist_{\setR^k}c_k,\leb^k,0_{\setR^k},)$ as $r\downarrow0$, where
	\begin{equation*}
	\meas_r^x:=\left( \int_{B(x,r)} 1-\frac{\dist(x,y)}{r} \di \meas(y)\right)^{-1}\meas,\quad c_k:=\left(\int_{B(0,1)} \left(1-|y|\right) \di \leb^k(y)\right)^{-1}
	\end{equation*}  
	and the convergence is understood with respect to the pointed measured Gromov-Hasdorff topology.
	
	In \cite{MondinoNaber14} Mondino-Naber proved that, for any $\RCD(K,N)$ metric measure space $(X,\dist,\meas)$, the regular sets $\mathcal{R}_k$ are $(\meas,k)$-rectifiable and that
	\begin{equation}\label{eq:MN}
	\meas\left(X\setminus\bigcup_{k=1}^{[N]}\mathcal{R}_k\right)=0.
	\end{equation}
	Later on, their result was sharpened in the independent works \cite{MondinoKell16,GigliPasqualetto16,DePhilippisMarcheseRindler17}, where it was proved that $\RCD(K,N)$ spaces are rectifiable in the stronger sense of metric measures spaces, that is to say, for any $1\le k\le N$, the restriction of the reference measure $\meas$ to $\mathcal{R}_k$ is absolutely continuous with respect to the $k$-dimensional Hausdorff measure $\mathcal{H}^k$.
	
	Let us mention that a structure theory for Ricci limit spaces has been already developed by Cheeger-Colding in the aforementioned series of papers. Moreover, in \cite{CheegerColding97}, they conjectured that there should be exactly one $k$-dimensional regular set $\mathcal{R}_k$ having positive measure. However, it took more than ten years before the work \cite{ColdingNaber12}, where Colding-Naber affirmatively solved this conjecture. \newline
	The analogous problem in the abstract framework of $\RCD(K,N)$ metric measure spaces remained open since the work of Mondino-Naber and, building upon almost all the ingredients developed in this note, we are able to solve it, proving the following:
	
	\begin{theorem}[Constancy of the dimension]\label{thm:constintro}
		Let $(X,\dist,\meas)$ be an $\RCD(K,N)$ for some $K\in\setR$ and $1<N<+\infty$. Then, there is exactly one regular set $\mathcal{R}_k$ having positive $\meas$-measure in the Mondino-Naber decomposition of $(X,\dist,\meas)$.
	\end{theorem}
	
	\paragraph*{}
	
	In order to motivate the development of this work, we find it relevant to spend a few words, first trying to explain why it seems hard to adapt the strategy pursued by Colding-Naber in the case of Ricci limits to the setting of $\RCD(K,N)$ spaces and then to present the heuristic standing behind our new approach.
	
	The technique developed in \cite{ColdingNaber12} is based on fine estimates on the behaviour of balls of small radii centered along the interior of a minimizing geodesic over a smooth Riemannian manifold (with Ricci bounded from below) that are stable enough to pass through the possibly singular Gromov-Hausdorff limits.\newline
	When dealing with an abstract $\RCD(K,N)$ space there is no smooth approximating sequence one can appeal on. Nevertheless, one could try to reproduce their main  estimate (see \cite[Theorem 1.3]{ColdingNaber12}) directly at the level of the given metric measure space but, up to our knowledge, the calculus tools available at this stage, although being quite powerful (see for instance \cite{Gigli14}), are still not sufficient to such an issue.
	
	This being said, the study of the flow of a suitably chosen vector field, which was at the technical core of the proof in the case of Ricci limit spaces, will be the starting point in our approach too. A key idea is the following: one would expect the geometry to change continuously along a flow and that, as a consequence, flow maps might be a useful tool to prove that the space has a certain ``homogeneity'' property.\newline    
	Let us illustrate what we mean with a completely elementary example.
	Consider a smooth and connected differentiable manifold $M$. Given $x,y\in M$, there exists a diffeomorphism $\phi:M\to M$ such that $\phi(x)=y$. Moreover, a rather common way to build such a map is as flow map at a fixed time of a suitably chosen smooth vector field and we could rephrase this statement by saying that ``flows of smooth vector fields act transitively on $M$''. Of course this construction gives nothing more, at the level of the structure of $M$, than a confirmation of the fact that that different points in $M$ have diffeomorphic neighbourhoods. Instead, an adaptation to the non smooth setting of this strategy will have as deep consequence the proof of \autoref{thm:constintro} above.    
	
	Trying to pursue such a plan, we are left with the following questions: given an $\RCD(K,N)$ metric measure space $(X,\dist,\meas)$,
	\begin{itemize}
		\item[i)] can we find a notion of vector field, a notion of flow associated to it and a class of vector fields ``rich enough'' for the sake of the applications and ``regular enough'' to prove existence and uniqueness for such a generalized notion of flow?
		\item[ii)] What do we mean by rich enough in the question above? And, can we gain some transitivity in the spirit of the smooth elementary example?
		\item[iii)] Are the flows considered in i) regular in some sense?
		\item[iv)] Eventually, are the regularity in iii) and the transitivity in ii) strong enough to be incompatible with the possibility of having regular sets of different dimensions with positive measure in \eqref{eq:MN}? 
	\end{itemize}
	
	We might say that this whole note will be dedicated to the treatment of the above raised questions. In the case of i) and ii), this will be a matter of collecting ingredients that were already present in the literature. Instead, our main original contribution stands in \autoref{sec:Gregularity} and \autoref{sec:constancyofdimension} below, where we give positive answers to iii) and iv), respectively. 
	
	The remaining of this introduction is dedicated to a brief overview of this plan.
	\paragraph*{}
	
	It is known that any metric measure space has a first order differential structure (see for instance \cite{Gigli15,Weaver00} and \cite{Cheeger99}, dealing with the particular case of spaces satisfying doubling and Poincaré inequalities) which let it possible to give a meaning to the notion of vector field in such an abstract framework.\newline
	Furthermore in \cite{Gigli14}, building upon some results obtained by Savaré in \cite{Savare14}, Gigli proved that any $\RCD(K,\infty)$ metric measure space has also a second order differential structure (see also \cite{Sturm14}). Roughly speaking, this turns into the possibility of finding many functions with second order derivatives in $L^2$.
	At the level of vector fields, which are by themselves first order differential objects, this has the outcome that one can define a notion of covariant derivative and that the class of Sobolev vector fields with covariant derivative in $L^2$ is a rich one.
	Let us stress that the bounds on second derivatives being only of integral nature is one of the key differences between the theory of lower Ricci bounds and the theory of Alexandrov spaces, where lower bounds on the sectional curvature turn into the existence of many semi-convex functions thus having, at least formally, bounded second order derivatives (see \cite{Petrunin07}). 
	
	Nevertheless, having at disposal a rich family of vector fields, one can start investigating the associated theory of flows, trying to answer to the second part of question i).\newline
	Let us remark, however, that the study of ordinary differential equations and flows associated to non Lipschitz vector fields is a difficult task also in the Euclidean context, where there is a huge literature on this topic that still reserves challenging open problems and questions.
	The first contribution to this theory was due to Di Perna-Lions: motivated by the study of some PDEs in kinetic theory and fluid mechanics, in \cite{lions} they introduced a generalized notion of flow and studied the associated existence and uniqueness problem when the vector field has Sobolev regularity. Later on, their theory was revisited and extended to the case of vector fields with BV spatial regularity by Ambrosio in \cite{Ambrosio04}, where the notion of regular Lagrangian Flow was introduced as a good global selection of integral curves of the vector field.
	
	Keeping in mind what we observed at the beginning of this paragraph, when trying to develop a theory of flows on $\RCD(K,\infty)$ spaces it is more natural to look at a generalized theory of regular Lagrangian Flows than at the Cauchy-Lipschitz one. In \cite{AmbrosioTrevisan14}, Ambrosio-Trevisan introduced the natural extension to this framework of the notion of regular Lagrangian flow and proved well posedness for the associated existence and uniqueness problem for Sobolev vector fields over a large class of metric measure spaces including that one of $\RCD(K,\infty)$ spaces.

	\paragraph*{}
	Let us deal with question ii), that is to say the ``transitivity'' issue. Over an $\RCD(K,N)$ metric measure space $(X,\dist,\meas)$ a pointwise notion of transitivity might be out of reach. Nevertheless, what we learn from the literature about analysis on metric spaces is that some of the known constructions and results of the smooth category can be recovered in this fairly more general framework by looking at measures absolutely continuous with respect to the reference measure and curves of measures instead of points and curves (see \cite{AmbrosioGigliSavare14calc,AmbrosioColomboDiMarino15}, \cite{Gigli13} and \cite{GigliHan15a} for the application of this general idea to the construction of Sobolev spaces, the proof of the splitting theorem in non smooth context and the study of the continuity equation, respectively).\newline
	In this regard, here is a natural question towards a weak form of transitivity: is it true that, for any probability measures $\mu_0,\mu_1$ absolutely continuous and with bounded densities with respect to $\meas$, we can find a Sobolev vector field such that, calling $F$ its regular Lagrangian flow at a fixed time, it holds $F_{\sharp}\mu_0=\mu_1$? In \autoref{subsec:regularityintheinterior} below, we will see how the Lewy-Stampacchia inequality, proved in this abstract framework by Gigli-Mosconi in \cite{GigliMosconi15}, allows to give an (almost) affirmative answer to this question.
	
	\paragraph*{}
	
	Next we pass to the regularity issue. As in the case of the well-posedness problem, the study of regularity for Lagrangian flows associated to Sobolev vector fields is far from being trivial also in the Euclidean setting. The first result in this direction was obtained by Crippa-De Lellis in \cite{CrippaDeLellis08} building upon some ideas that have previously appeared in \cite{AmbrosioLecumberryManiglia}. 
	\newline
	We wish to explain the reasons behind these difficulties and why they required some new ideas to be settled in the case of a Riemannian manifold and a change of perspective in the case of an $\RCD(K,N)$ metric measure space.
	
	To this aim let $b:\setR^d\to \setR^d$ be a vector field and denote by $\XX:[0,T]\times\setR^d\to\setR^d$ its flow map, that we assume to be well-defined for every $t\in [0,T]$ and for every $x\in \setR^d$. A natural way to measure the regularity of $\XX$ is in terms of Lipschitz continuity and it is a rather elementary fact that, whenever $b$ is Lipschitz, the flow map $\XX_t$ is Lipschitz as well. Indeed, willing to control the distance between trajectories starting from different points $x,y\in\setR^d$, it is sufficient to compute
	\begin{equation}\label{computation1}
	\frac{\di}{\di t} |\XX_t(x)-\XX_t(y)|\leq |b(\XX_t(x))-b(\XX_t(y))|\leq \Lip(b)|\XX_t(x)-\XX_t(y)|,
	\end{equation}
	to obtain that
	\begin{equation*}
	|\XX_t(x)-\XX_t(y)|\le e^{t\Lip(b)} |x-y|, \qquad \text{for any $t\in [0,T]$.}
	\end{equation*}
	
	Lowering the regularity assumption on the vector field from Lipschitz to Sobolev, the second inequality in \eqref{computation1} fails and we cannot expect Lipschitz regularity for the regular Lagrangian flow $\XX_t$ that, in general, might even be discontinuous.
	However, in the aforementioned paper, Crippa-De Lellis obtained a Lusin-Lipschitz regularity result for Lagrangian flows associated to vector fields $b\in W^{1,p}(\setR^d;\setR^d)$ for $p>1$. That is to say, they proved that, for every bounded $K\subset\setR^d$ and for every $\eps>0$, there exist $C=C(\eps, \norm{b}_{W^{1,p}},K)>0$ and $E\subset K$ with $\leb^d(K\setminus E)<\eps$ such that $\XX_t$ is $C$-Lipschitz over $E$, for any $t\in [0,T]$.
	
	The key tool exploited by Crippa-De Lellis seeking for an analogue of \eqref{computation1} is the so-called maximal estimate for Sobolev functions: there exists $C_d>0$, such that any $f\in W^{1,p}(\setR^d;\setR)$ admits a representative, still denoted by $f$, such that
	\begin{equation}\label{MaximalEstimateScalar}
	|f(x)-f(y)|\leq C_d(M|\nabla f|(x)+M|\nabla f|(y))\abs{x-y},\qquad \text{for any $x,y\in X$,}
	\end{equation}
	where $M |\nabla f|$ is the maximal operator applied to $\abs{\nabla f}$. Observe that, if $p>1$, then $\norm{M|\nabla f|}_ {L^{p}}\le C_{p,d}\norm{\abs{\nabla f}}_{L^p}$ for some constant $C_{p,d}>0$. Moreover, since on $\setR^d$ a vector field is Sobolev if and only if its components are so, \eqref{MaximalEstimateScalar} holds true also for any $b\in W^{1,p}(\setR^d;\setR^d)$.
	
	This being said, the replacement of \eqref{computation1} in the Sobolev case is
	\begin{equation}\label{computation2}
	\frac{\di}{\di t} |\XX_t(x)-\XX_t(y)|\leq C \left\lbrace M|\nabla b|(\XX_t(x))+M|\nabla b|(\XX_t(y))\right\rbrace |\XX_t(x)-\XX_t(y)|.
	\end{equation}
	The sought regularity for $\XX_t$ does not follow any more applying Gronwall lemma to \eqref{computation2}. However, one might think of \eqref{computation2} as a quantitative infinitesimal version of the regularity result for the Lagrangian flow.
	
	Having such a perspective in mind, the situation changes significantly passing from the Euclidean space to an $\RCD(K,N)$ metric measure space or, more simply, to a smooth Riemannian manifold. 
	Indeed, while the maximal estimate for real valued Sobolev functions \eqref{MaximalEstimateScalar} is a very robust result, which holds true in every doubling metric measure space satisfying the Poincaré inequality and in a relevant class of non-doubling spaces (see \cite{AmbrosioColomboDiMarino15,AmbrosioBrueTrevisan17}), we are not aware of any intrinsic way to lift it to the level of vector fields. 
	
	Therefore we chose for an alternative approach to the problem. Let us introduce the more appealing notation $\dist$ for the distance function but still think, for sake of simplicity to the Euclidean case. Trying to turn the Sobolev regularity of the vector field into some bound for the right hand side in the expression 
	\begin{equation}\label{eq:derdist}
	\frac{\di}{\di t} \dist(\XX_t(x),\XX_t(y))= b\cdot \nabla \dist_{\XX_t(x)}(\XX_t(y))+ b\cdot\nabla \dist_{\XX_t(y)}(\XX_t(x)),
	\end{equation}
	a natural attempt could be to appeal to the interpolation
	\begin{equation}\label{z17}
	b\cdot\nabla \dist_x(y)+b\cdot\nabla \dist_y(x)=\int_0^1 \nabla_{\sym}b(\gamma'(s),\gamma'(s)) \di s,
	\end{equation}
	where $\gamma:[0,1]\to\setR^d$ is the geodesic joining $x$ to $y$ and $\nabla_{\sym}b$ is the symmetric part of the covariant derivative of $b$. However, when the bounds on $\nabla_{\sym}b$ are only of integral type, it is not clear how to obtain useful estimates from \eqref{z17} without deeply involving the Euclidean structure, that is something to be avoided in view of the extensions to the metric setting.
	
	The starting point of the study in our previous paper \cite{BrueSemola18} was, instead, the following observation: suppose that $d\ge 3$, then, calling $G$ the Green function of the Laplacian on $\setR^d$, it holds $G(x,y)=c_d\dist(x,y)^{2-d}$. This implies in turn that controlling the distance between two trajectories of the flow is the same as controlling the Green function along them. Moreover, computing the rate of change of the Green function along the flow, we end up with the necessity to find bounds for the quantity
	\begin{equation*}
	b\cdot\nabla G_x(y)+b \cdot \nabla G_y(x),
	\end{equation*}
	that, assuming $\div b=0$ for sake of simplicity, we can formally rewrite as
	\begin{align}\label{eq:intint}
	\nonumber b\cdot\nabla G_x(y)+b \cdot \nabla G_y(x)
	& = -\int_{\setR^d} b(w)\cdot \nabla G_x(w) \di\Delta G_y(w)-\int_{\setR^d}b(w)\cdot \nabla G_y(w) \di\Delta G_x(w)  \\
	& = 2 \int_{\setR^d} \nabla_{\sym}b(\nabla G_x, \nabla G_y) \di \meas.
	\end{align}
	Observe that, being \eqref{eq:intint} in integral form, we can expect it to fit better than \eqref{z17} with the assumption $\nabla_{\sym}b\in L^2$ and this expectation is confirmed by the validity, for some $C>0$, of the key estimate
	\begin{equation}\label{eq:Gintro}
	\int_{\setR^d} f|\nabla G_x||\nabla G_y|\di \meas\leq C G(x,y)(Mf(x)+Mf(y)),\quad \text{for any $x,y\in X$}
	\end{equation} 
	and for any Borel function $f:\setR^d\to[0,+\infty)$, see \autoref{prop:maximalGestimate} below. 
	
	Starting from this idea, in \cite{BrueSemola18} we obtained the first extension of the Lusin-Lipschitz regularity estimate for Lagrangian flows of Sobolev vector fields outside from the Euclidean setting, covering the case of Ahlfors regular, compact $\RCD(K,N)$ spaces. In that case, which includes for instance Riemannian manifolds, Alexandrov spaces and non-collapsed $\RCD(K,N)$ spaces (see \cite{DePhilippisGigli17}), a uniform control over the volume growth of balls turns into a global comparison between the Green function, that was introduced as an intermediate tool, and a negative power of the distance, by means of which we wanted to measure regularity.
	
	The key novelty in the approach of this paper is a change of perspective. Inspired by \cite{Colding12}, where the deep links between the Green function and the geometry of the space are investigated in the case of Riemannian manifolds with nonnegative Ricci curvature, in \autoref{thm: sobolev vectorfield, G-regularity} we prove that, on any $\RCD(0,N)$ metric measure space $(X,\dist,\meas)$ satisfying suitable volume growth assumptions at infinity, a version of Crippa-De Lellis' result holds true if we understand Lusin-Lipschitz regularity, not anymore with respect to the distance $\dist$, but with respect to a newly defined quasi-metric $\dist_{G}=1/G$, $G$  being the minimal positive Green function of the Laplacian over $(X,\dist,\meas)$. In \autoref{subsec:extnegative} we also adapt our arguments to cover the case of an arbitrary lower Ricci curvature bound.     
	
	Both our main regularity results, in their intrinsic and quantitative form, and some of the tools we develop in order to prove them are new, up to our knowledge, also in the case of smooth Riemannian manifolds.  
	
	\paragraph*{} 
	
	After finding the answers above to i), ii) and iii) we address iv). Let us point out, just at a speculative level, that, having at our disposal a perfect extension of Crippa-De Lellis' result to the metric setting, it would have been rather easy to exclude the possibility of regular sets of different dimensions with positive measure in the Mondino-Naber decomposition of an $\RCD(K,N)$ metric measure space, just building on the transitivity result in ii) and the observation that, given $k<n$, it is impossible to find a Lipschitz map $\Phi:\setR^k\to\setR^n$ such that $\Phi_{\sharp}\leb^k\ll\leb^n$.  
	
	Here we exploit a modification of this idea. As we explained at the very beginning of this introduction, the $\RCD(K,N)$ condition concerns neither the distance nor the reference measure by themselves but a coupling of these objects, as it happens for the Laplacian, the heat flow and the Green function. Moreover, it is possible to catch, in a quantitative way, the asymptotic behaviour of the Green function near to regular points of the metric measure space in terms of distance, measure and dimension (see \autoref{lemma:asymptotics}). This allows to find a counterpart for the ``preservation of the Hausdorff dimension via biLipschitz maps'' formulated just in terms of Green functions (see \autoref{thm:conservationhausdim}) and to complete the proof of \autoref{thm:constintro}, the spirit being that a control over two among distance, reference measure and Green function gives in turn a control over the remaining one.    
	\paragraph*{}
	
	Let us conclude this introduction with two comments. The first concerning the fact that our regularity result for Lagrangian flows, although being quantitative, enters in the proof of \autoref{thm:constintro} only in its qualitative form. The second is about its range of applications, that we expect to be wide in the study of fine structure properties of spaces with Ricci curvature bounded from below, both in the smooth and in the non-smooth category.    
	
	\paragraph*{}
	
	{\bf Acknowledgements.}
	The authors would like to thank Luigi Ambrosio for encouraging them in the study of this problem and for his kind and expert advise. They wish to express their gratitude to Nicola Gigli and Andrea Mondino for their helpful comments and numerous inspiring discussions around the topic of the paper and to the anonymous reviewer for the valuable report.\\ 
	The authors acknowledge the support of the PRIN2015 MIUR Project “Calcolo delle Variazioni”.

	\section{Preliminaries}
	
	We begin this preliminary part of the note introducing the basic terminology and notions about metric measure spaces, that will be the main object of our study. In \autoref{section: preliminaries RCD} and \autoref{sec:introflows} below we will focus on the class of finite dimensional $\RCD$ spaces and on the notion of Regular Lagrangian flow, respectively. 
	
	Throughout this note by metric measure space (m.m.s. for short) we mean a triple $(X,\dist,\meas)$ where $(X,\dist)$ is a complete and separable metric space and $\meas$ is a nonnegative Borel measure satisfying the exponential volume growth condition $\meas(B(x,r))\le M\exp(cr^2)$ for some constants $M,c\ge 0$.
	We shall adopt the standard metric notation. We indicate by $B(x,r)$ the open ball of radius $r>0$ centred at $x\in X$, by $\Lip(X)$, $\Liploc(X)$, $\Lipbs(X)$ the spaces of real valued Lipschitz, locally Lipschitz, and boundedly supported Lipschitz  functions over $(X,\dist)$, respectively. We denote by $\Lip f$ the global Lipschitz constant of $f\in \Lip(X)$. Moreover we introduce the notation
	\begin{equation*}
	\lip f(x):=
	\begin{cases*}
	\limsup_{y\to x}\frac{\abs{u(x)-u(y)}}{d(x,y)}& \text{if $x$ is not isolated,}\\
	0&\text{otherwise,}
	\end{cases*},
	\end{equation*}
	for the so-called slope of a function $f:X\to\setR$.
	
	By $\Geo(X)$ we indicate the space of (constant speed minimizing) geodesics on $(X,\dist)$, parametrized on $[0,1]$, endowed with the $\sup$ distance and by $e_t:\Geo(X)\to X, t\in[0,1]$ the evaluation map defined by $e_t(\gamma):=\gamma(t)$.
	
	We denote by $\Prob(X)$ the class of Borel probability measures on $(X,\dist)$ and by $\Probtwo(X)$ the subclass of those with finite second moment.\newline
	For $\mu_0,\mu_1\in\Probtwo(X)$ the quadratic Wasserstein distance $W_2(\mu_0,\mu_1)$ is defined by
	\begin{equation*}
	W_2^2(\mu_0,\mu_1):=\inf_{\pi}\int_{X\times X}\dist^2(x,y)\di\pi(x,y),
	\end{equation*}
	where the infimum runs over all $\pi\in\Prob(X\times X)$ with $\mu_0$ and $\mu_1$ as the first and the second marginal.
	It is well known that many geometric properties of $(X,\dist)$ are inherited by the Wasserstein space $(\Probtwo(X), W_2)$. For instance, if $(X,\dist)$ is a length space $(\Probtwo(X),W_2)$ is a length space itself. Moreover, it turns out that any geodesic $\mu_t\in\Geo(\Probtwo(X))$ can be lifted to a measure $\Pi\in\Prob(\Geo(X))$, that we shall call optimal geodesic plan, in such a way that $(e_t)_{\sharp}\Pi=\mu_t$ for any $t\in[0,1]$. \newline
	
	We denote by $L^p(X,\meas)=L^p(X)=L^p$ the spaces of Borel $p$-integrable functions over $(X,\meas)$ for any $1\le p\le+\infty$ and by $L^0(X,\meas)$ the space of $\meas$-measurable functions over $X$.\newline 
	For any $f\in L^1(X,\meas)$ we shall denote by $Mf$ the Hardy-Littlewood maximal function of $f$, which is defined by
	\begin{equation}\label{Maximal function}
	Mf(x):=\sup_{r>0}\dashint_{B(x,r)}\abs{f(z)}\di\meas(z),
	\qquad \forall x\in \supp\meas,
	\end{equation} 
	where
	\begin{equation*}
	\dashint_{B(x,r)}f(z)\di\meas(z):=\frac{1}{\meas(B(x,r))}\int_{B(x,r)}f(z)\di\meas(z).
	\end{equation*} 
	
	We go on with a brief discussion about Sobolev functions and Sobolev spaces over an arbitrary metric measure space, referring to \cite{AmbrosioGigliSavare13,AmbrosioGigliSavare14,Gigli14} and \cite{AmbrosioColomboDiMarino15} for a more detailed discussion about this topic. 
	
	\begin{definition}\label{def:Sobolevspace}
		For any $1<p<+\infty$ we define $W^{1,p}(X,\dist,\meas)(=W^{1,p}(X))$ to be the space of those $f\in L^{p}(X,\meas)$ such that there exist $f_n\to f$ in $L^p(X,\meas)$ with $f_n\in\Lip(X)$ for any $n\in\setN$ and $\sup_n\norm{\lip f_n}_{L^p}<+\infty$.
	\end{definition}
	
	The definition of Sobolev space is strongly related to the introduction of the Cheeger energy $\Ch_p:L^{p}(X,\meas)\to[0,+\infty]$ which is defined by
	\begin{equation}\label{eq:Cheeger}
	\Ch_p(f):=\inf\left\lbrace \liminf_{n\to\infty}\int \left(\lip f_n\right)^p\di\meas: f_n\to f\text{ in }L^p, f_n\in\Lip(X)\right\rbrace 
	\end{equation}
	and turns out to be a convex and lower semicontinuous functional from $L^p(X,\meas)$ to $[0,+\infty]$ whose finiteness domain coincides with $W^{1,p}(X,\dist,\meas)$.\newline
	By looking at the optimal approximating sequence in \eqref{eq:Cheeger} one can identify a distinguished object, called minimal relaxed gradient and denoted by $\abs{\nabla f}_p$, which provides the integral representation
	\begin{equation*}
	\Ch_p(f)=\int_X\abs{\nabla f}_p^p\di\meas,
	\end{equation*}    
	for any $f\in  W^{1,p}(X,\dist,\meas)$.
	As the notation suggests, $\abs{\nabla f}_p$ depends a priori on the integrability exponent $p$.\newline
	The space $W^{1,p}(X,\dist,\meas)$ is a Banach space when endowed with the norm $\norm{f}_{W^{1,p}}^p:=\norm{f}^p_{L^p}+\Ch_p(f)$. Moreover, the inequality $\abs{\nabla f}_p\le\lip f$ holds true $\meas$-a.e. on $X$ for any $f\in\Liploc(X)\cap W^{1,p}(X,\dist,\meas)$. 
	
	We further introduce, following \cite{AmbDiM14}, the space of functions of bounded variation over $(X,\dist,\meas)$.
	\begin{definition}
		For any $f\in L^1(X,\meas)$ and for any open set $A\subset X$ we introduce the relaxed total variation of $f$ over $X$ by
		\begin{equation*}
		\abs{Df}_*(A):= \inf\left\lbrace \liminf_{n\to \infty} \int_A \lip f_h\di\meas\ :\ f_h\in \Liploc(A),\  f_h\to f\ \text{in}\ L^1(A,\meas) \right\rbrace.
		\end{equation*}
		Moreover we let $\BV(X,\dist,\meas):=\set{f\in L^{1}(X,\meas)\ :\abs{Df}_*(X)<+\infty}$ be the space of functions with bounded variation on $(X,\dist,\meas)$.
	\end{definition}

	It is proved in \cite{MirandaJr03} that, for any $f\in \BV(X,\dist,\meas)$, the map $A\mapsto \abs{Df}_*(A)$ is the restriction to open sets of a finite Borel measure, for which we keep the same notation.
	It is then natural to let $W^{1,1}(X,\dist,\meas)$ be the space of those functions $f\in \BV(X,\dist,\meas)$ with the following property: there exists $\abs{\nabla f}_*\in L^1(X,\meas)$ such that $\abs{Df}_*=\abs{\nabla f}_*\meas$. The space $W^{1,1}(X,\dist,\meas)$ endowed with the norm $\norm{f}_{W^{1,1}}:=\norm{f}_{L^1}+\norm{\abs{\nabla f}_{*}}_{L^1}$ is a Banach space.
	
	Next we pass to the introduction of the spaces of locally Sobolev functions. For every $1<p<\infty$, we define $W^{1,p}_{\loc}(X,\dist,\meas)$ as the space of those $f\in L^p_{\loc}(X,\meas)$ such that $f\chi \in W^{1,p}(X,\dist,\meas)$ for every $\chi\in \Lip(X)$ with compact support. Using the strong locality property of the minimal relaxed gradient in $W^{1,p}(X,\dist,\meas)$, $1<p<\infty$, i.e.
	\begin{equation}\label{eq:stronglocality}
	|\nabla g|_p=|\nabla h|_p\quad \text{$\meas$-a.e. in $\set{h=g}$, for any $h,g\in W^{1,p}(X,\dist,\meas)$,}
	\end{equation}
	it is possible to define a minimal relaxed gradient $|\nabla f|_p\in L^p_{\loc}(X,\meas)$, which retains the same strong locality property, for every $f\in W^{1,p}_{\loc}(X,\dist,\meas)$.\newline
	In an analogous way one can define the space $W^{1,1}_{\loc}(X,\dist,\meas)$, exploiting the strong locality of the relaxed gradient $\abs{\nabla f}_*$ for $f\in W^{1,1}(X,\dist,\meas)$.

	\subsection{$\RCD(K,N)$ metric measure spaces}\label{section: preliminaries RCD}
	
	The notion of $\RCD(K,N)$ m.m.s. was proposed in \cite{Gigli15} (see also \cite{AmbrosioMondinoSavare15,ErbarKuwadaSturm15}), as a finite dimensional counterpart to $\RCD(K,\infty)$ m.m. spaces which were introduced and firstly studied in \cite{AmbrosioGigliSavare14} (see also \cite{AmbGigliMondRaj12}, dealing with the case of $\sigma$-finite reference measures). We point out that those spaces can be introduced and studied both from an Eulerian point of view, based on the so-called $\Gamma$-calculus, and from a Lagrangian point of view, based on optimal transportation techniques, which is the one we shall adopt in this brief introduction. 
	
	Let us start recalling the so-called curvature dimension condition $\CD(K,N)$. Its introduction dates back to the seminal and independent works \cite{LottVillani} and \cite{Sturm06a,Sturm06b}, while in this presentation we closely follow \cite{BacherSturm10}.
	
	\begin{definition}[Curvature dimension bounds]\label{def:CD}
		Let $K\in\setR$ and $1\le N<+\infty$. We say that a m.m.s. $(X,\dist,\meas)$ is a $\CD(K,N)$ space if, for any $\mu_0,\mu_1\in\Prob(X)$ absolutely continuous w.r.t. $\meas$ with bounded support, there exists an optimal geodesic plan $\Pi\in\Prob(\Geo(X))$ such that for any $t\in[0,1]$ and for any $N'\ge N$ we have 
		\begin{equation*}
		-\int\rho_t^{1-\frac{1}{N'}}\di\meas\le-\int\left\lbrace\tau_{K,N'}^{(1-t)}(\dist(\gamma(0),\gamma(1)))\rho_0^{-\frac{1}{N'}}(\gamma(0))+\tau_{K,N'}^{(t)}(\dist(\gamma(0),\gamma(1)))\rho_1^{-\frac{1}{N'}}(\gamma(1)) \right\rbrace\di\Pi(\gamma), 
		\end{equation*} 
		where $(e_t)_{\sharp}\Pi=\rho_t\meas$, $\mu_0=\rho_0\meas$, $\mu_1=\rho_1\meas$ and the distortion coefficients $\tau_{K,N}^{t}(\cdot)$ are defined as follows. First we define the coefficients $[0,1]\times[0,+\infty)\ni(t,\theta)\mapsto\sigma_{K,N}^{(t)}(\theta)$ by
		\begin{equation*}
		\sigma_{K,N}^{(t)}(\theta):=
		\begin{cases*}
		+\infty &\text{if $K\theta^2\ge N\pi^2$,}\\
		\frac{\sin(t\theta\sqrt{K/N})}{\sin(\theta\sqrt{K/N})}&\text{if $0<\theta<N\pi^2$,}\\
		t&\text{if $K\theta^2=0$,}\\
		\frac{\sinh(t\theta\sqrt{K/N})}{\sinh(\theta\sqrt{K/N})}&\text{if $K\theta^2<0$,}
		\end{cases*}
		\end{equation*}
		then we set $\tau_{K,N}^{(t)}(\theta):=t^{1/N}\sigma_{K,N-1}^{(t)}(\theta)^{1-1/N}$.
	\end{definition}
	
	The main object of our study in this paper will be $\RCD(K,N)$ spaces, that we introduce below.
	
	\begin{definition}\label{defRCDKN}
		We say that a metric measure space $(X,\dist,\meas)$ satisfies the \textit{Riemannian} $\CD(K,N)$ condition (it is an $\RCD(K,N)$ m.m.s. for short) for some $K\in\setR$ and $1\le N<+\infty$ if it is a $\CD(K,N)$ m.m.s. and the Banach space $W^{1,2}(X,\dist,\meas)$ is Hilbert.
	\end{definition}
	
	Notice that, if $(X,\dist,\meas)$ is an $\RCD(K,N)$ m.m.s., then so is $(\supp\meas,\dist,\meas)$, hence in the following we will always tacitly assume that $\supp\meas=X$.
	
	Let us point out that, in the last few years, many results have been proven for spaces satisfying the so-called \textit{reduced curvature dimension condition} $\CD^*(K,N)$ or \textit{reduced Riemannian curvature-dimension condition} $\RCD^*(K,N)$, which were known to have better localization and tensorization properties since the work \cite{BacherSturm10}. However, one of the main consequences of the recent work \cite{CavallettiMilman16} is that the classes of $\RCD^*(K,N)$ and $\RCD(K,N)$ spaces do actually coincide provided $\meas(X)<\infty$.
	
	\subsubsection{Geometric and analytical properties of $\RCD(K,N)$ spaces}\label{subsub:geoanproperties}
	
	Unless otherwise stated, from now on we let $(X,\dist,\meas)$ be an $\RCD(K,N)$ m.m.s. for some $K\in\setR$ and $1\le N<+\infty$.
	
	In \autoref{prop:finepropertiesgeodesics} below we collect some results concerning the improved regularity of $W_2$-geodesics on $\RCD(K,N)$ metric measure spaces. The results are mainly taken from \cite{Gigli12,Gigli13,Tapio12}.
	
	\begin{proposition}\label{prop:finepropertiesgeodesics}
		Let $(X,\dist,\meas)$ be an $\RCD(K,N)$ m.m.s. for some $K\in\setR$ and $1<N<+\infty$. Let $\mu_0,\mu_1\in\Probtwo(X)$ be absolutely continuous w.r.t. $\meas$, with bounded densities and bounded supports. Then:
		\begin{itemize}
			\item[(i)] there exists a unique $W_2$-geodesic $(\mu_{t})_{t\in[0,1]}$ joining $\mu_0$ and $\mu_1$. Moreover, it holds $\mu_t\le C\meas$ for any $t\in[0,1]$ for some $C>0$;
			\item[(ii)] letting $\rho_t$ be the density of $\mu_t$ w.r.t. $\meas$, it holds that, for any $t\in[0,1]$ and for any sequence $(t_{k})_{k\in\setN}$ converging to $t$, there exists a subsequence $(t_{n_k})_{k\in\setN}$ such that
			\begin{equation*}
			\rho_{t_{n_k}}\to\rho_t\quad \text{$\meas$-a.e. as $k\to\infty$.}
			\end{equation*}  
		\end{itemize}
	\end{proposition}
	
	As a non trivial geometric property of $\RCD(K,N)$ metric measure spaces, we recall that they satisfy the Bishop-Gromov inequality (which holds true more generally for any $\CD(K,N)$ m.m.s., see \cite{Sturm06b,Villani09}). That is to say
	\begin{equation}\label{eq:BG}
	\frac{\meas(B(x,R))}{\meas(B(x,r))}\le\frac{V_K(R)}{V_K(r)},
	\end{equation}
	for any $0<r<R$ and for any $x\in X$, where $V_K(s)$ stands for the volume of the ball of radius $s$ in the model space for the curvature-dimension condition $\CD(K,N)$. In particular, when $K\ge 0$, \eqref{eq:BG} implies that $(X,\dist,\meas)$ is doubling with doubling constant $2^N$, i.e.
	\begin{equation}\label{eq:BishopGromov}
	\meas(B(x,2r))\le 2^N\meas(B(x,r))\quad\text{for any $x\in X$ and for any $r>0$.}
	\end{equation} 
	In the case of a possibly negative lower Ricci curvature bound we can achieve the weaker conclusion that $(X,\dist,\meas)$ is locally uniformly doubling, that is to say, for any $R>0$ there exists $C_R>0$ such that
	\begin{equation}\label{eq:locdoubling}
	\meas(B(x,2r))\le C_R\meas(B(x,r))\quad \text{for any $x\in X$ and for any $0<r<R$.}
	\end{equation}
	As a consequence of \eqref{eq:locdoubling}, any $\RCD(K,N)$ m.m.s. $(X,\dist,\meas)$ is proper. 
	
	We recall that, when the space $(X,\dist,\meas)$ is doubling, the maximal operator $M$ is bounded from $L^p(X,\meas)$ into itself for any $1<p\le+\infty$, however if $X$ satisfies only the local doubling condition \eqref{eq:locdoubling} one can prove a following local version of this fact. 
	Let us fix $1<p\le \infty$ and a compact set $P\subset X$. Then, there exists a constant $C>0$, depending only on the diameter of $P$ and the local doubling constant of the space, such that for every $f\in L^{p}(X, \meas)$ with $\supp f\subset P$, it holds
	\begin{equation}\label{eq:LocalBoundnessMaximaloperator}
	\norm{Mf}_{L^p(P)}\leq C\norm{f}_{L^p(X)}.
	\end{equation}

	Below we pass to an overview about the consequences of the $\RCD(K,N)$ condition at the level of Sobolev calculus.
	
	One of the main contributions of \cite{GigliHan16} has been to prove, roughly speaking, that, under the $\RCD(K,\infty)$ assumption, the minimal relaxed slope $\abs{\nabla f}_p$ is independent of $p$, for any $1<p<\infty$.\newline
	Moreover (see \cite[Remark 3.5]{GigliHan16}), since $\RCD(K,N)$ spaces are proper as we already remarked, if $f\in W^{1,1}(X,\dist,\meas)\cap L^p(X,\meas)$ and $|\nabla f|_*\in L^p(X,\meas)$ for some $p>1$, then $f\in W^{1,p}(X,\dist,\meas)$ with $|\nabla f|_p=|\nabla f|_*$ $\meas$-a.e.. Vice versa, if $f\in W^{1,p}(X,\dist,\meas)\cap L^1(X,\meas)$ and $|\nabla f|_p\in L^1(X,\meas)$, then $f\in W^{1,1}(X,\dist,\meas)$ and $|\nabla f|_*=|\nabla f|_p$ $\meas$-a.e..\newline
	For these reasons we will omit the dependence on the integrability exponent in the notation for the minimal relaxed gradient.
	
	The following deep identification result was first proven by Cheeger in the seminal paper \cite{Cheeger99} in the context of metric measure spaces satisfying doubling and Poincaré inequalities. We refer to \cite[Theorem 8.4]{AmbrosioColomboDiMarino15} for the present formulation and for a different proof. We remark that \autoref{thm:lipugualegrad} applies in particular to $\RCD(K,N)$ spaces since, they are locally doubling (see \eqref{eq:locdoubling} above) and satisfy a local Poincaré inequality (see \cite{VonRenesse08}).
	
	\begin{theorem}\label{thm:lipugualegrad}
		Let $(X,\dist,\meas)$ be a m.m.s. with $\meas$ doubling, $\supp\meas=X$ and supporting a weak $(1,p)$-Poincaré inequality for some $1<p<+\infty$.
		Then, for any $f\in  W^{1,p}(X,\dist,\meas)\cap \Liploc(X)$, it holds $\lip f=\abs{\nabla f}_p$ $\meas$-a.e. on $X$.
	\end{theorem}
	
	Next we recall that the global assumption about $W^{1,2}(X,\dist,\meas)$ being a Hilbert space (a condition known in the literature as \textit{infinitesimal Hilbertianity}, which can be rephrased by saying that $\Ch$ is a quadratic form) has the following outcome.
	We can introduce a symmetric bilinear operator $\Gamma: W^{1,2}(X,\dist,\meas)\times W^{1,2}(X,\dist,\meas)\to L^{1}(X,\meas)$ by
	\begin{equation*}
	\Gamma(f,g)(x):=\frac{1}{4}\abs{\nabla (f+g)}^2(x)-\frac{1}{4}\abs{\nabla (f-g)}^2(x) \quad\text{for $\meas$-a.e. $x\in X$,}
	\end{equation*} 
	for any $f,g\in W^{1,2}(X,\dist,\meas)$. In the rest of the note we shall adopt the notation $\nabla f\cdot\nabla g$ to indicate $\Gamma(f,g)$.
	
	In order to introduce the heat flow and its main properties we begin by recalling the notion of Laplacian.
	
	\begin{definition}\label{def:laplacian}
		The Laplacian $\Delta:D(\Delta)\to L^2(X,\meas)$ is a densely defined linear operator whose domain consists of all functions $f\in W^{1,2}(X,\dist,\meas)$ satisfying 
		\begin{equation*}
		\int hg\di\meas=-\int \nabla h\cdot\nabla f\di\meas \quad\text{for any $h\in W^{1,2}(X,\dist,\meas)$}
		\end{equation*}
		for some $g\in L^2(X,\meas)$. The unique $g$ with this property is denoted by $\Delta f$.\footnote{The linearity of $\Delta$ follows from the quadraticity of $\Ch$.}
	\end{definition}
	
	More generally, we say that $f\in W^{1,2}_{\loc}(X,\dist,\meas)$ is in the domain of the measure valued Laplacian, and we write $f\in D(\boldsymbol{\Delta})$, if there exists a Radon measure $\mu$ on $X$ such that, for every $\psi\in\Lip(X)$ with compact support, it holds
	\begin{equation*}
	\int\psi\di\mu=-\int\nabla f\cdot\nabla \psi\di\meas.
	\end{equation*} 
	In this case we write $\boldsymbol{\Delta}f:=\mu$. If moreover $\boldsymbol{\Delta}f\ll\meas$ with $L^{2}_{\loc}$ density we denote by $\Delta f$ the unique function in $L^{2}_{\loc}(X,\meas)$ such that $\boldsymbol{\Delta}f=\Delta f\meas$ and we write $f\in D_{\loc}(\Delta)$.
	
	The heat flow $P_t$ is defined as the $L^2(X,\meas)$-gradient flow of $\frac{1}{2}\Ch$, whose existence and uniqueness follow from the Komura-Brezis theory. It can equivalently be characterized by saying that for any $u\in L^2(X,\meas)$ the curve $t\mapsto P_tu\in L^2(X,\meas)$ is locally absolutely continuous in $(0,+\infty)$ and satisfies
	\begin{equation*}
	\frac{\di}{\di t}P_tu=\Delta P_tu \quad\text{for $\leb^1$-a.e. $t\in(0,+\infty)$}.
	\end{equation*}  
	
	Under our assumptions the heat flow provides a linear, continuous and self-adjoint contraction semigroup in $L^2(X,\meas)$. Moreover $P_t$ extends to a linear, continuous and mass preserving operator, still denoted by $P_t$, in all the $L^p$ spaces for $1\le p<+\infty$.  
	
	In \cite{AmbrosioGigliSavare14,AmbGigliMondRaj12} it is proved that for $\RCD(K,\infty)$ metric measure spaces the dual semigroup $\bar{P}_t:\Probtwo(X)\to\Probtwo(X)$ of $P_t$, defined by
	
	\begin{equation*}
	\int_X f \di \bar{P}_t \mu := \int_X P_t f \di \mu\qquad\quad \forall \mu\in \Probtwo(X),\quad \forall f\in \Lipb(X),
	\end{equation*}
	is $K$-contractive (w.r.t. the $W_2$-distance) and, for $t>0$, maps probability measures into probability measures absolutely continuous w.r.t. $\meas$. Then, for any $t>0$, we can introduce the so called \textit{heat kernel} $p_t:X\times X\to[0,+\infty)$ by
	\begin{equation*}
	p_t(x,\cdot)\meas:=\bar{P}_t\delta_x.
	\end{equation*}  
	From now on, for any $f\in L^{\infty}(X,\meas)$ we will denote by $P_tf$ the representative pointwise everywhere defined by
	\begin{equation*}
	P_tf(x)=\int_X f(y)p_t(x,y)\di\meas(y).
	\end{equation*}
	
	Since $\RCD(K,N)$ spaces are locally doubling, as we already remarked, and they satisfy a local Poincaré inequality (see \cite{VonRenesse08}) the general theory of Dirichlet forms as developed in \cite{Sturm96} grants that we can find a locally H\"older continuous representative of $p$ on $X\times X\times(0,+\infty)$. 
	
	Moreover in \cite{JangLiZhang} the following finer properties of the heat kernel over $\RCD(K,N)$ spaces, have been proved: there exist constants $C_1>1$ and $c\ge0$ such that
	\begin{equation}\label{eq:kernelestimate}
	\frac{1}{C_1\meas(B(x,\sqrt{t}))}\exp\left\lbrace -\frac{\dist^2(x,y)}{3t}-ct\right\rbrace\le p_t(x,y)\le \frac{C_1}{\meas(B(x,\sqrt{t}))}\exp\left\lbrace-\frac{\dist^2(x,y)}{5t}+ct \right\rbrace  
	\end{equation}
	for any $x,y\in X$ and for any $t>0$. Moreover it holds
	\begin{equation}\label{eq:gradientestimatekernel}
	\abs{\nabla p_t(x,\cdot)}(y)\le \frac{C_1}{\sqrt{t}\meas(B(x,\sqrt{t}))}\exp\left\lbrace -\frac{\dist^2(x,y)}{5t}+ct\right\rbrace \quad\text{for $\meas$-a.e. $y\in X$},
	\end{equation}
	for any $t>0$ and for any $x\in X$.
	We remark that in \eqref{eq:kernelestimate} and \eqref{eq:gradientestimatekernel} above one can take $c=0$ whenever $(X,\dist,\meas)$ is an $\RCD(0,N)$ m.m.s..

	We go on by stating a few regularizing properties of the heat flow on $\RCD(K,N)$ spaces (which hold true more generally for any $\RCD(K,\infty)$ m.m.s.) referring again to \cite{AmbrosioGigliSavare14,AmbGigliMondRaj12} for a more detailed discussion and the proofs of these results.   
	
	First we have the \textit{Bakry-\'Emery contraction} estimate:
	\begin{equation}\label{eq:BakryEmery}
	\abs{\nabla P_t f}^2\le e^{-2Kt}P_t\abs{\nabla f}^2\quad \text{$\meas$-a.e.,}
	\end{equation}
	for any $t>0$ and for any $f\in W^{1,2}(X,\dist,\meas)$.
	
	Another non trivial regularity property is the so-called \textit{$L^{\infty}-\Lip$ regularization} of the heat flow, that is to say, for any $f\in L^{\infty}(X,\meas)$, we have $P_tf\in\Lip(X)$ with
	\begin{equation}\label{eq:linftylipregularization}
	\sqrt{2I_{2K}(t)}\Lip(P_tf)\le\norm{f}_{L^\infty},\quad\text{for any $t>0$},
	\end{equation}
	where $I_L(t):=\int_0^te^{Lr}\di r$ and $\Lip(P_tf)$ denotes the Lipschitz constant of $P_tf$.\newline
	Then we have the so-called \textit{Sobolev to Lipschitz property}: any $f\in W^{1,2}(X,\dist,\meas)$ with $\abs{\nabla f}\in L^{\infty}(X,\meas)$ admits a Lipschitz representative $\bar{f}$ such that $\Lip \bar{f}\le \norm{\nabla f}_{\infty}$.
	We also have a local version of the \textit{Sobolev to Lipschitz property}: any $f\in W^{1,2}_{\loc}(X,\dist,\meas)$ with $\abs{\nabla f}\in L^{\infty}(B(x,2r),\meas)$ for some $x\in X$ and $r>0$, admits a Lipschitz representative $\bar{f}$ in $B(x,r)$ such that $\Lip \bar{f}_{|_{B(x,r)}}\le \norm{\nabla f}_{L^\infty(B(x,2r),\meas)}$.

	Following \cite{Gigli14} we  introduce the space of ``test'' functions $\Test(X,\dist,\meas)$ by 
	\begin{equation}\label{eq:test}
	\Test(X,\dist,\meas):=\{f\in D(\Delta)\cap L^{\infty}(X,\meas): \abs{\nabla f}\in L^{\infty}(X)\quad\text{and}\quad\Delta f\in W^{1,2}(X,\dist,\meas) \}.
	\end{equation}
	We remark that, for any $g\in L^{2}\cap L^{\infty}(X,\meas)$, it holds that $P_tg\in \Test(X,\dist,\meas)$ for any $t>0$, thanks to \eqref{eq:BakryEmery}, \eqref{eq:linftylipregularization}, the fact that $P_t$ maps $L^2(X,\meas)$ into $D(\Delta)$ and the commutation $\Delta P_t f= P_t\Delta f$, which holds true for any $f\in D(\Delta)$.
	
	Below we state a useful result about the existence of smooth cut-off functions on $\RCD(K,N)$ spaces. Its proof can be obtained arguing as in \cite[Lemma 3.1]{MondinoNaber14}. 
	
	\begin{lemma}[Cut-off functions]\label{lemma:cutofffunctions}
		Let $(X,\dist,\meas)$ be an $\RCD(K,N)$ m.m.s. for some $K\in\setR$ and $1<N<+\infty)$. Then, for any $x\in X$ and for any $r>0$, there exists $\psi^r\in\Test(X,\dist,\meas)$ such that $0\le\psi^r\le 1$ on $X$, $\psi^r\equiv 1$ on $B(x,r)$ and $\supp\psi^r\subset B(x,2r)$.
	\end{lemma}

	\subsubsection{Structure theory of $\RCD(K,N)$ spaces}
	This subsection is dedicated to a brief overview about the so-called structure theory of $\RCD(K,N)$ metric measure spaces. The presentation is aimed at introducing the notation and the terminology that we shall adopt in the rest of the note. Moreover, the results below will play a deep role in the proof of \autoref{thm:constancyofdimension}.
	
	We assume the reader to be familiar with the notion of pointed measured Gromov Hausdorff convergence (pmGH-convergence for short), referring to \cite[Chapter 27]{Villani09} for an overview on the subject. Given a m.m.s. $(X,\dist,\meas)$, $x\in X$ and $r\in(0,1)$, we consider the rescaled and normalized pointed m.m.s. $(X,r^{-1}\dist,\meas_r^{x},x)$, where 
	\begin{equation*}
	\meas_r^x:=\left( \int_{B(x,r)} 1-\frac{\dist(x,y)}{r} \di \meas(y)\right)^{-1}\meas.
	\end{equation*}
	Then we introduce the following.
	
	\begin{definition}
		Let $(X,\dist,\meas)$ be an $\RCD(K,N)$ m.m.s. for some $1<N<+\infty$ and $K\in\setR$ and let $x\in X$. We say that a p.m.m.s. $(Y,\dist_Y,\eta,y)$ is tangent to $(X,\dist,\meas)$ at $x$ if there exists a sequence $r_i\downarrow 0$ such that $(X,r_i^{-1}\dist,\meas_{r_i}^x,x)\rightarrow(Y,\dist_Y,\eta,y)$ in the pmGH topology. The collection of all the tangent spaces of $(X,\dist,\meas)$ at $x$ is denoted by $\Tan(X,\dist,\meas,x)$.
	\end{definition}
	
	A compactness argument, which is due to Gromov, together with the rescaling and stability properties of the $\RCD(K,N)$ condition yields that $\Tan(X,\dist,\meas,x)$ is non empty for every $x\in X$ and its elements are all $\RCD(0,N)$ p.m.m. spaces.
	
	In \cite{GigliMondinoRajala15} it was proved that at $\meas$-a.e. $x\in X$ there exists at least an Euclidean space of dimension $1\le n\le N$ in $\Tan(X,\dist,\meas,x)$. This conclusion was greatly improved in \cite{MondinoNaber14}, where it was proved that $\RCD(K,N)$ spaces are rectifiable as metric spaces and that, up to an $\meas$-negligible set, the tangent space is always unique. To better state the main result proved therein we introduce the following.
	
	\begin{definition}[Regular $k$-dimensional set]\label{def:kregularset}
		Let $(X,\dist,\meas)$ be an $\RCD(K,N)$ m.m.s. for some $K\in\setR$ and $1<N<+\infty$. For any $k\in\setN$ the $k$-dimensional regular set $\mathcal{R}_k$ is the set of those $x\in X$ such that $\Tan(X,\dist,\meas,x)=\set{(\setR^k,\dist_{\setR^k},c_k\mathcal{H}^k,0_k)}$,
		where
		\begin{equation*}
		c_k:=\left(\int_{B(0,1)} \left(1-|y|\right) \di \leb^k(y)\right)^{-1}.
		\end{equation*}
	\end{definition}
	
	With the terminology above introduced we can rephrase \cite[Theorem 1.3]{MondinoNaber14} as follows.
	
	\begin{theorem}\label{thm:rectifiability}
		Let $(X,\dist,\meas)$ be an $\RCD(K,N)$ metric measure space for $K\in\setR$ and $1<N<+\infty$. Then
		\begin{equation*}
		\meas\left(X\setminus\bigcup_{k=1}^{[N]}\mathcal{R}_k\right)=0.
		\end{equation*}
		Moreover, there exists $\bar{\epsilon}=\bar{\epsilon}(K,N)>0$ such that, for every $0<\epsilon<\bar{\epsilon}(K,N)$, $\mathcal{R}_k$ is $(\meas,k)$-rectifiable via $(1+\epsilon)$-biLipschitz maps. That is to say, for any $\epsilon>0$ we can cover $\mathcal{R}_k$, up to an $\meas$-negligible set, with a countable family of subsets $\set{U^{k,i}_{\epsilon}}_{i\in\setN}$ with the property that each of them is $(1+\epsilon)$-biLipschitz to a subset of $\setR^k$.
	\end{theorem}
	
	The investigation was pushed further in the independent works \cite{MondinoKell16,GigliPasqualetto16,DePhilippisMarcheseRindler17}, where it was proved that $\RCD(K,N)$ spaces are rectifiable as metric measure spaces. We refer to \cite[Theorem 4.1]{AmbrosioHondaTewodrose17} for the present formulation.
	
	\begin{theorem}[Weak Ahlfors regularity]\label{thm:weakAhlforsregularity}
		Let $(X,\dist,\meas)$ be an $\RCD(K,N)$ metric measure space for some $K\in\setR$ and $1<N<+\infty$. Denote by
		\begin{equation}\label{eq:kregsub}
		\mathcal{R}_k^{*}:=\set{x\in\mathcal{R}_k:\quad\exists\lim_{r\to 0^+}\frac{\meas(B(x,r))}{\omega_kr^k}\in(0,+\infty)}.
		\end{equation}
		Then $\meas(\mathcal{R}_k\setminus\mathcal{R}_k^{*})=0$. Moreover $\meas\llcorner\mathcal{R}_k^{*}$ and $\mathcal{H}^{k}\llcorner\mathcal{R}_{k}^*$ are mutually absolutely continuous and 
		\begin{equation*}
		\lim_{r\to 0^+}\frac{\meas(B(x,r))}{\omega_kr^k}=\frac{\di\meas\llcorner\mathcal{R}_k^*}{\di\mathcal{H}^k\llcorner\mathcal{R}_k^*}(x),
		\end{equation*}
		for $\meas$-a.e. $x\in\mathcal{R}_k^{*}$.
	\end{theorem}

	\subsection{Regular Lagrangian flows and Sobolev vector fields}\label{sec:introflows}
	
	In this subsection we recall the notion of regular Lagrangian flow (RLF for short), firstly introduced in the Euclidean setting by Ambrosio in \cite{Ambrosio04}, inspired by the earlier work by Di Perna and Lions \cite{lions}. 
	It was defined as a generalized notion of flow in order to study ordinary differential equations associated to weakly differentiable vector fields.
	It is indeed well-known that, in general, it is not possible to define in a unique way a flow associated to a non Lipschitz vector field, since the trajectories starting from a given point might be non unique.
	
	We begin by introducing vector fields over $(X,\dist,\meas)$ as derivations over an algebra of test functions, following the approach adopted in \cite{AmbrosioTrevisan14,AmbrosioTrevisan15}.
	
	\begin{definition}\label{def:derivation}
		We say that a linear functional $b:\Lip(X)\to L^0(X,\meas)$ is a derivation if it satisfies the Leibniz rule, that is
		\begin{equation}\label{eq:Leibnizrule}
		b(fg)=b(f)g+fb(g),
		\end{equation}
		for any $f,g\in\Lip(X)$.
		
		Given a derivation $b$, we write $b \in L^p(TX)$ if there exists $g\in L^p(X,\meas)$ such that
		\begin{equation}\label{eq:continuityofderivation}
		b(f)\le g\abs{\nabla f} \quad\text{$\meas$-a.e. on $X$,}
		\end{equation}
		for any $f\in \Lip(X)$ and we denote by $\abs{b}$ the minimal (in the $\meas$-a.e. sense) $g$ with such property. We also say that $b$ has compact support if $|b|$ has compact support.
	\end{definition}
	
	We will use the notation $b\cdot\nabla f$ in place of $b(f)$ in the rest of the paper.
	We remark that, if a derivation $b$ is in $L^p(TX)$ for $1<p\le \infty$, then it can be extended in a unique way to a linear functional 
	\begin{equation*}
	b:W^{1,q}_{\loc}(X,\dist,\meas)	\to L^1_{\loc}(X,\meas)
	\end{equation*}
	still satisfying \eqref{eq:continuityofderivation}, where $q$ is the dual exponent of $p$. If $b\in L^p(TX)$ has compact support, then the associated derivation maps $W^{1,q}_{\loc}(X,\dist,\meas)$ into $L^1(X,\meas)$.\newline
	Let us remark that any $f\in W^{1,2}(X,\dist,\meas)$ defines in a canonical way a derivation $b_f\in L^2(TX)$ through the formula $b_f(g)=\nabla f\cdot\nabla g$, usually called the \textit{gradient derivation} associated to $f$.
	
	A notion of divergence can be introduced via integration by parts.
	
	\begin{definition}\label{def:divergence}
		Let $b$ be a derivation in $L^1(TX)$. We say that $\div b\in L^p(X,\meas)$ if there exists $g\in L^p(X,\meas)$ such that
		\begin{equation}\label{eq:divergence}
		\int_X b(f)\di\meas=-\int_Xgf\di\meas
		\end{equation}
		for any $f\in \Lipbs(X)$. By a density argument it is easy to check that such a $g$ is unique (when it exists) and we will denote it by $\div b$.
	\end{definition}
	We refer to \cite{Gigli14} for the introduction of the so-called tangent and cotangent moduli over an arbitrary metric measure space and for the identification results between derivations and elements of the tangent modulus $L^2(TX)$ which justifies the use of this notation, at least in the case $p=2$. Given a derivation $b\in L^{p}(TX)$ we shall denote by $\norm{b}_{L^p}:=\norm{\abs{b}}_{L^{p}(X,\meas)}$ its $L^p$-norm.
	
	A notion of time dependent vector field over $(X,\dist,\meas)$ can be introduced in the natural way.
	
	\begin{definition}\label{time dependent vector field}
		Let us fix $T>0$ and $p\in[1,+\infty]$. We say that $b:[0,T]\to L^p(TX)$ is a time dependent vector field if, for every $f\in W^{1,q}(X,\dist,\meas )$ (where $q$ is the dual exponent of $p$), the map
		$$
		(t,x)\mapsto b_t\cdot \nabla f(x)
		$$
		is measurable with respect to the product sigma-algebra $\leb^1 \otimes \Borel(X)$.
		We say that $b$ is bounded if
		$$
		\norm{b}_{L^{\infty}} :=\norm{|b|}_{L^{\infty}([0,T]\times X)}<\infty,
		$$
		and that $b\in L^1((0,T);L^p(TX))$ if
		$$
		\int_0^T \norm{b_s}_{L^p} \di s<\infty.
		$$
	\end{definition}
	
	In the context of $\RCD(K,\infty)$ spaces the definition of Regular Lagrangian flow reads as follows (see \cite{AmbrosioTrevisan14,AmbrosioTrevisan15}).\\ 	
		In the sequel we shall stress the dependence of a vector field $b$ on the time variable only in case it is relevant for the sake of clarity.
	\begin{definition}\label{def:Regularlagrangianflow}
		Let us fix a time dependent vector field $b$ (see \autoref{time dependent vector field}). We say that $\XX:[0,T]\times X\rightarrow X$ is a Regular Lagrangian flow associated to $b$ if the following conditions hold true:
		\begin{itemize}
			\item [1)] $\XX(0,x)=x$ and $X(\cdot,x)\in C([0,T];X)$ for every $x\in X$;
			\item [2)] there exists $L\ge0$, called compressibility constant, such that
			\begin{equation*}
			\XX(t,\cdot)_{\sharp} \meas\leq L\meas,\qquad\text{for every $t\in [0,T]$};
			\end{equation*}
			\item [3)] for every $f\in \Test(X,\dist,\meas)$ the map $t\mapsto f(\XX(t,x))$ belongs to $AC([0,T])$ for $\meas$-a.e. $x\in X$ and
			\begin{equation}\label{eq: RLF condition 3}
			\frac{\di}{\di t} f(\XX(t,x))= b_t\cdot \nabla f(\XX(t,x)) \quad \quad \text{for a.e.}\ t\in (0,T).	
			\end{equation}
		\end{itemize}
	\end{definition}
	The selection of ``good'' trajectories is encoded in condition 2), which is added to ensure that the RLF does not concentrate too much the reference measure $\meas$.
	
	We remark that the notion of RLF is stable under modification in a negligible set of initial conditions, but we prefer to work with a pointwise defined map in order to avoid technical issues.
	
	\begin{remark}\label{remark: dense class definition RLF}
		Under the additional assumption $b\in L^1((0,T); L^2(TX))$, equality \eqref{eq: RLF condition 3} holds true for every $g\in W^{1,2}(X,\dist,\meas)$ (where it is understood that in this case the map $t\mapsto g(\XX_t(x))$ belongs to $W^{1,1}((0,T))$ for $\meas$-a.e. $x\in X$) if and only if it holds for every $h\in D$ with $D\subset W^{1,2}(X,\dist,\meas)$ dense with respect to the strong topology. Indeed, if this is the case, for any $g\in W^{1,2}(X,\dist,\meas)$ and every $\eps>0$ we can find $h\in D$ such that $\norm{g-h}_{W^{1,2}(X,\dist,\meas)}< \eps$. Hence, since \eqref{eq: RLF condition 3} holds true for $h$, we can estimate
		\begin{align*}
		\int_X \Big|g(\XX(t,x))-g(x)-&\int_0^t b_s\cdot \nabla g(\XX(s,x))\di s\Big|^2 \di \meas(x)\\
		\leq &
		2\int_X \abs{g(\XX(t,x))-h(\XX(t,x))}^2 \di \meas(x)
		+ 2\int_X \abs{g(x)-h(x))}^2 \di \meas(x)\\
		& + 2\int_X \abs{\int_0^t b_s\cdot \nabla (g-h)(\XX(s,x))\di s}^2\di\meas(x)\\
		\le & 2(L+1)\norm{g-h}_{L^2(X,\meas)}^2+2L\norm{g-h}_{W^{1,2}(X,\dist,\meas)}^2 \sqrt{t}\int_0^t \norm{b_s}^2_{L^2} \di s\\
		\le & \eps^2 C(L,t,\norm{b}_{L^1((0,T);L^2(TX))}),
		\end{align*}
		that, together with an application of Fubini's theorem, implies the validity of \eqref{eq: RLF condition 3} for $g$.
		
		Moreover, one can easily prove, via a localization procedure, that also functions in the class $W^{1,2}_{\loc}(X,\dist,\meas)$ are admissible tests in \eqref{eq: RLF condition 3}.
	\end{remark}

	The theory of existence and uniqueness for regular Lagrangian flows in the context of $\RCD(K,\infty)$ metric measure spaces was developed by Ambrosio and Trevisan in \cite{AmbrosioTrevisan14}. The authors work with a very weak notion of symmetric covariant derivative for a vector field (see \cite[Definition 5.4]{AmbrosioTrevisan15}) and prove existence and uniqueness of the RLF associated to any bounded vector field $b$ with symmetric derivative in $L^2$ and bounded divergence, over an $\RCD(K,\infty)$ space (actually the results in \cite{AmbrosioTrevisan14} cover also more general settings).\newline
	When trying to develop a regularity theory for RLFs, as we shall do in the forthcoming \autoref{sec:Gregularity}, we will consider a class of vector fields smaller than the one for which existence and uniqueness hold. Nevertheless, this class will be still rich enough for the sake of the applications.\newline
	Below we introduce our working definition of Sobolev vector field with symmetric covariant derivative in $L^2$. We refer the reader to \cite[Chapter 1]{Gigli14} for the construction of the modulus $L^2(T^{\otimes2}X)$.

	\begin{definition}\label{def:Sobolevvectorfield}
		The Sobolev space $W^{1,2}_{C,s}(TX)\subset L^2(TX)$ is the space of all $b\in L^2(TX)$ with $\div b\in L^{2}(X,\meas)$ for which there exists a tensor $S\in L^{2}(T^{\otimes 2}X)$ such that, for any choice of $h,g_1,g_2\in\Test(X,\dist,\meas)$, it holds
		\begin{equation}\label{eq:Sobvectfield}
		\int h S(\nabla g_1,\nabla g_2)\di\meas=\frac{1}{2}\int\left\lbrace -b(g_2)\div(h\nabla g_1)-b(g_1)\div(h\nabla g_2)+\div(hb)\nabla g_1\cdot\nabla g_2\right\rbrace \di\meas.
		\end{equation}
		In this case we shall call $S$ the symmetric covariant derivative of $b$ and we will denote it by $\nabla_{\sym} b$. 
		We endow the space $W^{1,2}_{C,s}(TX)$ with the norm $\norm{\cdot}_{W^{1,2}_{C,s}(TX)}$ defined by
		\begin{equation*}
		\norm{b}^2_{W^{1,2}_{C,s}(TX)}:=\norm{b}^2_{L^2(TX)}+\norm{\nabla_{\sym} b}^2_{L^2(T^{\otimes 2}TX)}.
		\end{equation*}
	\end{definition}

	\begin{remark}\label{rm:sobolevcovimpliessobolevsym}
		It easily follows from the definition that the symmetric covariant derivative of any vector field in $W^{1,2}_{C,s}(TX)$ is a symmetric tensor.\newline
		Moreover, any $b\in W^{1,2}_C(TX)$ such that $\div b\in L^{2}(X,\meas)$ belongs to $W^{1,2}_{C,s}(TX)$ and $\nabla_{\sym}b$ is the symmetric part of $\nabla b$ (we refer to \cite[Section 3.4]{Gigli14} for the definition of space $W^{1,2}_C(TX)$ and of the associated notion of covariant derivative). 
		In particular, the assertion we made above about $W^{1,2}_{C,s}(TX)$ being a rich class is justified, since we know from \cite{Gigli14} that $W^{1,2}_C(TX)$ is dense in $L^2(TX)$ for any $\RCD(K,\infty)$ m.m.s. $(X,\dist,\meas)$.\newline
		Finally, let us remark that \autoref{def:Sobolevvectorfield} above is stronger than \cite[Definition 5.4]{AmbrosioTrevisan15}, as one can check passing to the moduli in \eqref{eq:Sobvectfield} and applying H\"older's inequality. Therefore, the theory developed by Ambrosio and Trevisan grants in particular existence and uniqueness of the RLF associated to any bounded vector field $b\in W^{1,2}_{C,s}(TX)$ with bounded divergence. 
	\end{remark}

	\section{$G$-regularity of Lagrangian flows}\label{sec:Gregularity}
	
	This section is dedicated to establish a regularity results for Lagrangian flows of Sobolev vector fields. As we anticipated in the introduction, the main novelty with respect to our previous work \cite{BrueSemola18} is that regularity is understood with respect to a newly defined quasi-metric $\dist_G=1/G$, where $G$ is the Green function of the Laplacian. 
	
	A natural setting to have existence of a positive Green function is that one of $\RCD(0,N)$ metric measure spaces satisfying suitable volume growth conditions (see assumptions \autoref{assumption: good definition of G} and \autoref{ass:doublingG} below). Under these assumptions, in \autoref{subsec:Gquasimetric} we prove that $(X,\dist_G,\meas)$ is a doubling quasi-metric space and in \autoref{subsec:Lusinreg} we exploit this structural result, together with the maximal estimate \eqref{eq:maxvectorvalued}, to implement the Crippa-De Lellis' scheme (see \cite{CrippaDeLellis08}).  
	
	In \autoref{subsec:extnegative} we show how this strategy can be adapted to cover the case of a possibly negative lower Ricci curvature bound.

	\subsection{Key properties of the Green function}\label{subsec:Green}
	In this section we study the properties of the Green function, that is a central object in our work. From now on we assume that $(X,\dist, \meas)$ is an $\RCD(0,N)$ metric measure space. Further assumptions on the space will be added in the sequel.
	
	We set
	\begin{equation*}
	G(x,y):=\int_{0}^{\infty}p_t(x,y)\di t
	\end{equation*}
	and, for every $\eps >0$, 
	\begin{equation}\label{eq:Geps}
	G^{\eps}(x,y):=\int_{\eps}^{\infty} p_t(x,y) \di t.
	\end{equation}
	We shall adopt in the sequel also the notation $G_x(\cdot):=G(x,\cdot)$ (and analogously for $G^{\epsilon}$).
	
	Before going on let us observe that, at least at a formal level, the Green function is the fundamental solution of the Laplace operator. Indeed
	\begin{equation*}
	\Delta_yG_x(\cdot)=\Delta_y\left(\int_{0}^{\infty}p_t(x,\cdot)\di t\right)=\int_0^{\infty}\Delta_yp_t(x,\cdot)\di t=\int_0^{\infty}\frac{\di}{\di t}p_t(x,\cdot)\di t=\left[ p_t(x,\cdot)\right]_{0}^{\infty}=-\delta_x. 
	\end{equation*}
	
	In order to get the good definition of both $G$ and $G^{\eps}$, up to the end of this section, unless otherwise stated, we will work under the following assumption.
	
	\begin{assumption}\label{assumption: good definition of G}
		There exists $x\in X$ such that 
		\begin{equation}\label{eq:nonparabolicity}
		\int_1^{\infty}\frac{s}{\meas(B(x,s))}\di s<\infty.
		\end{equation}
	\end{assumption}
	
	Recall that, for a non compact Riemannian manifold with nonnegative Ricci curvature, it was proved by Varopoulos that \eqref{eq:nonparabolicity} is a necessary and sufficient condition for the existence of a positive Green function of the Laplacian (and this condition is known as \textit{non-parabolicity} in the literature).

	\begin{remark}\label{remark:when assumption is satisfied}
		Let us observe that all the metric measure spaces obtained as tensor products between an arbitrary $\RCD(0,N)$ m.m.s. $(X,\dist,\meas)$ and an Euclidean factor $(\setR^k,\dist_{\setR^k},\leb^k)$ for $k\ge 3$ do satisfy assumption \autoref{assumption: good definition of G}. 
	\end{remark}
	
	We now introduce functions $F,\ H:X\times(0,+\infty)\to(0,+\infty)$ by 
	\begin{equation}\label{eq:introF}
	F(x,r):=\int_r^{\infty}\frac{s}{\meas(B(x,s))}\di s
	\end{equation}
	and 
	\begin{equation}\label{eq:introH}
	H(x,r):=\int_r^{\infty}\frac{1}{\meas(B(x,s))}\di s.
	\end{equation}
	They are the objects we will use to estimate the Green function and its gradient (see \cite{Grygorian06} for analogous results in the smooth setting). As for the Green function, we will often write $F_x(r)$ or $H_x(r)$ in place of $F(x,r)$ and $H(x,r)$.
	
	\begin{remark}\label{rm:continuityFH}
		Let us remark that both $F$ and $H$ are continuous w.r.t. the first variable. It can be seen recalling that spheres are negligible on doubling m.m.s and using the continuity of the function $x\mapsto \meas(B(x,r))$ (with $r>0$ fixed).
	\end{remark}
	
	The next proposition has the aim to provide estimates for the Green function and its gradient, in terms of $F_x(\dist(x,y))$ and $H_x(\dist(x,y))$ that are simpler objects to work with.
	
	\begin{proposition}[Main estimates for $G$]\label{prop:estimateG}
		Let $(X,\dist,\meas)$ be an $\RCD(0,N)$ m.m.s. satisfying assumption \autoref{assumption: good definition of G}.
		Then there exists a constant $C_2\ge 1$, depending only on $N$, such that, for any $x\in X$,
		\begin{equation}\label{eq:estimateG}
		\frac{1}{C_2}F_x(\dist(x,y))\le G_x(y)\le C_2 F_x(\dist(x,y))\quad\text{for any $y\in X$.}
		\end{equation}
		Moreover for any $x\in X$ it holds that $G_x\in W^{1,1}_{\loc}(X,\dist,\meas)$ and 
		\begin{equation}\label{eq:estimate nablaGeps}
		|\nabla G_x|(y)\le \int_0^{\infty} |\nabla p_t(x,\cdot)|(y) \di t
		\le C_2 H_x(\dist(x,y)),
		\quad \text{for $\meas$-a.e. $y\in X$.}
		\end{equation}
	\end{proposition}
	
	Before giving the proof of \autoref{prop:estimateG} let us state and prove some technical lemmas. The first one deals with the integrability properties of the maps $y\mapsto F_x(\dist(x,y))$ and $y\mapsto H_x(\dist(x,y))$. Since its formulation and its proof do not require any regularity assumption for the metric measure space, apart from the validity of assumption \autoref{assumption: good definition of G}, we state it in this great generality.
	
	\begin{lemma}\label{lemma: integrabilityFxHx} 
		Let $(X,\dist,\meas)$ be a m.m.s. satisfying assumption \autoref{assumption: good definition of G}. Then for every $x\in X$, the functions $y\mapsto F_x(\dist(x,y))$ and $y\mapsto H_x(\dist(x,y))$ belong to $L^1_{\text{\loc}}(X,\meas)$. 
		Moreover the map $(w,z)\mapsto H(w,\dist(w,z))$ belongs to $L^1_{\text{\loc}}(X\times X,\meas\times\meas)$.
	\end{lemma}
	\begin{proof}
		Let $g:\setR\to [0,+\infty)$ be a Borel function, define $f(r):=\int_r^{\infty}g(s) \di s$. Observe that
		\begin{equation}\label{eq: abstract integrability}
		\int_{B(x,R)} f(\dist_x(w)) \di \meas(w)=\int_0^R g(s)\meas(B(x,s))\di s
		+f(R)\meas(B(x,R)), \quad \text{for any $R>0$,}
		\end{equation}
		as an application of Fubini's theorem shows.
		Fix now any $x\in X$.
		Applying \eqref{eq: abstract integrability}, first with $g(s)=\frac{s}{\meas(B(x,s))}$ and then with $g(s)=\frac{1}{\meas(B(x,s))}$, we get
		\begin{equation}
		\int_{B(x,R)} F_x(\dist_x(w)) \di \meas(w)=\frac{R^2}{2}+F_x(R)\meas(B(x,R)),
		\end{equation}
		and
		\begin{equation}\label{z18}
		\int_{B(x,R)} H_x(\dist_x(w)) \di \meas(w)=R+H_x(R)\meas(B(x,R)),
		\end{equation}
		that imply in turn that $y\mapsto F_x(\dist(x,y))$ and $y\mapsto H_x(\dist(x,y))$ belong to  $L^1_{\loc}(X,\meas)$.
		
		We now prove the local integrability of $(w,z)\mapsto H(w,\dist(w,z))$. It suffices to show that
		\begin{equation}
		\int_{B(\bar{x},R)}\int_{B(\bar{x},R)} H(w,\dist(w,z)) \di \meas(z)\di \meas(w)<\infty,\qquad \forall R>0,\ \forall \bar{x}\in X.
		\end{equation}
		Observe that for every $w\in B(\bar{x},R)$ it holds $B(\bar{x},R)\subset B(w,2R)$. Hence
		\begin{align*}
		\int_{B(\bar{x},R)}\int_{B(\bar{x},R)} & H(w,\dist(w,z)) \di \meas(z)\di \meas(w)\\
		\leq &	\int_{B(\bar{x},R)}\int_{B(w,2R)} H(w,\dist(w,z)) \di \meas(z)\di \meas(w)\\
		=&\int_{B(\bar{x},R)}\left[  2R+\meas(B(w,2R))H_w(2R)\right]  \di \meas(w)\\
		\leq& 2R\meas(B(\bar{x},R))+\meas(B(\bar{x},3R))\int_{B(\bar{x},R)} H_w(2R) \di \meas(w),
		\end{align*}
		where we used \eqref{z18} passing from the second to the third line above.
		Since $B(\bar{x},s/2)\subset B(w,s)$ for every $w\in B(\bar{x},R)$ and $s>2R$, we obtain
		\begin{align*}
		\int_{B(\bar{x},R)} H_w(2R) \di \meas(w)=& \int_{2R}^{\infty} \int_{B(\bar{x},R)} \frac{1}{\meas(B(w,s))}\di \meas(w)\di s\\
		\leq & \int_{2R}^{\infty} \int_{B(\bar{x},R)} \frac{1}{\meas(B(\bar{x},s/2))}\di \meas(w)\di s\\
		=& \meas(B(\bar{x},R))\int_{2R}^{\infty}\frac{1}{\meas(B(\bar{x},s/2))}\di s<\infty.
		\end{align*}
	\end{proof}
	
	The following lemma deals with the regularity properties of $G_x^{\eps}$, that is a regular approximation of $G_x$.
	
	\begin{lemma}\label{lemma:G eps regularity}
		Let $(X,\dist,\meas)$ be an $\RCD(0,N)$ space satisfying assumption \autoref{assumption: good definition of G} and fix $x\in X$.
		For every $0<\eps<1$ the function $G^{\eps}_x$ belongs to $\Lip_b(X)\cap D_{\loc}(\Delta)$ and it holds $\Delta G^{\eps}_x=-p_{\eps}(x,\cdot)$.
		Moreover $G_x\in W_{\loc}^{1,1}(X,\dist,\meas)$ and 
		\begin{equation}\label{z10}
		\lim_{\eps\to 0} G^{\eps}_x= G_x\qquad \text{in}\ W^{1,1}_{\loc}(X,\dist,\meas).
		\end{equation}	
	\end{lemma}
	\begin{proof}
		First of all let us prove that $G^{\eps}_x\in L^{\infty}(X,\meas)$. Using \eqref{eq:kernelestimate} and assumption \autoref{assumption: good definition of G} we have
		\begin{align*}
		G^{\eps}_x(y) =  \int_{\eps}^{\infty}p_t(x,y) \di y
		\leq  \int_{\eps}^{\infty}\frac{C_1}{\meas(B(x,\sqrt{t}))} \di t
		=2C_1 \int_{\sqrt{\eps}}^{\infty} \frac{t}{\meas(B(x,t))}\di t<\infty.
		\end{align*}
		The proof of the regularity statement $G^{\eps}_x\in \Lipb(X)$ will follow after proving that the identity $G^{\alpha+t}_x=P_{t} G^{\alpha}_x$ holds true for any $\alpha,t\in(0,+\infty)$ by the regularization properties of the heat semigroup (since we proved that $G^{\alpha}\in L^{\infty}$).
		To this aim, for any $x,y\in X$ and for any $t,\alpha>0$, we compute
		\begin{align*}
		P_tG^{\alpha}_x(y)=&\int_X p_t(y,z)G^{\alpha}_x(z)\di\meas(z)
		=\int_{\alpha}^{\infty}\int_Xp_t(y,z)p_s(x,z)\di\meas(z)\di s\\
		=&\int_{\alpha}^{\infty}p_{t+s}(x,y)\di s
		=\int_{\alpha+t}^{\infty}p_s(x,y)\di s=G^{\alpha+t}_x(y).
		\end{align*}
		In order to prove that $G_x^{\eps}\in D_{\loc}(\Delta)$ and $\Delta G^{\eps}_x=p_{\eps}(x,\cdot)$ we consider a function $f\in \Test(X,\dist, \meas)$ and we compute
		\begin{equation*}
		\int_X G^{\epsilon}_x(w)\Delta f(w)\di \meas (w)=\int_{\eps}^{\infty} P_t \Delta f(x) \di t
		=-P_{\eps} f(x),
		\end{equation*}
		where the last equality follows from the observation that $P_rf\to 0$ pointwise as $r\to\infty$ for any $f\in L^1\cap L^2(X,\meas)$, that is a consequence of the estimates for the heat kernel \eqref{eq:kernelestimate} and the fact that $\meas(X)=\infty$.
		
		Let us prove \eqref{z10}. We preliminary observe that $G_x^{\eps}\to G_x$ in $L^1_{\loc}(X,\meas)$, since $G_x-G_x^{\eps}\geq 0$ and
		\begin{equation*}
		\int_X \left\lbrace G_x(y)-G^{\eps}_x(y)\right\rbrace  \di \meas(y)=
		\int_X\int_{0}^{\epsilon}p_t(x,y)\di t\di \meas(y)=\int_0^{\eps} \int_X p_t(x,y) \di \meas(y) \di t=\eps.
		\end{equation*}
		To conclude the proof it suffices to show that $G_x^{\eps}$ is a Cauchy sequence in $W_{\loc}^{1,1}(X,\dist,\meas)$. We claim that, for every $0<\eps_1<\eps_2<1$,
		\begin{equation}\label{eq: cauchynablaGeps}
		|\nabla (G_x^{\eps_1}-G_x^{\eps_2})|(y)=\lip(G_x^{\eps_1}-G_x^{\eps_2})(y)\leq \int_{\eps_2}^{\eps_1} \lip p_t(x,\cdot)(y) \di t,
		\quad \text{for $\meas$-a.e. $y\in X$.}
		\end{equation} 
			As a consequence of the Bishop-Gromov inequality \eqref{eq:BishopGromov} we get 
			\begin{align*}
			\sup_{t>0} \int_X\frac{e^{- \frac{\dist^2(x,y)}{5t}}}{\meas(B(x,\sqrt{t}))}\di \meas(y)
			=& \sup_{t>0} \frac{1}{\meas(B(x,\sqrt{t}))}\int_X \int_{\dist^2(x,y)/t}^{\infty} \frac{e^{-s/5}}{5} \di s \di \meas(y)\\
			=& \sup_{t>0} \int_0^{\infty}\frac{ e^{-s/5}}{5}\frac{\meas(B(x,\sqrt{st}))}{\meas(B(x,\sqrt{t}))} \di s\\
			\leq & \int_0^{\infty} \frac{ e^{-s/5}}{5}\max\set{s;1}^{N/2} \di s<\infty,
			\end{align*}
			that, together with the estimates for the gradient of the heat kernel \eqref{eq:gradientestimatekernel}, implies
		\begin{equation*}
		\int_0^1\int_X |\nabla p_t(x,\cdot)|(y) \di \meas(y) \di t\leq\int_0^1 \frac{C_2}{\sqrt{t}}\int_X \frac{e^{- \frac{\dist^2(x,y)}{5t}}}{\meas(B(x,\sqrt{t}))}\di \meas(y)\di t<\infty,
		\end{equation*} 
		therefore \eqref{eq: cauchynablaGeps} will yield the desired conclusion.
		This being said let us pass to the verification of \eqref{eq: cauchynablaGeps}. Observe that the $\meas$-a.e. identifications between slopes and minimal weak upper gradients above follow from the local Lipschitz regularity of the heat kernel and $G_{x}^{\epsilon}$ for $\epsilon>0$ thanks to \autoref{thm:lipugualegrad}. Observe that the very definition of $G^{\epsilon}$ grants that
		\begin{equation}\label{eq:firststepFatou}
		\lip(G_x^{\eps_1}-G_x^{\eps_2})(y)\leq \limsup_{z\to y} \int_{\eps_2}^{\eps_1}  \frac{ |p_t(x,y)-p_t(x,z)|}{\dist(y,z)} \di t,
		\quad \text{for every}\ y\in X.
		\end{equation}
		Moreover, for any $r<\frac{1}{2}d(x,y)$, the gradient estimate for the heat kernel \eqref{eq:gradientestimatekernel} yields
		\begin{equation}\label{z20}
		|\nabla p_t(x,\cdot)|(w)\leq \frac{C_1e^{ -\frac{r^2}{5t}}}{\sqrt{t}\meas(B(x,\sqrt{t}))}\quad \text{for $\meas$-a.e $w\in B(y,r)$.}
		\end{equation}
		Hence $p_t(x,\cdot)$ is Lipschitz in $B(y,r/2)$ with Lipschitz constant bounded from above by the right hand side of \eqref{z20}, thanks to a local version of the \textit{Sobolev to Lipschitz property} (see \autoref{section: preliminaries RCD}).
		Summarizing we obtain the bound
		\begin{equation}
		\frac{ |p_t(x,y)-p_t(x,z)|}{\dist(y,z)}\le \frac{C_1e^{ -\frac{r^2}{5t}}}{\sqrt{t}\meas(B(x,\sqrt{t}))},
		\end{equation}
		for every $z\in B(y,r/2)$ and every $t\in (0,\infty)$. Hence we can apply Fatou's lemma and pass from \eqref{eq:firststepFatou} to \eqref{eq: cauchynablaGeps}. 
	\end{proof}
	
	\begin{remark}\label{rm::testlocregularity}
		Proceeding as in the proof of \autoref{lemma:G eps regularity} above, one can prove that, for any $\eta\in \Test(X,\dist,\meas)$ with compact support, it holds that $\eta G^{\epsilon}_x\in\Test(X,\dist,\meas)$ for any $x\in X$ and for any $\epsilon>0$.
	\end{remark}

	We state another technical lemma, its elementary proof can be obtained with minor modifications to the proof of \cite[Lemma 5.50]{Grygorian06}.
	
	\begin{lemma}\label{lemma:extimateintegral}
		Let $\phi:(0,+\infty)\to(0,+\infty)$ be monotone increasing and set
		\begin{equation*}
		\psi(r):=\int_0^{+\infty}\frac{1}{\phi(\sqrt{t})}\exp\left(-\frac{r^2}{t}\right)\di t.
		\end{equation*}
		If $\phi$ satisfies the local doubling property
		\begin{equation*}
		\phi(2r)\le C(R)\phi(r)\quad\text{for any $0<r<R$,}
		\end{equation*}
		for some non decreasing function $C:(0,+\infty)\to(0,+\infty)$,
		then there exists a non decreasing function $\Lambda:(0,+\infty)\to(0,+\infty)$, whose values depend only on the function $C$, such that
		\begin{equation}\label{eq:soughtbound}
		\frac{1}{\Lambda(R)}\int_r^{\infty}\frac{s}{\phi(s)}\di s\le\psi(r)\le \Lambda(R)\int_r^{\infty}\frac{s}{\phi(s)}\di s,
		\end{equation}
		for any $0<r<R$ and for any $R\in (0,+\infty)$.
		Moreover, when $C$ is constant, we can choose $\Lambda$ to be constant.
	\end{lemma}

	\begin{proof}[Proof of \autoref{prop:estimateG}]
		The proof of \eqref{eq:estimateG} follows from the estimates for the heat kernel \eqref{eq:kernelestimate} applying \autoref{lemma:extimateintegral} with $\phi(r):=\meas(B(x,r))$.
		
		In order to prove \eqref{eq:estimate nablaGeps} we observe that, arguing exactly as in the proof of \eqref{z10}, one can prove that, for any $\epsilon>0$ and any $x\in X$,
		\begin{equation}\label{eq:slope estimate G}
		\abs{\nabla G^{\eps}_x}(y)\leq \int_{\eps}^{\infty} |\nabla p_t(x,\cdot)|(y)\quad\text{for $\meas$-a.e. $y\in X$.}
		\end{equation}
		The sought conclusion follows from \eqref{z10}. 
		The proof of the inequality 
		\begin{equation*}
		\int_0^{\infty} |\nabla p_t(x,\cdot)|(y) \di t \le C_2 H_x(\dist(x,y)),
		\qquad\ \text{for}\ \meas\text{-}\text{a.e.}\ y\in X
		\end{equation*}
		follows from the gradient estimate for the heat kernel \eqref{eq:gradientestimatekernel}, applying \autoref{lemma:extimateintegral} with choice $\phi(r):=r\meas(B(x,r))$.
	\end{proof}
	
	\begin{remark}\label{remark:Geps estimate}
		It is clear from the proof of \autoref{prop:estimateG} that the regularized functions $G^{\epsilon}$ satisfy
		\begin{equation}
		|\nabla G^{\eps}_x|(y)\le \int_{\eps}^{\infty} |\nabla p_t(x,\cdot)|(y) \di t
		\le C_2 H_x(\dist(x,y)),
		\qquad\ \text{for}\ \meas\text{-}\text{a.e.}\ y\in X.
		\end{equation}
	\end{remark}

	\begin{remark}\label{rm:continuityfarfrompole}
		As a consequence of \eqref{eq:estimate nablaGeps} and of the continuity of the map $x\mapsto H_x(r)$, exploiting the monotonicity w.r.t. $r$ of $H$ and a local version of the \textit{Sobolev to Lipschitz} property, one can prove that $G_{x}$ is continuous in $X\setminus\set{x}$. 
	\end{remark}
	
	Let us state and prove a maximal estimate that, as we anticipated in the introduction, is a key tool to bound the rate of change of the Green function along trajectories of a Lagrangian flow. It will be crucial in the proof of the vector-valued maximal estimate \autoref{prop: G-maximal estimate}.
	
	\begin{proposition}[Maximal estimate, scalar version]\label{prop:maximalGestimate}
		Let $(X,\dist,\meas)$ be an $\RCD(0,N)$ m.m.s. satisfying assumption \autoref{assumption: good definition of G}.
		Then there exists $C_M>0$, depending only on $N$, such that, for any Borel function $f:X\to[0,+\infty)$, it holds
		\begin{equation}\label{eq:auxiliaryestimate}
		\int f(w)\abs{\nabla G_x(w)}\abs{\nabla G_y(w)}\di\meas(w)\le C_M G(x,y)\left(Mf(x)+Mf(y)\right),
		\end{equation}
		for every $x,y\in X$.
	\end{proposition}
	
	\begin{proof}
		Fix two different points in $x,y\in X$.
		Thanks to \eqref{eq:estimate nablaGeps} we can estimate the left hand side of \eqref{eq:auxiliaryestimate} with
		\begin{equation*}
		C_2^2\int_0^{\infty}\int_0^{\infty}\int_X f(w) \frac{\id_{B(x,r)}(w)}{\meas (B(x,r))}\frac{\id_{B(y,s)}(w)}{\meas (B(y,s))} \di \meas(w) \di s\di r.
		\end{equation*}
		By splitting the domain $(0,+\infty)\times(0,+\infty)$ into $A_1,A_2$ and $A_3$, with $A_1:=\set{(s,r)|\ \dist(x,y)+s\le r}$, $A_2:=\set{(s,r)|\ \dist(x,y)+r\le s}$ and $A_3:=\set{(s,r)|\ \dist(x,y)>|r-s|}$ we are left with the estimates of the following quantities:
		\begin{equation*}
		I_1:=\int_{A_1}\int_X f(w) \frac{\id_{B(x,r)}(w)}{\meas (B(x,r))}\frac{\id_{B(y,s)}(w)}{\meas (B(y,s))} \di \meas(w) \di s\di r,
		\end{equation*}	
		\begin{equation*}
		I_2:=\int_{A_2}\int_X f(w) \frac{\id_{B(x,r)}(w)}{\meas  (B(x,r))}\frac{\id_{B(y,s)}(w)}{\meas (B(y,s))} \di \meas(w) \di s\di r
		\end{equation*}
		and
		\begin{equation*}
		I_3:=\int_{A_3}\int_X f(w) \frac{\id_{B(x,r)}(w)}{\meas (B(x,r))}\frac{\id_{B(y,s)}(w)}{\meas (B(y,s))} \di \meas(w) \di s\di r.
		\end{equation*}
		In order to estimate $I_1$, we observe that $B(y,s)\subset B(x,r)$ for every $(s,r)\in A_1$, thus
		\begin{align*}
		I_1= & \int_{A_1}\frac{1}{\meas(B(x,r))}\dashint_{B(y,s)} f(w) \di \meas(w) \di s\di r\\
		\leq & Mf(y) \int_{\dist(x,y)}^{\infty} \int_0^{r-\dist(x,y)} \frac{1}{\meas(B(x,r))} \di s \di r\\
		\leq& Mf(y)\int_{\dist(x,y)}^{\infty}\frac{r}{\meas(B(x,r))}\di r\\
		\leq & C_2G(x,y)Mf(y).
		\end{align*}
		By symmetry we get
		\begin{equation*}
		I_2\le C_2G(x,y)Mf(x).
		\end{equation*}
		To estimate $I_3$ let us observe that, if $r+s<\dist(x,y)$, then $B(x,r)\cap B(y,s)=\emptyset$. Thus the integration can be restricted to the smaller domain $B:=\set{(s,r)|\ \dist(x,y)> |r-s|,\ r+s\ge \dist(x,y)}$ that we split once more into $B_1:=\set{(s,r)|\ \dist(x,y)> r-s,\ r+s\ge \dist(x,y),\ r\geq s}$ and $B_2:=\set{(s,r)|\ \dist(x,y)> s-r,\ r+s\ge \dist(x,y),\ r<s}$. Therefore we have
		\begin{align*}
		I_3 = &\int_B\int_X f(w) \frac{\id_{B(x,r)}(w)}{\meas (B(x,r))}\frac{\id_{B(y,s)}(w)}{\meas (B(y,s))} \di \meas(w) \di s\di r\\
		= &\int_{B_1}\int_X f(w) \frac{\id_{B(x,r)}(w)}{\meas (B(x,r))}\frac{\id_{B(y,s)}(w)}{\meas (B(y,s))} \di \meas(w) \di s\di r\\
		&+  \int_{B_2}\int_X f(w) \frac{\id_{B(x,r)}(w)}{\meas (B(x,r))}\frac{\id_{B(y,s)}(w)}{\meas (B(y,s))} \di \meas(w) \di s\di r\\
		=: & I_3^1+I_3^2.
		\end{align*}
		We now deal with $I_3^1$. Using the rough estimate $\id_{B(x,r)}\le 1$ we obtain
		\begin{align*}
		I_3^1\leq & \int_{B_1}\frac{1}{\meas(B(x,r))}\dashint_{B(y,s)} f(w)\di \meas(w) \di s\di r\\
		\leq & Mf(y)\int_{\dist(x,y)/2}^{\infty} \int_{\abs{\dist(x,y)-r}}^r\frac{1}{\meas(B(x,r))} \di s \di r\\
		\leq & Mf(y)\int_{\dist(x,y)/2}^{\infty} \frac{r}{\meas(B(x,r))} \di r\\
		= & Mf(y) \frac{1}{4}\int_{\dist(x,y)}^{\infty} \frac{r}{\meas(B(x,r/2))} \di r.
		\end{align*}
		With a simple application of \eqref{eq:BishopGromov} and \eqref{eq:estimateG} we conclude that
		$I_3^1\leq C(C_2,N) Mf(y) G(x,y)$. By symmetry we also have $I_3^2\leq C(C_2,N) Mf(x) G(x,y)$. Putting all these estimates together we obtain the desired result.
	\end{proof}

	\begin{remark}\label{remark:variants of maximal estimate}
		It is clear from the proof of \autoref{prop:maximalGestimate} and from \autoref{remark:Geps estimate} that the same estimate holds true if one puts $\nabla G_x^{\eps}$ and $\nabla G_y^{\epsilon}$ in place of $\nabla G_x$ and $\nabla G_y$ at the left hand side of \eqref{eq:auxiliaryestimate}. More precisely it holds that
		\begin{equation}\label{eq:scalarregularizedemaximal}
		\int f(z)\abs{\nabla G^{\eps}_x(z)}\abs{\nabla G^{\eps}_y(z)}\di\meas(z)\le C_M G(x,y)\left(Mf(x)+Mf(y)\right),
		\end{equation}
		for every $x,y\in X$.
	\end{remark}
	
	\subsubsection{The Green quasi-metric}\label{subsec:Gquasimetric}
	This section is devoted to the study of the following function
	\begin{equation}\label{eq:Green distance}
	\dist_G(x,y):=
	\begin{cases*}
	\frac{1}{G(x,y)}& if $x\neq y$,\\
	0  & otherwise.
	\end{cases*}	
	\end{equation}
	Our aim is to prove that $\dist_G$ is a \textit{quasi-metric} on $X$ (i.e. it satisfies an approximated triangle inequality, see \autoref{prop: almost triangle inequality} below) and that $\meas$ is still a doubling measure over $(X,\dist_G)$ (see \autoref{prop: doubling property} below for a precise statement). The terminology, quite common in the literature about analysis on metric spaces, is borrowed from \cite[Chapter 14]{Heinonen01}. In order to do so we will need to impose an assumption slightly stronger than assumption \autoref{assumption: good definition of G} to the m.m.s. $(X,\dist,\meas)$.
	
	\begin{assumption}\label{ass:doublingG}
		There exists an $\RCD(0,N-3)$ metric measure space $(\bar{X},\bar{\dist},\bar{\meas})$ such that $(X,\dist,\meas)$ is the tensor product between $(\bar{X},\bar{\dist},\bar{\meas})$ and $(\setR^3,\dist_{\setR^3},\leb^3)$.
	\end{assumption}
	
	First of all observe that $\dist_G$ is symmetric and positive whenever $x\neq y$. Moreover, for every $x\in X$, the map $y\mapsto \dist_G(x,y)$ is continuous. Indeed, thanks to the continuity of $G_x$ in $X\setminus \set{x}$ (see \autoref{rm:continuityfarfrompole} above), we need only to show that $\dist_{G}(x, \cdot)$ is continuous at $x$, and this is the content of the following lemma.
	
	\begin{lemma}\label{lemma:continuitycontinuity}
		Let $(X,\dist,\meas)$ be an $\RCD(0,N)$ m.m.s. satisfying assumption \autoref{ass:doublingG}.
		Then for any $x\in X$ it holds that $\dist_G(x,y)\to 0$ if and only if $\dist(x,y)\to 0$.
	\end{lemma}
	
	\begin{proof}
		Suppose that $\dist_G(x,y)\to 0$. Then, by the very definition of $\dist_G$, it must be $G(x,y)\to +\infty$. Hence, since we have the uniform control $G(x,y)\le C_2 F(x,\dist(x,y))$ and $F(x,\cdot)$ is bounded away from $0$, we conclude $\dist(x,y)\to 0$.
		
		In order to prove the converse we observe that, if $\dist(x,y)\to 0$, then $F(x,\dist(x,y))\to \infty$. Indeed, under our assumptions, $s\mapsto s/\meas(B(x,s))$ is not integrable at $0$ and to conclude we just need to exploit the bound $G(x,y)\ge 1/C_2F(x,\dist(x,y))$ (see \autoref{prop:estimateG} above).
	\end{proof} 
	
	Let us now state and prove the almost triangle inequality property of $\dist_G$.
	
	\begin{proposition}[Almost triangle inequality]\label{prop: almost triangle inequality} 
		Let $(X, \dist,\meas)$ be an $\RCD(0,N)$ space satisfying assumption \autoref{assumption: good definition of G}. Then there exists a constant $C_T\ge 1$, depending only on $N$, such that
		\begin{equation}\label{eq:almost triangle inequality}
		\dist_G(x,y)\leq C_T (\dist_G(x,z)+\dist_G(z,y)) \quad \text{for any $x,y$ and $z\in X$.}	
		\end{equation}	
	\end{proposition}
	
	The core of the proof of \autoref{prop: almost triangle inequality} is contained in the following elementary lemma.
	
	\begin{lemma}\label{lemma:firststeptriangle} Let $(X,\dist, \meas)$ be a doubling 
		metric measure space satisfying assumption \autoref{assumption: good definition of G}. Then there exists a constant $C$, depending only on the doubling constant of $(X,\dist,\meas)$, such that
		\begin{equation*}
		F_x(\dist(x,z))F_y(\dist(y,z))\le C(F_x(\dist(x,y))F_y(\dist(y,z))+F_y(\dist(x,y))F_x(\dist(x,z))),	
		\end{equation*}
		for any $x,y,z$ in $X$.
	\end{lemma}
	\begin{proof}
		Let us fix $x,y,z$ in $X$. We can assume without loss of generality that they are all different.
		Starting from the identity
		\begin{equation*}
		F_x(\dist(x,z))F_y(\dist(y,z))=\int_{\dist(x,z)}^\infty\int_{\dist(y,z)}^\infty \frac{t}{\meas(B(x,t))}\frac{s}{\meas(B(y,s))} \di t\di s,
		\end{equation*}
		and exploiting the inclusion 
		\begin{equation*}
		\set{\dist(x,z)<t,\ \dist(y,z)<s}\subset \set{\dist(x,y)<2t,\ \dist(y,z)<s}\cup \set{\dist(x,y)<2s,\ \dist(x,z)<t},	
		\end{equation*}
		we get
		\begin{align*}
		F_x(\dist(x,z))F_y(\dist(y,z))\leq & \int_{\dist(x,y)/2}^\infty\int_{\dist(y,z)}^\infty \frac{t}{\meas(B(x,t))}\frac{s}{\meas(B(y,s))} \di t\di s\\
		&+\int_{\dist(x,z)}^\infty\int_{\dist(x,y)/2}^\infty \frac{t}{\meas(B(x,t))}\frac{s}{\meas(B(y,s))} \di t\di s\\
		=& F_x(\dist(x,y)/2)F_y(\dist(y,z))+F_y(\dist(x,y)/2)F_x(\dist(x,z)).
		\end{align*}
		To conclude, observe that $F_w(r/2)\le C F_w(r)$ for any $r>0$ and $w\in X$, where $C$ depends only on the doubling constant of $\meas$.
	\end{proof}

	\begin{proof}[Proof of \autoref{prop: almost triangle inequality}]
		The desired conclusion
		\eqref{eq:almost triangle inequality} is equivalent to
		\begin{equation*}
		G_x(z)G_y(z)\leq C_TG(x,y)(G_x(z)+G_y(z)),
		\end{equation*}
		that follows from \autoref{lemma:firststeptriangle} taking into account \autoref{prop:estimateG}.
	\end{proof}
	
	We introduce the notation 
	\begin{equation}\label{eq: Gballs}
	B^G(x,r):=\set{y\in X | \dist_G(x,y)<r}
	\end{equation}
	to denote the balls with respect to the quasi-metric $\dist_{G}$.
	The next result of this short section is about the doubling property of the measure $\meas$ in the quasi-metric space $(X, \dist_G)$.
	

	\begin{lemma}[Reverse Bishop-Gromov inequality]\label{lemma: reverse doubling}	
		Let $(\bar{X}, \bar{\dist},\mu)$ be a doubling m.m.s. with doubling constant $C_{\mu}$ and denote by $(X,\dist,\meas)$ the tensor product between $(\bar{X},\bar{\dist},\mu)$ and $(\setR^k,\dist_{\setR^k},\leb^k)$ for some $k\ge 1$. Then
		\begin{equation}\label{eq:reverse BishopGromov}
		\frac{\meas(B(x,R))}{\meas(B(x,r))}\geq \frac{1}{C_{\mu}\sqrt{2}^k}\left(\frac{R}{r}\right)^k,
		\end{equation}
		for every $0<r<R$ and for any $x\in X$.
	\end{lemma}
	\begin{proof}
		The following chain of inclusions holds true for any $x\in X$, any $v\in\setR^k$ and any $r>0$:
		\begin{equation*}
		B(x,r/\sqrt{2})\times B(v,r/\sqrt{2})\subset	 B((x,v),r)\subset B(x,r)\times B(v,r).
		\end{equation*}
		It follows that
		\begin{align*}
		\frac{\meas(B(x,v),R)}{\meas(B(x,v),r)} \ge & \frac{\meas( B(x,R/\sqrt{2})\times B(v,R/\sqrt{2}))}{\meas(B(x,r)\times B(v,r))}
		=\frac{1}{\sqrt{2}^k} \frac{\mu(B(x,R/\sqrt{2}))}{\mu(B(x,r))}\left(\frac{R}{r} \right)^k\\
		\geq &  \frac{1}{\sqrt{2}^k} \frac{\mu(B(x,r/2))}{\mu(B(x,r))}\left(\frac{R}{r} \right)^k
		\geq\frac{1}{C_{\mu}\sqrt{2}^k}\left(\frac{R}{r} \right)^k.
		\end{align*}
	\end{proof}
	
	We finally state and prove the doubling property of $\meas$ with respect to the new quasi-metric $\dist_G$.
	
	\begin{proposition}[Doubling property]\label{prop: doubling property}
		Assume that $(X,\dist, \meas)$ is an $\RCD(0,N)$ m.m.s. satisfying assumption \autoref{ass:doublingG}.
		Then there exists a constant $C^G>0$, depending only on $N$, such that
		\begin{equation}\label{eq:Gdoubling}
		\meas (B^G(x,2r))\leq C^G \meas (B^G(x,r)),
		\end{equation}
		for every $x\in X$ and every $r>0$.
	\end{proposition}
	
	\begin{proof}
		We begin by observing that, under assumption \autoref{ass:doublingG}, an application of \autoref{lemma: reverse doubling} yields the existence of constants $a>1$ and $b<1$ such that $F_x(aR)\le bF_x(R)$ for any $x\in X$ and for any $R>0$.\
		It follows that
			\begin{equation*}
			F_x(R)=F_x(aR)+\int_R^{aR}\frac{s}{\meas(B(x,s))}\di s\leq bF_x(R)+a\frac{R^2}{\meas(B(x,R))},
			\end{equation*}
			for any $x\in X$ and $R>0$, thus we get
		\begin{equation}\label{eq:boundabove}
		F_x(R)\le \frac{a}{1-b}\frac{R^2}{\meas(B(x,R))}\quad\text{for any $x\in X$ and for any $R>0$.}
		\end{equation}
		The inequality in \eqref{eq:boundabove} yields the existence of $\alpha>0$ such that $r\mapsto r^{\alpha}F_x(r)$ is nonincreasing on $(0,+\infty)$ (just take $\alpha:=(1-b)/a$ and differentiate w.r.t. $r$). Hence we can find $\gamma<1$ such that
		\begin{equation}\label{eq:intermboundabove}
		F_x(2R)\le \gamma F_x(R)\quad\text{for any $x\in X$ and for any $R>0$.}
		\end{equation}
		Inequality \eqref{eq:intermboundabove} implies in turn that 
		\begin{equation}\label{a3}
		F^{-1}_x(\gamma R)\leq 2 F^{-1}_x(R)\quad \text{for any $x\in X$ and for any $R>0$.}
		\end{equation}
		
		The last ingredient we need are the estimates, valid for any $\lambda>0$ and for any $x\in X$: 
		\begin{equation}\label{a4}
		\meas(\set{G_x>\lambda})\leq \meas(\set{ F_x(\dist_x)>\lambda/C_2}),\quad
		\meas(\set{F_x(\dist_x)>\lambda})\leq \meas(\set{G_x>\lambda/C_2}),
		\end{equation}
		where $C_2$ is the constant appearing in \eqref{eq:estimateG}.\newline
		In order to conclude let us show that
		\begin{equation}\label{a5}
		\meas\left(B^G\left(x, \frac{1}{\gamma^MC_2}r\right)\right)\leq 2^{M\cdot N} \meas(B^G(x,C_2r)),
		\end{equation}
		for every $r>0$, $x\in X$ and $M\in \setN$.\newline
		Using \eqref{a4}, the definition of $B^G$ and the fact that $F^{-1}_x$ is non-increasing we find
		\begin{equation*}
		\meas\left(B^G\left(x,\frac{1}{\gamma^MC_2}r \right)\right) =\meas \left(\left\lbrace G_x>\frac{C_2\gamma^M}{r}\right\rbrace\right)
		\leq  \meas\left(\left\lbrace F_x(\dist_x)>\frac{\gamma^M}{r}\right\rbrace\right)
		= \meas \left(B\left(x, F^{-1}_x\left(\gamma^M/r\right)\right)\right).
		\end{equation*}
		Applying first $M$-times \eqref{a3} and then the doubling inequality \eqref{eq:BishopGromov} in the $\RCD(0,N)$ m.m.s. $(X,\dist,\meas)$, we get
		\begin{align*}
		\meas\left(B^G\left(x,\frac{1}{C_2\gamma^M}r \right)\right) &\leq \meas (B(x, 2^M F^{-1}_x(1/r)))
		\leq 2^{M\cdot N} \meas (B(x, F^{-1}_x(1/r)))\\
		&= 2^{M\cdot N} \meas(\set{F_x(\dist_x)>1/r}).
		\end{align*}
		Using again \eqref{a4}, we obtain:
		\begin{equation}\label{eq:ineqb3}
		\meas\left(B^G\left(x,\frac{1}{C_2\gamma^M}r \right)\right)\leq 2^{M\cdot N}\meas(\set{G_x>\frac{1}{C_2 r}})=2^{M\cdot N}\meas(B^G(x,C_2r)).
		\end{equation}
		Setting $s:=C_2r$ in \eqref{eq:ineqb3}, we have that
		\begin{equation*}
		\meas\left(B^G\left(x, \frac{1}{C_2^2\gamma^M}s\right)\right)\leq 2^{(M+1)N} \meas(B^G(x,s)),
		\end{equation*}
		for any $s>0$, any $x\in X$ and every real number $M>1$. Choosing $M=\log_{\gamma^{-1}}(2C_2^2)$ we eventually obtain \eqref{eq:Gdoubling}.
	\end{proof}
		
		The last lemma of this subsection deals with the integrability properties of the maximal function associated to the quasi metric $\dist_G$.
		
		\begin{lemma}\label{lemma: maximal function associated to dG}
			Assume that $(X,\dist,\meas)$ is an $\RCD(0,N)$ m.m.s. satisfying assumption \autoref{ass:doublingG}. Then there exists a constant $C>0$, depending only on $N$, such that for any $f\in L^1_{\loc}(X,\meas)$ it holds
			\begin{equation}\label{z}
			M^Gf(x)\le C Mf(x),\qquad \forall x\in X,
			\end{equation}
			where $M^Gf(x):=\sup_{r>0}\dashint_{B^G(x,r)} |f| \di \meas$ and $Mf$ is the Hardy Littlewood maximal function associated to $f$.
			
			In particular $M^G$ is a bounded operator from $L^2(X,\meas)$ into itself.	
		\end{lemma}
		\begin{proof}
			Exploiting the inclusions (see \eqref{eq:estimateG})
			\begin{equation*}
			\set{F_x(\dist_x)>C_2r}\subset B^G(x,r)\subset \set{F_x(\dist_x)>r/C_2},
			\end{equation*}
			we get
			\begin{equation*}
			\dashint_{B^{G}(x,r)} |f| \di \meas \leq
			\frac{\meas(B(x,F_x^{-1}(r/C_2)))}{\meas(B(x,F_x^{-1}(C_2r)))} Mf(x).
			\end{equation*}
			Using the Bishop-Gromov inequality \eqref{eq:BG} and \eqref{eq:intermboundabove} we get
			\begin{equation*}
			\frac{\meas(B(x,F_x^{-1}(r/C_2)))}{\meas(B(x,F_x^{-1}(C_2r)))}
			\le \left( \frac{F_x^{-1}(r/C_2)}{F_x^{-1}(C_2r)}\right)^N
			\le C(\gamma,C_2,N).
			\end{equation*}
			Recalling that $C_2$ and $\gamma$ are constants depending only on $N$, \eqref{z} follows.
			
			Finally observe that the maximal operator $M$ maps $L^{2}(X,\meas)$ into itself since $(X,\dist,\meas)$ is a doubling m.m.s., this property is inherited by $M^G$ thanks to \eqref{z}.
		\end{proof}
		
	\subsection{A Lusin-type regularity result}\label{subsec:Lusinreg}
	This section is dedicated to the study of the regularity of a flow $\XX_t$ associated to a Sobolev time dependent vector field $b$. The regularity will be understood with respect to the newly introduced quasi-metric $\dist_G$.

	Let us begin with a crucial maximal estimate for vector fields. As we pointed out in the introduction of the paper it can be seen as an infinitesimal version of \autoref{thm: sobolev vectorfield, G-regularity}.
	
	\begin{proposition}[Maximal estimate, vector-valued version]\label{prop: G-maximal estimate}
		Let $(X,\dist,\meas)$ be an $\RCD(0,N)$ m.m.s. satisfying assumption \autoref{assumption: good definition of G}.
		Assume that $b\in W^{1,2}_{C,s}(TX)$ is a compactly supported and bounded vector field. Then, setting $g:=|\nabla_{\sym} b|+|\div b|$, it holds
		\begin{equation}\label{eq:maxvectorvalued}
		\abs{ b\cdot\nabla G_x(y)+ b\cdot\nabla G_y(x)}\le 2C_MG(x,y)(Mg(x)+Mg(y)),
		\end{equation}
		for $\meas\times \meas$-a.e. $(x,y)\in X\times X$, where $M$ stands for the maximal operator.
	\end{proposition}
	
	\begin{proof}
		The heuristic standing behind the proof of this result, which was already present in \cite{BrueSemola18}, is the following one: assuming that $b$ is divergence free we can formally compute
		\begin{align*}
		b\cdot\nabla G_x(y)+ b\cdot\nabla G_y(x)=&-\int_{X}b\cdot\nabla G_x(w)\di\Delta G_y(w)-\int_X b\cdot\nabla G_y(w)\di\Delta G_x(w)\\
		=& 2\int_X\nabla_{\sym}b(w)(\nabla G_x(w),\nabla G_y(w))\di\meas(w),
		\end{align*}
		so that, taking the moduli and applying \autoref{prop:maximalGestimate}, we would reach the desired conclusion.
		
		This being said the proof of this result will be divided into two steps: in the first one we are going to prove an estimate for the regularized functions $G^{\epsilon}$; in the second one the sought conclusion will be recovered by an approximation procedure.
		
		\textbf{Step 1}
		We start proving that, for every $\eps\in (0,1)$ and for every $x,y\in X$, it holds
		\begin{equation}\label{eq: intermediateMaximalVectorfileds}
		\abs{\int_X \left\lbrace b\cdot\nabla G^{\eps}_x \Delta G^{\eps}_y+b\cdot\nabla G^{\eps}_y \Delta G^{\eps}_x\right\rbrace  \di \meas}
		\le 2C_M G(x,y)\left(M g(x)+Mg(y)\right).
		\end{equation}  
		To this aim, we choose a cut-off function with compact support $\eta\in \Test(X,\dist,\meas)$ such that $\eta\equiv 1$ on $\supp b$ (the existence of such function follows from \autoref{lemma:cutofffunctions}). Applying \eqref{eq:Sobvectfield} with $h=\eta$, $f=\eta G^{\epsilon}_x$ and $g=\eta G^{\epsilon}_y$ (observe that they are admissible test function in the definition of symmetric covariant derivative thanks to \autoref{rm::testlocregularity}) we obtain:  
		\begin{align*}
		\Big|\int_X   \left\lbrace b\cdot\nabla G^{\eps}_x \Delta G^{\eps}_y+b\cdot\right. & \left.\nabla G^{\eps}_y \Delta G^{\eps}_x\right\rbrace   \di \meas\Big|\\
		\leq &
		\abs{\int_X \left\lbrace  b\cdot\nabla G^{\eps}_x \Delta G^{\eps}_y+b\cdot\nabla G^{\eps}_y \Delta G^{\eps}_x-\div b\ \nabla G^{\eps}_x\cdot \nabla G^{\eps}_y \right\rbrace \di \meas}\\
		&+\abs{\int_X \div b\ \nabla G^{\eps}_x\cdot \nabla G^{\eps}_y \di \meas}\\
		=& 2\abs{\int_X  \nabla_{\sym}b\ (\nabla G^{\eps}_x, \nabla G^{\eps}_y) \di \meas}+\abs{\int_X \div b\ \nabla G^{\eps}_x\cdot \nabla G^{\eps}_y \di \meas}\\
		\leq & 2\int_X g(w)|\nabla G^{\eps}_x(w)||\nabla G^{\eps}_y(w)| \di \meas(w).
		\end{align*}	
		The estimate in \eqref{eq: intermediateMaximalVectorfileds} follows from the inequality we just obtained applying \eqref{eq:scalarregularizedemaximal}.
		
		\textbf{Step 2}
		The second step of the proof aims into proving that, as $\eps \to 0$, it holds
		\begin{equation}\label{z4}
		\abs{\int_X \left\lbrace b\cdot\nabla G^{\eps}_x \Delta G^{\eps}_y+b\cdot \nabla G^{\eps}_y \Delta G^{\eps}_x\right\rbrace  \di \meas}\to \abs{b\cdot \nabla G_x(y)+b\cdot \nabla G_y(x)}
		\end{equation}
		in $L^1_{\loc}(X\times X,\meas\times\meas)$. 
		This will allow us to get \eqref{eq:maxvectorvalued} by choosing a sequence $\epsilon_i\downarrow 0$ such that the convergence in \eqref{z4} holds true $\meas\times\meas$-a.e. on $X\times X$ and exploiting what we proved in the first step. 
		
		In order to prove \eqref{z4}, we start recalling that $\Delta G^{\eps}_y(w)=-p_{\eps}(y,w)$ for any $\epsilon >0$ (see \autoref{lemma:G eps regularity}). Thus
		\begin{equation}\label{z15}
		\int_X b\cdot\nabla G^{\eps}_x\Delta G^{\eps}_y \di \meas = -P_{\eps}(b\cdot \nabla G^{\eps}_x)(y)\quad \text{for any $x,y\in X$.}
		\end{equation}
		Moreover for our purposes it suffices to check that $\int_K\int_K |P_{\eps}(b\cdot \nabla G^{\eps}_x)(y)-b\cdot\nabla G_x(y)| \di\meas(x)\di \meas(y)\to 0$ as $\epsilon\to 0$, for every compact $K\subset X$. Adding and subtracting $P_{\eps}(b\cdot \nabla G_x)(y)$ (that is well defined since $b\cdot \nabla G_x\in L^1(X,\meas)$), we obtain
		\begin{align*}
		\int_K\int_K |P_{\eps}&(b\cdot \nabla G^{\eps}_x)(y)-b\cdot\nabla  G_x(y)| \di\meas(x)\di \meas(y)\\ 
		\leq & \int_K \norm{P_{\eps}(b\cdot \nabla (G^{\eps}_x-G_x))}_{L^1(X,\meas)}\di \meas(x)
		+ \int_K \norm{P_{\eps}(b\cdot \nabla G_x)-b\cdot\nabla G_x}_{L^1(X,\meas)}\di\meas(x).
		\end{align*}
		Using the $L^1$-norm contractivity property of the semigroup $P_{\eps}$, we deduce that
		\begin{equation*}
		\norm{P_{\eps}(b\cdot \nabla (G^{\eps}_x-G_x))}_{L^1(X,\meas)}\le \norm{b\cdot \nabla (G^{\eps}_x-G_x)}_{L^1(X,\meas)}\quad \text{for any $x\in X$.}
		\end{equation*}
		Hence, for any $x\in X$,
		\begin{equation*}
		\norm{P_{\eps}(b\cdot \nabla (G^{\eps}_x-G_x))}_{L^1(X,\meas)}\to 0\quad\text{as $\epsilon\downarrow 0$,}
		\end{equation*}
		since $G^{\eps}_x\to G^x$ in $W^{1,1}_{\loc}(X,\dist,\meas)$ by \autoref{prop:estimateG} and $b$ has compact support by assumption. Also the term $\norm{P_{\eps}(b\cdot \nabla G_x)-b\cdot\nabla G_x}_{L^1(X,\meas)}$ goes to zero for every $x\in X$ since, as just remarked, $b\cdot \nabla G_x\in L^1(X,\meas)$. Moreover both these terms are uniformly bounded by the function $x\mapsto C\norm{b}_{L^{\infty}} \norm{H_x(\dist_x(\cdot))}_{L^1(\supp(b),\meas)}$ that is locally integrable, since the map $(x,y)\mapsto H(x,\dist(x,y))$ belongs to $L^1_{\loc}(X\times X,\meas\times\meas)$ in view of \autoref{lemma: integrabilityFxHx}. The conclusion of \eqref{z4} can now be recovered applying the dominated convergence theorem.
	\end{proof}

	Up to the end of this section we let $(X,\dist,\meas)$ be an $\RCD(0,N)$ m.m.s. satisfying assumption \autoref{ass:doublingG}.\newline
	Let us fix $T>0$ and let $b_t\in L^{\infty}((0,T)\times X)$ be a time dependent vector field with compact support, uniformly w.r.t. time. We further assume that $b_t\in W^{1,2}_{C,s}(TX)$ for a.e. $t\in(0,T)$ and that $|\nabla_{\sym} b_t|\in L^{1}((0,T);L^2(X,\meas))$ and $\div b_t\in L^1((0,T); L^{\infty}(X,\meas))$.\newline
	Under these assumptions, the theory developed in \cite{AmbrosioTrevisan14} grants existence and uniqueness of the Regular Lagrangian Flow $(\XX_t)_{t\in[0,T]}$ of $b$. We shall denote by $L\ge0$ its compressibility constant.
	
	Our aim is to implement a strategy very similar to the one adopted in \cite{CrippaDeLellis08} (in the Euclidean setting) and in our previous work \cite{BrueSemola18} (in the context of Ahlfors regular, compact $\RCD(K,N)$ m.m. spaces), in order to prove a Lusin-type regularity result for RLFs in terms of the newly defined quasi-metric $\dist_G$ (see \autoref{thm: sobolev vectorfield, G-regularity} below for a precise statement).
	
	Let us spend some words to explain the very simple idea behind the just mentioned strategy. Having in mind the standard Gronwall argument explained in the introduction (see \eqref{computation1}), it is natural to try to estimate the time derivative of $\dist_{G}(\XX_t(x), \XX_t(y))$. In order to do so, we use \autoref{cor: coupled derivative with G} and \autoref{prop: G-maximal estimate} obtaining
	\begin{equation}\label{z21}
	\abs{ \frac{\di}{\di t}  G(\XX_t(x),\XX_t(y))} \leq 2C_MG(\XX_t(x),\XX_t(y))\left\lbrace Mg(
	\XX_t(x))+Mg(\XX_t(y))\right\rbrace ,
	\end{equation}
	for $\leb^1$-a.e. $t\in (0,T)$ and
	for $\meas\times \meas$-a.e. $(x,y)\in X\times X$.
	
	Integrating with respect to the time variable and recalling that $\dist_G:=1/G$, we get, for any $t\in [0,T]$,
	\begin{equation}\label{z16}
	\dist_{G}(\XX_t(x),\XX_t(y))\leq \dist_{G}(x,y) \exp\left\lbrace \int_0^T Mg(\XX_s(x))\di s+\int_0^T Mg(\XX_s(y))\di s\right\rbrace
	\end{equation}
	for $\meas\times \meas$-a.e. $(x,y)\in X\times X$.
	Note that the function $g^*(x):=\int_0^T Mg(\XX_s(x))\di s$ belongs to $L^2$ with
	\begin{equation}\label{z19}
	\norm{g^*}_{L^2}\leq CL\int_0^T \norm{|\nabla b_s|+|\div b_s|}_{L^2} \di s,
	\end{equation}
	where $C$ is a universal constant and $L$ is as in \autoref{def:Regularlagrangianflow}.
	Putting \eqref{z16} and \eqref{z19} together, we get a \textit{weak} version of the sought Lusin-Lipschitz estimate, since the inequality in \eqref{z16} is not point-wise but it holds only $\meas\times \meas$-almost everywhere. In order to fix this issue we adopt a slightly different approach (borrowed from \cite{CrippaDeLellis08}). We consider the family of functionals
	\begin{equation}\label{CDL functional}
	\Phi_{t,r}(x):=\dashint_{B^G(x,r)}\log \left( 1+\frac{\dist_G(\XX_t(x),\XX_t(y))}{r} \right) \di \meas(y),
	\end{equation}
	for $r\in (0,\infty)$ and $t\in [0,T]$ and we bound its time derivative performing the same estimate in \eqref{z21}.
	
	This gives an $L^2$ bound on the function
	\begin{equation}\label{def: Phi star}
	\Phi^{*}(x):=\sup_{0\leq t\leq T}\sup_{0<r< \infty}\Phi_{t,r}(x),
	\end{equation}
	that will play the role of $g^*$ in \eqref{z16}. 
	Let us remark that in order to perform such a plan we need to use the doubling and quasi-metric property of $\dist_{G}$ (see \autoref{subsec:Gquasimetric}).
	
	\begin{proposition}\label{th 1}
		Let $(X,\dist,\meas)$ be an $\RCD(0,N)$ m.m.s. satisfying assumption \autoref{ass:doublingG}.
		Let moreover $b$, $(\XX_t)_{t\in [0,T]}$ and $L$ be as in the discussion above.
		Then for any compact $P\subset X$ there exists a constant $C=C(T,\meas(P),N)$ such that
		\begin{equation}\label{eq:boundnromreg}
		\norm{\Phi^*}_{L^2(P,\meas)}\leq C\left[1+  L \int_0^T \norm{|\nabla_{\sym}b_s|+|\div   b_s|}_{L^2(X,\meas)} \di s\right].
		\end{equation}
	\end{proposition}
	\begin{proof}
		By \autoref{cor: coupled derivative with G} we can say that, for $\meas\times \meas$-a.e. $(x,y)\in X\times X$, the map $t\mapsto G(\XX_t(x),\XX_t(y))$ belongs to $W^{1,1}((0,T))$ and its derivative is given by
		\begin{equation}\label{eq: dettaglio infernale}
		\frac{\di}{\di t}  G(\XX_t(x),\XX_t(y))= b_t\cdot \nabla G_{\XX_t(x)}(\XX_t(y))+b_t\cdot \nabla G_{\XX_t(y)}(\XX_t(x))\quad\text{for $\leb^1$-a.e. $t\in (0,T)$.}
		\end{equation}
		It follows that, for $\meas$-a.e. $x\in X$, the map $t\mapsto \Phi_{t,r}$ belongs to $W^{1,1}((0,T))$ as well (and actually it is absolutely continuous, since it is continuous) and it holds
		\begin{align*}
		\Phi_{t,r}(x)=&\Phi_{0,r}(x)+ \int_0^t \frac{\di}{\di s}\Phi_{s,r}(x) \di s\\
		\leq &\Phi_{0,r}(x)+\int_0^t\dashint_{B^G(x,r)} \frac{|\frac{\di}{\di s}G(\XX_s(x),\XX_s(y))|}{G(\XX_s(x),\XX_s(y))}\cdot
		\frac{1}{G(\XX_s(x),\XX_s(y))r+1} \di \meas(y)\di s\\
		\leq & \Phi_{0,r}(x)+\int_0^t\dashint_{B^G(x,r)} \frac{| b_s\cdot\nabla G_{\XX_s(x)}(\XX_s(y))+b_s\cdot\nabla G_{\XX_s(y)}(\XX_s(x))|}{G(\XX_s(x),\XX_s(y))} \di \meas(y)\di s.
		\end{align*}
		
		Applying \autoref{prop: G-maximal estimate} and observing that $\Phi_{0,r}\leq \log2$, we conclude that
		\begin{equation*}
		\Phi_{t,r}(x)\leq \log2+\int_0^t \dashint_{B^G(x,r)} \left\lbrace Mg_s(\XX_s(x))+Mg_s(\XX_s(y)) \right\rbrace \di \meas(y) \di s \quad\text{for $\meas$-a.e. $x\in X$},
		\end{equation*}
		where $g_s:=|\nabla_{\sym} b_s|+|\div b_s|$. 
		We can finally estimate $\Phi^*$ obtaining that, for $\meas$-a.e. $x\in X$, it holds
		\begin{align*}
		\Phi^{*}(x)  \leq & \log2+
		\sup_{0\leq t\leq T}\sup_{0<r< \infty}\int_0^t \dashint_{B^G(x,r)} \left\lbrace Mg_s(\XX_s(x))+Mg_s(\XX_s(y)) \right\rbrace \di \meas(y) \di s\\
		\leq & \log2 + \int_0^T Mg_s(\XX_s(x)) \di s+\int_0^T M^GMg_s(\XX_s(\cdot))(x) \di s,
		\end{align*}
		where $M^G$ is the maximal operator associated to the quasi-metric $\dist_G$ (while $M$ still denotes the maximal operator associated to the m.m.s. $(X,\dist,\meas)$).\newline 
		We remark that $M^G$ maps $L^2(X,\meas)$ into itself (see \autoref{lemma: maximal function associated to dG}).	
		
		Passing to the $L^2(\meas)$-norms over $P$ and taking into account the assumption that the RLF has compressibility constant $L<\infty$ we obtain \eqref{eq:boundnromreg}.	
	\end{proof}
	
	Below we state and prove our main regularity result for Regular Lagrangian flows.
	
	\begin{theorem}\label{thm: sobolev vectorfield, G-regularity}
		Let $(X,\dist,\meas)$, $b$, $(\XX_t)_{t\in [0,T]}$ and $L$ be as in the assumptions of \autoref{th 1} above. Then there exists $C=C(N)$ such that, for any $x,y\in X$ and for any $t\in[0,T]$, it holds
		\begin{equation}\label{lip regularity of X 1}
		\dist_G(\XX_t(x),\XX_t(y))\le C e^{C\left(\Phi^*(x)+\Phi^*(y)\right)}\dist_G(x,y),
		\end{equation}
		where $\Phi^*$ was defined in \eqref{def: Phi star}.
		
		Moreover, for any compact $P\subset X$, the following property holds:
		for every $\epsilon>0$ there exists a Borel set $E\subset P$ such that
		$\meas(P\setminus E)<\eps$ and 
		\begin{equation}\label{eq:Lipregularity}
		\dist_G(\XX_t(x),\XX_t(y))\leq C\exp\left(2C\frac{\norm{\Phi^*}_{L^2(P,\meas)}}{\sqrt{\epsilon}}\right) \dist_G(x,y)
		\quad \text{for any $x,y\in E$,}
		\end{equation}
		for every $t\in [0,T]$. We remark that this last statement is meaningful since, under our regularity assumptions on $b$, \autoref{th 1} grants that $\norm{\Phi^*}_{L^2(P,\meas)}<+\infty$. 
	\end{theorem}
	
	\begin{proof}
		Fix $x,y\in X$ such that $x\neq y$ and set $r:=\dist_G(x,y)$. Exploiting \autoref{prop: almost triangle inequality} and the subadditivity and monotonicity of $s\mapsto\log(1+s)$, we obtain that, for any $z\in X$, it holds
		\begin{align*}
		\log \left( 1+\frac{ \dist_G(\XX_t(x),\XX_t(y))}{C_Tr} \right) 
		\leq &\log \left( 1+\frac{ \dist_G(\XX_t(x),\XX_t(z))  }{r} \right)\\
		&+\log \left( 1+\frac{ \dist_G(\XX_t(z),\XX_t(y))}{r} \right).
		\end{align*}
		Taking the mean value w.r.t. the $z$ variable of the above written inequality over $B^G(x,r)$ and exploiting the inclusion $B^G(x,r)\subset B^G(y, 2C_Tr)$ which follows from \autoref{prop: almost triangle inequality}, we obtain
		\begin{align*}
		\log \left( 1+\frac{ \dist_G(\XX_t(x),\XX_t(y))}{C_Tr} \right)
		\le&	\dashint_{B^G(x,r)}\log \left( 1+\frac{ \dist_G(\XX_t(x),\XX_t(z))  }{r}\right)\di\meas(z)\\
		&+\dashint_{B^G(x,r)}\log \left( 1+\frac{ \dist_G(\XX_t(z),\XX_t(y))}{r} \right)\di\meas(z)\\
		\le& \Phi^{*}(x)\\
		&+   \frac{\meas(B^G(y,2C_Tr))}{\meas(B^G(x,r))}\dashint_{B^G(y,2C_Tr)}\log \left( 1+\frac{ \dist_G(\XX_t(z),\XX_t(y))}{r}\right)\di\meas(z).
		\end{align*}
		Thanks to \autoref{prop: doubling property} we can estimate $\frac{\meas(B^G(y,2C_Tr))}{\meas(B^G(x,r))}$ with a constant depending only on $C^G$ and $C_T$. Hence, there exists $C=C(N)$ (recall that $C_T$ and $C^G$ depend only on $N$) such that
		\begin{equation}\label{eq:almconclus}
		\log \left( 1+\frac{ \dist_G(\XX_t(x),\XX_t(y))}{C_T\dist_G(x,y)} \right)
		\le C(\Phi^{*}(x)+\Phi^*(y)),
		\end{equation}
		for any $x,y\in X$ such that $x\neq y$ and for any $t\in[0,T]$ and it is easily seen that \eqref{eq:almconclus} implies \eqref{lip regularity of X 1}.	 
		
		Letting now $E:=\set{x\in P\ :\Phi^*(x)\leq \norm{\Phi^*}_{L^2(P,\meas))}/\sqrt{\epsilon}}$, by Chebyshev inequality we deduce that $\meas(P\setminus E)<\epsilon$. 
		Conclusion \eqref{eq:Lipregularity} directly follows now from \eqref{lip regularity of X 1}.
	\end{proof}

	\subsection{Extension to the case of an arbitrary lower Ricci bound}\label{subsec:extnegative}
	The aim of this section is to provide regularity results for Regular Lagrangian Flows of Sobolev vector fields over $\RCD(K,N)$ metric measure spaces in the case of a possibly negative lower Ricci bound $K$.\newline
	The main difference with respect to the previously treated case of nonnegative lower Ricci curvature bound is that the regularity has to be understood in terms of the fundamental solution of an elliptic operator different from the Laplacian.\newline
	This being said the spirit of this part will be to show how to adapt the estimates of \autoref{subsec:Green} and \autoref{subsec:Lusinreg} above to this more general setting up to pay the price that they become local and less intrinsic.

	\begin{assumption}\label{assumptionnegative}
		Throughout this section we assume that $(X,\dist, \meas)$ is the tensor product between an arbitrary $\RCD(K,N-3)$ m.m.s. for some $K\in\setR$ and $4< N<+\infty$ and a Euclidean factor $(\setR^3,\dist_{\setR^3},\leb^3)$.
	\end{assumption}

	Let us stress once more that, for the purposes of the upcoming \autoref{sec:constancyofdimension}, it will be not too restrictive to have a regularity result for Regular Lagrangian flows just over spaces satisfying assumption \autoref{assumptionnegative}.
	
	Let $c\ge0$ be the constant appearing in \eqref{eq:kernelestimate} and \eqref{eq:gradientestimatekernel} and set
	\begin{equation}
	\bar{G}(x,y):=\int_0^{\infty} e^{-c t} p_t(x,y) \di t\quad\text{for any $x,y\in X$},
	\end{equation}
	and, in analogy with \eqref{eq:Geps},
	\begin{equation}
	\bar{G}^{\eps}(x,y):=\int_{\eps}^{\infty} e^{-c t} p_t(x,y)\quad \text{for any $\epsilon>0$ and any $x,y\in X$.}
	\end{equation}
	As in the case of the Green function $G$, we shall adopt in the sequel also the notation $\bar{G}_x(\cdot)=\bar{G}(x,\cdot)$ (and analogously for $\bar{G}^{\epsilon}$).\newline
	Observe that, assuming that $c>0$, $\bar{G}_x$ is well defined and belongs to $L^1(X,\meas)$ for every $x\in X$. Indeed an application of Fubini's theorem yields
	\begin{equation}\label{eq:barG in L1}
	\int_X \bar{G}_x(w) \di \meas(w)=\int_0^{\infty} e^{-ct}\int_X p_t(x,w) \di \meas(w) \di t= \int_0^{\infty}e^{-ct} \di t<\infty.
	\end{equation}
	We can also remark that the above stated conclusion holds true without any extra hypothesis on the $\RCD(K,N)$ m.m.s. $(X,\dist,\meas)$. Nevertheless, the validity of assumption \autoref{assumptionnegative} will be crucial in order to obtain meaningful estimates for $\bar{G}$ and its gradient in terms of the functions $F$ and $H$ introduced in \eqref{eq:introF}, \eqref{eq:introH}.
	
	At least at a formal level one can check that $\bar{G}$ solves the equation $\Delta\bar{G}_x=-\delta_x+c\bar{G}_x$. Indeed
	\begin{align*}
	\Delta_y\bar{G}_x(\cdot)=&\Delta_y\left(\int_{0}^{\infty}e^{-ct}p_t(x,\cdot)\di t\right)\\
	=&\int_0^{\infty}e^{-ct}\Delta_yp_t(x,\cdot)\di t
	=\int_0^{\infty}e^{-ct}\frac{\di}{\di t}p_t(x,\cdot)\di t\\
	=&\left[ p_t(x,\cdot)\right]_{0}^{\infty}+c\int_0^{\infty}e^{-ct}p_t(x,\cdot)\di t=-\delta_x+c\bar{G}_x(\cdot). 
	\end{align*}
	
	To let the above computation become rigorous, one can proceed as in the proof of \autoref{lemma:G eps regularity} and check firstly that $\bar{G}_x^{\eps}\in \Lip_b\cap D_{\loc}(\Delta)$ for any $x\in X$ and any $\epsilon>0$, with  
	\begin{equation}\label{eq: Delta barG}
	\Delta \bar{G}_x^{\eps}(y) =-e^{-c\epsilon}p_{\eps}(x,y)+c\bar{G}^{\epsilon}_x(y),\qquad
	\text{for $\meas$-a.e. $y\in X$,}
	\end{equation}
	and then that
	\begin{equation}\label{eq:barGeps to barG}
	\lim_{\eps\to 0}\bar{G}^{\eps}_x\to \bar{G}_x\qquad \text{in}\ W^{1,1}(X,\dist,\meas).
	\end{equation}
	
	Our primary goal is now to obtain useful local estimates for $\bar{G}$ and its gradient in terms of $F$ and $H$.
	
	\begin{proposition}[Main estimates for $\bar{G}$]\label{prop:estimatebarG}
		Let $(X,\dist,\meas)$ be an $\RCD(K,N)$ m.m.s. satisfying assumption \autoref{assumptionnegative}. Then, for any compact $P\subset X$, there exists $\bar{C}=\bar{C}(P)\ge 1$ such that 
		\begin{equation}\label{eq:estimatebarG}
		\frac{1}{\bar{C}}F_x(\dist(x,y))\le \bar{G}_x(y)\le \bar{C} F_x(\dist(x,y))\quad\text{for any  $x,y\in P$.}
		\end{equation}
		Moreover for any $x\in X$ it holds $\bar{G}_x\in W^{1,1}_{\loc}(X,\dist,\meas)$ and, for any $x\in P$,
		\begin{equation}\label{eq:estimate nablaGbareps}
		|\nabla \bar{G}_x|(y)\le \int_0^{\infty} e^{-ct}|\nabla p_t(x,\cdot)|(y) \di t
		\le \bar{C} H_x(\dist(x,y))
		\quad\ \text{for $\meas$-a.e. $y\in P$.}
		\end{equation}
	\end{proposition}
	
	\begin{proof}
		Applying the estimates for the heat kernel \eqref{eq:kernelestimate} we find out that
		\begin{equation}\label{z12}
		\frac{1}{C_1}\int_0^{\infty} \frac{e^{-2ct}e^{-\frac{\dist(x,y)^2}{3t}}}{\meas(B(x, \sqrt{t}))} \di t \leq
		\bar{G}_x(y) \le C_1 \int_0^{\infty} \frac{e^{-\frac{\dist(x,y)^2}{5t}}}{\meas(B(x, \sqrt{t}))} \di t\quad \text{for any $x,y\in X$.}
		\end{equation}
		Exploiting \eqref{eq:locdoubling} and \autoref{lemma:extimateintegral}, we obtain from \eqref{z12} that
		\begin{equation}\label{eq:intestbarG}
		\bar{G}_x(y)\le C_1\Lambda(R)F_x(\frac{\dist(x,y)}{\sqrt{5}})\quad\text{for any $x,y$ such that $\dist(x,y)<R$,}
		\end{equation}
		where $\Lambda$ is the function in the statement of \autoref{lemma:extimateintegral}.\newline
		The bound from above in \eqref{eq:estimatebarG} follows from \eqref{eq:intestbarG} together with the following observation, that will play a role also in the sequel: for any compact $P\subset X$, for any $R>0$ and for any $\lambda<1$, there exists $C(P,R,\lambda)\ge 0$ such that 
		\begin{equation}\label{eq:inversecontrol}
		F_x(\lambda r)\le C(P,R,\lambda)F_x(r)\quad\text{for any $x\in P$ and any $0<r<R$.}
		\end{equation}
		Indeed \eqref{eq:inversecontrol} can be checked splitting
		\begin{equation}\label{eq:firstsplit}
		F_x(\lambda r)=\int_{\lambda r}^{\lambda R}\frac{s}{\meas(B(x,s))}\di s+\int_{\lambda R}^{\infty}\frac{s}{\meas(B(x,s))}\di s,
		\end{equation}
		\begin{equation}\label{eq:secondsplit}
		F_x(r)=\int_r^{R}\frac{s}{\meas(B(x,s))}\di s+\int_{R}^{\infty}\frac{s}{\meas(B(x,s))}\di s
		\end{equation}
		and using the local doubling property \eqref{eq:locdoubling} together with a change of variables to bound the first term in \eqref{eq:firstsplit} with the first one in \eqref{eq:secondsplit} and the continuity of $x\mapsto F_x(R)$ to compare the second terms (here the compactness of $P$ comes into play).

		To obtain the lower bound in \eqref{eq:estimatebarG} we proceed as follows. 
		Starting from the lower bound in \eqref{z12}, exploiting the elementary inequality $e^{-d^2/3t}\geq e^{-1/3} \mathds{1}_{[d,\infty]}(\sqrt{t})$ and changing variables, we obtain
		\begin{equation*}
		\int_0^{\infty} \frac{e^{-2ct}e^{-\frac{\dist(x,y)^2}{3t}}}{\meas(B(x, \sqrt{t}))} \di t \ge e^{-1/3}\int_{\dist(x,y)}^{\infty} e^{-2ct^2} \frac{t}{\meas(B(x,t))} \di t.
		\end{equation*}
		To conclude it suffices to observe that, splitting the integral in two parts and using a continuity argument, as in the verification of \eqref{eq:inversecontrol} above, it is possible to find a constant $C(P)> 0$ such that
		\begin{equation*}
		\int_{\dist(x,y)}^{\infty} e^{-2ct^2} \frac{t}{\meas(B(x,t))} \di t\ge C(P)\int_{\dist(x,y)}^{\infty}\frac{t}{\meas(B(x,t))}\di t=C(P)F_x(\dist(x,y)),
		\end{equation*}
		for any $x,y\in P$.
		
		The proof of \eqref{eq:estimate nablaGbareps} can be obtained with arguments analogous to those one we presented above, starting from \eqref{eq:gradientestimatekernel} and following the strategy we adopted to prove \eqref{eq:estimate nablaGeps}.	
	\end{proof}

	Another crucial ingredient to perform the regularity scheme by Crippa-De Lellis in the case of nonnegative lower Ricci curvature bound was the scalar maximal estimate we obtained in \autoref{prop:maximalGestimate}. In \autoref{prop:maximalbarGestimate} below we prove that an analogous result holds true, in local form, also in the case of an arbitrary lower Ricci bound.
	
	\begin{proposition}[Maximal estimate, scalar version]\label{prop:maximalbarGestimate}
		Let $(X,\dist,\meas)$ be an $\RCD(K,N)$ m.m.s. satisfying assumption \autoref{assumptionnegative}.
		For any compact $P\subset X$, there exists $C_M(P)> 0$ such that, for any Borel function $f:X\to[0,+\infty)$ supported in $P$, it holds
		\begin{equation}\label{eq:barGmaximalscalar}
		\int_X f(w)\abs{\nabla \bar{G}_x(w)}\abs{\nabla \bar{G}_y(w)}\di\meas(w)\le \bar{C}_M(P) \bar{G}(x,y)\left(Mf(x)+Mf(y)\right),
		\end{equation}
		for any $x,y\in P$.
	\end{proposition}
	
	\begin{proof}
		We begin by recalling that, as an intermediate step in the proof of \autoref{prop:maximalGestimate}, we obtained the following inequality:
		\begin{align}\label{eq:generalelem}
		\int_Xf(w)H_x&(\dist(x,w))H_y(\dist(y,w))\di\meas(w)\\\nonumber &\le C\left(F_x\left(\frac{\dist(x,y)}{2}\right)+F_y\left(\frac{\dist(x,y)}{2}\right)+F_x(\dist(x,y))+F_y(\dist(x,y))\right)\left(Mf(x)+Mf(y)\right), 
		\end{align}
		for any $x,y\in X$, for some numerical constant $C>0$ (the assumptions concerning the m.m.s. $(X,\dist,\meas)$ played no role in that part of the proof).\newline
		Let us observe then that, thanks to \eqref{eq:estimate nablaGbareps},
		\begin{equation}\label{eq:aux1}
		\int_X f(w)\abs{\nabla \bar{G}_x(w)}\abs{\nabla \bar{G}_y(w)}\di\meas(w)\le \bar{C}(P)^2\int_Xf(w)H_x(\dist(x,w))H_y(\dist(y,w))\di\meas(w)
		\end{equation}
		for any $x,y\in P$. Exploiting \eqref{eq:inversecontrol} with $\lambda=1/2$, \eqref{eq:generalelem} and \eqref{eq:aux1}, we obtain that, up to increasing the constant $\bar{C}(P)$, it holds
		\begin{equation}\label{eq:aux2}
		\int_X f(w)\abs{\nabla \bar{G}_x(w)}\abs{\nabla \bar{G}_y(w)}\di\meas(w)\le \bar{C}(P)\left(F_x(\dist(x,y))+F_y(\dist(x,y))\right)(Mf(x)+Mf(y)),
		\end{equation}
		for any $x,y\in P$.\newline
		The sought conclusion \eqref{eq:barGmaximalscalar} follows from \eqref{eq:aux2} and the lower bound in \eqref{eq:estimatebarG}. 
	\end{proof}
	
	\begin{remark}\label{rm:regularizedmaxscalarbar}
		It follows from the proof of \autoref{prop:maximalbarGestimate} above that also the estimate
		\begin{equation}
		\int_X f(w)\abs{\nabla \bar{G}^{\epsilon}_x(w)}\abs{\nabla \bar{G}^{\epsilon}_y(w)}\di\meas(w)\le \bar{C}_M(P) \bar{G}(x,y)\left(Mf(x)+Mf(y)\right)
		\end{equation}
		holds true, for any $\epsilon >0$ and for any $x,y\in P$.
	\end{remark}

	By analogy with \eqref{eq:Green distance}, we introduce a function $\dist_{\bar{G}}$, that we will use to measure the regularity of RLFs, in the following way:
	\begin{equation}\label{eq:barGreen distance}
	\dist_{\bar{G}}(x,y):=
	\begin{cases*}
	\frac{1}{\bar{G}(x,y)}& if $x\neq y$,\\
	0  & otherwise.
	\end{cases*}	
	\end{equation} 
	It is immediate to check that is symmetric, nonnegative and that $\dist_{\bar{G}}(x,y)=0$ if and only if $x=y$.
	Moreover, following verbatim the proof of \autoref{lemma:continuitycontinuity} and exploiting the two-sided bounds in \eqref{eq:estimatebarG}, it is easy to prove that, for any $x\in X$, the map $y\to \dist_{\bar{G}}(x,y)$ is continuous with respect to $\dist$.\newline
	By analogy with \eqref{eq: Gballs}, we introduce the notation $B^{\bar{G}}$ for the ``balls'' associated to $\dist_{\bar{G}}$, that is to say, for any $x\in X$ and for any $r>0$, we let
	\begin{equation*}
	B^{\bar{G}}(x,r):=\set{y\in X\text{}:\dist_{\bar{G}}(x,y)<r}.
	\end{equation*}

	The aim of \autoref{prop: almost triangle inequality barG} and \autoref{prop: doubling property barG} below is to show that, at least locally, $\dist_{\bar{G}}$ is a quasi-metric on $X$ and that $(X,\dist_{\bar{G}},\meas)$ is a locally doubling quasi-metric measure space.
	
	\begin{proposition}[Local almost triangle inequality]\label{prop: almost triangle inequality barG} 
		Let $(X,\dist,\meas)$ be an $\RCD(K,N)$ m.m.s. satisfying assumption \autoref{assumptionnegative}.
		Then, for any compact $P\subset X$, there exists a constant $\bar{C}_T(P)\ge 1$ such that
		\begin{equation}\label{eq:trianglebarG}
		\dist_{\bar{G}}(x,y)\leq \bar{C}_T(P) (\dist_{\bar{G}}(x,z)+\dist_{\bar{G}}(z,y))\quad\text{for any $x,y,z\in P$.}	
		\end{equation}	
	\end{proposition}
	
	\begin{proof}
		Recall that, as an intermediate step in the proof of \autoref{lemma:firststeptriangle}, we proved that, without any further assumption on the m.m.s. $(X,\dist,\meas)$, it holds 
		\begin{equation}\label{eq:recalllemma}
		F_x(\dist(x,z))F_y(\dist(y,z))\le F_x(\dist(x,y)/2)F_y(\dist(y,z))+F_y(\dist(x,y)/2)F_x(\dist(x,z)),
		\end{equation}
		for any $x,y,z\in X$.\newline
		Applying \eqref{eq:inversecontrol} with $\lambda=1/2$ and exploiting the two-sided bounds in \eqref{eq:estimatebarG}, we pass from \eqref{eq:recalllemma} to the sought \eqref{eq:trianglebarG}.
	\end{proof}
	
	\begin{remark}\label{rm:boundednessbarG}
		A first non completely trivial consequence of \autoref{prop: almost triangle inequality barG} is that any compact $P\subset X$ is bounded w.r.t. the $\dist_{\bar{G}}$ quasi-metric. 
	\end{remark}

	\begin{proposition}[Local doubling property]\label{prop: doubling property barG}
		Let $(X,\dist,\meas)$ be an $\RCD(K,N)$ m.m.s. satisfying assumption \autoref{assumptionnegative}.
		Then, for any compact $P\subset X$ and for any $R>0$, there exists a constant $\bar{C}^{\bar{G}}(P,R)>0$ such that
		\begin{equation}\label{eq:doublocbarG}
		\meas (B^{\bar{G}}(x,2r))\leq \bar{C}^{\bar{G}} \meas (B^{\bar{G}}(x,r)),\quad\text{for any $x\in P$ and for any $0<r<R$.}
		\end{equation}
	\end{proposition}
	
	\begin{proof}
		The conclusion can be obtained arguing as in the proof of \autoref{prop: doubling property}, exploiting the fact that $(X,\dist,\meas)$ is locally doubling (see \eqref{eq:locdoubling}) and the local comparison between $\bar{G}$ and $F$ obtained in \eqref{eq:estimatebarG}. We just indicate here the adjustments one has to do.
		
		First of all, we observe that a local version of \autoref{lemma: reverse doubling} holds true, namely if $(\bar{X},\bar{\dist},\mu)$ is a locally doubling m.m.s. with function $C_{\mu}:(0,+\infty)\to(0,+\infty)$ and $(X,\dist,\meas)$ is the tensor product between $(\bar{X},\bar{\dist},\mu)$ and $(\setR^k,\dist_{\setR^k},\leb^k)$, it holds
		\begin{equation}\label{eq:locreversedoubling}
		\frac{\meas(B(x,R))}{\meas(B(x,r))}\geq \frac{1}{C_{\mu}(R)\sqrt{2}^k}\left(\frac{R}{r}\right)^k,
		\end{equation}
		for any $x\in X$ and for any $0<r<R$.\newline
		We wish to obtain a local version of \eqref{eq:intermboundabove}, that is to say the existence of $\gamma=\gamma(P,R)$ such that
		\begin{equation}
		F_x(\frac{r}{2})\le\gamma F_x(r),\quad\text{for any $x\in P$ and for any $0<r<R$.}
		\end{equation}
		This can be obtained arguing as in the proof of \eqref{eq:intermboundabove}, exploiting the splitting of the integration intervals we introduced in \eqref{eq:firstsplit}, \eqref{eq:secondsplit} and assumption \autoref{assumptionnegative} together with \eqref{eq:locreversedoubling} in place of \autoref{lemma: reverse doubling}.\newline
		The validity of \eqref{eq:locreversedoubling} implies in turn that, for any $S>0$, we can find $\gamma<1$ such that
		\begin{equation}\label{eq:locinversiondoub}
		F_x^{-1}(\gamma s)\le F_x^{-1}(s),\quad\text{for any $x\in P$ and for any $s>S$.}
		\end{equation} 
		Having \eqref{eq:locinversiondoub} at our disposal, we can achieve \eqref{eq:doublocbarG} exploiting the local version of \eqref{a4}, which is a consequence of \eqref{eq:estimatebarG}\footnote{In the whole proof we tacitly exploited the fact that any compact subset of $X$ is both $\dist$-bounded and $\dist_{\bar{G}}$-bounded, see \autoref{rm:boundednessbarG} above.}.  
	\end{proof}
	
	We end this introductory discussion about the properties of the modified Green function $\bar{G}$ with a vector valued maximal estimate. In the proof of \autoref{thm: sobolev vectorfield, barG-regularity} below it will play the same role that \autoref{prop: G-maximal estimate} played in the proof of \autoref{th 1}.
	
	\begin{proposition}[Maximal estimate, vector-valued version]\label{prop: barG-maximal estimate}	
		Let $(X,\dist,\meas)$ be an $\RCD(K,N)$ m.m.s. satisfying assumption \autoref{assumptionnegative} and let $P\subset X$ be a compact set.
			Then, for any $b\in W^{1,2}_{C,s}(TX)$ bounded and with compact support in $P$, there exists a positive function $F\in L^2(P,\meas)$ such that
			\begin{equation}\label{z14}
			\abs{ b\cdot\nabla \bar{G}_x(y)+ b\cdot\nabla \bar{G}_y(x)}\le \bar{G}(x,y)(F(x)+F(y))\quad\text{for $\meas\times\meas$-a.e. $(x,y)\in P\times P$,}
			\end{equation}
			and
			\begin{equation}\label{eq: z}
			\norm{F}_{L^2(P,\meas)}\le C_V \norm{|\nabla_{\sym} b|+|\div b|}_{L^2(X,\meas)},
			\end{equation}
			where $C_V=C_V(P)>0$.
		
	\end{proposition}
	
	\begin{proof}
		The strategy we follow is the same proposed in the proof of \autoref{prop: G-maximal estimate}.\newline
		First we are going to prove that there exists $F$ as above such that
		\begin{equation}\label{eq: intermediateMaximalVectorfileds bar}
		\abs{\int_X \left\lbrace b\cdot\nabla \bar{G}^{\eps}_x(w) p_{\eps}(y,w)+b\cdot\nabla \bar{G}^{\eps}_y(w) p_{\eps}(x,w)\right\rbrace  \di \meas(w)}
		\le \bar{G}(x,y)\left(F(x)+F(y)\right),
		\end{equation} 
		for any $x,y\in P$ and for any $0<\epsilon<1$. The stated conclusion will then follow from \eqref{eq: intermediateMaximalVectorfileds bar}, taking into account \eqref{eq:barGeps to barG} and following verbatim the second step of the proof of \autoref{prop: G-maximal estimate}.   
		
		Recall from \eqref{eq: Delta barG} that $p_{\eps}(x,w)=e^{c\epsilon}[-\Delta \bar{G}_x^{\eps}(w)+c\bar{G}^{\epsilon}_x(w)]$ for $\meas$-a.e. $w\in X$. Hence we can estimate
		\begin{align*}
		\Big|\int_X  &b\cdot\nabla \bar{G}^{\eps}_x  (w) p_{\eps}(y,w)+b\cdot\nabla \bar{G}^{\eps}_y(w) p_{\eps}(x,w)  \di \meas(w)\Big|\\
		=&
		e^{c\epsilon}\abs{\int_X \left\lbrace b\cdot\nabla \bar{G}^{\eps}_x(w) (-\Delta \bar{G}_y^{\eps}(w)+c\bar{G}^{\epsilon}_y(w))+b\cdot\nabla \bar{G}^{\eps}_y(w)(-\Delta \bar{G}_x^{\eps}(w)+c\bar{G}^{\epsilon}_x(w))\right\rbrace  \di \meas(w)}\\
		\le &e^{c\epsilon} 	
		\abs{\int_X \left\lbrace b\cdot\nabla \bar{G}^{\eps}_x \Delta\bar{G}_y^{\eps}+b\cdot\nabla \bar{G}^{\eps}_y\Delta \bar{G}_x^{\eps} \right\rbrace  \di \meas}+ce^{c\epsilon}\abs{\int_X \left\lbrace b\cdot\nabla \bar{G}^{\eps}_x \ \bar{G}_y^{\eps}+b\cdot\nabla \bar{G}^{\eps}_y\ \bar{G}_x^{\eps} \right\rbrace  \di \meas}\\
		=:& I_1^{\epsilon}(x,y)+I^{\epsilon}_2(x,y).
		\end{align*}
		Arguing as in the first step of the proof of \autoref{prop: G-maximal estimate} and applying \autoref{rm:regularizedmaxscalarbar}, we obtain that
		\begin{equation}\label{eq:estimatefirstterm}
		I^{\epsilon}_1(x,y)\le e^{c\epsilon} \bar{C}_M(P)\bar{G}(x,y)(Mg(x)+Mg(y)),\quad\text{for any $x,y\in P$ and for any $0<\epsilon<1$,}
		\end{equation}
		where $g:=|\nabla_{\sym} b|+|\div b|$.
		Dealing with $I^{\epsilon}_2$, let us apply \eqref{eq:divergence}, Leibniz rule and eventually \autoref{prop: almost triangle inequality barG}, to obtain that
		\begin{equation}\label{eq:maxa}
		I^{\epsilon}_2(x,y)=ce^{c\epsilon}\abs{\int_X \div b\ \bar{G}_x^{\eps}\bar{G}_y^{\eps} \di \meas}
		\le ce^{c\epsilon}\bar{C}_T(P)\bar{G}(x,y) \left( \int_X g\bar{G}_x\di\meas +\int_X g\bar{G}_y\di\meas \right),
		\end{equation}
		for any $x,y\in P$.\newline
		Let us set
		\begin{equation*}
		F(x):=e^cC_M(P)Mg(x)+ce^c\int_X g\bar{G_x}\di \meas,
		\qquad
		\forall x\in P.	
		\end{equation*} 
		It remains only to show \eqref{eq: z}. To this aim we recall \eqref{eq:LocalBoundnessMaximaloperator} and we
		observe that,
		\begin{equation}\label{eq:maxb}
		\int_X \left(\int_X g\bar{G_x}\di \meas\right)^2\di \meas(x)
		=\int_X \left( \int_{0}^{\infty}e^{-ct} P_tg(x) \di t\right)^2\di \meas(x)
		\le c^{-1}\norm{g}^2_{L^2(X,\meas)}
		\end{equation}
		The proof is complete.
	\end{proof}
	
	With \autoref{prop: almost triangle inequality barG}, \autoref{prop: doubling property barG} and \autoref{prop: barG-maximal estimate} at our disposal we can develop a regularity theory for Regular Lagrangian flows of Sobolev vector fields in terms of the quasi-metric $\dist_{\bar{G}}$.\newline
	To this aim let us fix $T>0$ and let $b_t\in L^{\infty}((0,T)\times X)$ be a time dependent vector field with compact support, uniformly w.r.t. time. We further assume that $b_t\in W^{1,2}_{C,s}(TX)$ for a.e. $t\in(0,T)$, that $|\nabla_{\sym} b_t|\in L^1((0,T); L^2(X,\meas))$ and that $\div b_t\in L^1((0,T);L^{\infty}(X,\meas))$.
	
	Let $(\XX_t)_{t\in[0,T]}$ be the Regular Lagrangian flow of $b$, whose existence and uniqueness follow by the theory developed in \cite{AmbrosioTrevisan14}. In analogy with the case of nonnegative lower Ricci curvature bound, we set 
	\begin{equation}
	\bar{\Phi}_{t,r}(x):=\dashint_{B^{\bar{G}}(x,r)}\log \left( 1+\frac{\dist_{\bar{G}}(\XX_t(x),\XX_t(y))}{r} \right) \di \meas(y),
	\end{equation}
	for $r\in (0,\infty)$ and $t\in [0,T]$ and, for any $R>0$,
	\begin{equation}\label{def: barPhi star}
	\bar{\Phi}_R^{*}(x):=\sup_{0\leq t\leq T}\sup_{0<r<R}\bar{\Phi}_{t,r}(x).
	\end{equation}

	Below we state the main regularity result of this part. Its proof can be obtained from the result of this subsection, using \autoref{remark: corollaryA3}, recalling \eqref{eq:LocalBoundnessMaximaloperator} and arguing as in the proofs of \autoref{th 1} and \autoref{thm: sobolev vectorfield, G-regularity}.

	\begin{theorem}\label{thm: sobolev vectorfield, barG-regularity}
		Let $(X,\dist,\meas)$ be an $\RCD(K,N)$ m.m.s. satisfying assumption \autoref{assumptionnegative}.
		Let $b$ and $(\XX_t)_{t\in[0,T]}$ be as in the discussion above.
		Then, for any compact $P\subset X$ such that $P$ contains the $\left(T\norm{b}_{L^{\infty}}\right)$-enlargement of $\supp b$, there exist $\bar{C}>0$ and $R>0$, depending on $P$, such that for any $x,y\in P $ and for any $t\in[0,T]$, it holds
		\begin{equation*}
		\dist_{\bar{G}}(\XX_t(x),\XX_t(y))\le \bar{C} e^{\bar{C}\left(\bar{\Phi}_R^*(x)+\bar{\Phi}_R^*(y)\right)}\dist_{\bar{G}}(x,y).
		\end{equation*}
		Moreover, $\bar{\Phi}^*_R$ belongs to $L^{2}(P,\meas)$ and the following Lusin-approximation property holds:
		for every $\epsilon>0$ there exists a Borel set $E\subset P$ such that
		$\meas(P\setminus E)<\eps$ and 
		\begin{equation*}
		\dist_{\bar{G}}(\XX_t(x),\XX_t(y))\leq \bar{C}\exp\left(2\bar{C}\frac{\norm{\bar{\Phi}^*_R}_{L^2(P,\meas)}}{\sqrt{\epsilon}}\right) \dist_{\bar{G}}(x,y)
		\quad \text{for any $x,y\in P$ and for any $t\in[0,T]$.}
		\end{equation*} 
	\end{theorem}
	
	In analogy with the case of real valued functions (where the Lipschitz regularity is understood w.r.t. the distance $\dist$) we introduce the following.
	
	\begin{definition}\label{def:dGLusinLip}
		Let $(X,\dist,\meas)$ be an $\RCD(0,N)$ m.m.s. satisfying assumption \autoref{assumption: good definition of G}. We say that a map $\Phi:X\to X$ is $\dist_G$-Lusin Lipschitz if there exists a family $\set{E_n: n\in\setN}$ of Borel subsets of $X$ such that $\meas(X\setminus\cup_{n\in\setN}E_n)=0$ and 
		\begin{equation*}
		\dist_G(\Phi(x),\Phi(y))\le n\dist_G(x,y),
		\end{equation*}
		for any $x,y\in E_n$ and for any $n\in\setN$.\newline
		By analogy, if $(X,\dist,\meas)$ is an $\RCD(K,N)$ m.m.s. satisfying assumption \autoref{assumptionnegative}, we say that $\Psi:X\to X$ is $\dist_{\bar{G}}$-Lusin Lipschitz if it satisfies the above conditions with $\dist_{\bar{G}}$ in place of $\dist_{G}$.
	\end{definition}
	
	Let us remark that, with the terminology we just introduced, we can combine \autoref{th 1} and \autoref{thm: sobolev vectorfield, G-regularity} above to say that the Regular Lagrangian flow of a sufficiently regular vector field over an $\RCD(0,N)$ m.m.s. satisfying assumption \autoref{ass:doublingG} is a $\dist_G$-Lusin Lipschitz map (the RLF of a sufficiently regular vector field over an $\RCD(K,N)$ m.m.s. satisfying assumption \autoref{assumptionnegative} is a $\dist_{\bar{G}}$-Lusin Lipschitz map, respectively). 
	
	\section{Constancy of the dimension}\label{sec:constancyofdimension}
	
	The aim of this section is to prove \autoref{thm:constancyofdimension} below, that could be restated by saying that, if $(X,\dist,\meas)$ is an $\RCD(K,N)$ m.m.s. for some $K\in\setR$ and $1\le N<+\infty$, then there exists a natural number $1\le n\le N$ such that the tangent cone of $(X,\dist,\meas)$ is the $n$-dimensional Euclidean space at $\meas$-almost every point in $X$. \footnote{We point out that, in the very recent \cite{Honda18}, Honda constructs a family of spaces satisfying the Bakry-\'Emery condition (see \cite{AmbrosioMondinoSavare15,ErbarKuwadaSturm15}), but not the \textit{Sobolev to Lipschitz} property, having regular sets of different dimensions with positive measure}
	In this way we extend to this abstract framework a relatively recent result obtained by Colding-Naber in \cite{ColdingNaber12} for Ricci-limit spaces. 
	
	Let us spend a few words about the strategy we are going to implement, which is of different nature with respect to the one adopted in \cite{ColdingNaber12}, since we cannot rely on the existence of a smooth approximating sequence for $(X,\dist,\meas)$.\newline
	We begin remarking that the statement of \autoref{thm:constancyofdimension} is not affected by taking the tensor product with Euclidean factors and, by means of this simple observation, we will put ourselves in position to apply the results of \autoref{sec:Gregularity} and \autoref{sec:RLFpreservedimension}.\newline
	This being said, we will argue as follows.
	In \autoref{sec:RLFpreservedimension} below we start proving that $\dist_{G}/\dist_{\bar{G}}$-Lusin Lipschitz maps with bounded compressiblity from an $\RCD(K,N)$ m.m.s. into itself are regular enough to carry an information about the dimension from their domain to their image. This rigidity result has to be compared with the standard fact that biLipschitz maps preserve the Hausdorff dimension.\newline 
	Then we are going to prove that the class of RLFs of Sobolev vector fields, that we know to be $\dist_{G}/\dist_{\bar{G}}$-Lusin Lipschitz from \autoref{sec:Gregularity}, is rich enough to gain ``transitivity'' at the level of probability measures with bounded support and bounded density w.r.t. $\meas$. Better said, the primary goal of \autoref{subsec:regularityintheinterior} below will be to show that any two probability measures which are intermediate points of a $W_2$-geodesic joining probabilities with bounded support and bounded density w.r.t. $\meas$ can be obtained one from the other via push-forward through the RLF of a vector field satisfying the assumptions of \autoref{thm: sobolev vectorfield, G-regularity} (or \autoref{thm: sobolev vectorfield, barG-regularity}).\newline 
	Eventually in \autoref{subsec:conclusion} we will combine all the previously developed ingredients to prove that the above mentioned ``transitivity'' is not compatible with the ``rigidity'' we obtain in \autoref{sec:RLFpreservedimension} and the possibility of having non negligible regular sets of different dimensions in the Mondino-Naber decomposition of $(X,\dist,\meas)$.    	
	
	\subsection{A rigidity result for $\dist_{G}$-Lusin Lipschitz maps}\label{sec:RLFpreservedimension}
	
	The aim of this subsection is to prove a rigidity result for $\dist_G$ and $\dist_{\bar{G}}$-Lusin Lipschitz maps (see \autoref{def:dGLusinLip}) that we are going to apply later on to Regular Lagrangian Flows of Sobolev vector fields.\newline	
	Roughly speaking, given an $\RCD(K,N)$ m.m.s. satisfying assumption \autoref{assumptionnegative}, we are going to prove that a $\dist_{G}/\dist_{\bar{G}}$-Lusin Lipschitz map with bounded compressibility cannot move a part of dimension $n$ of $(X,\dist,\meas)$ into a part of dimension $k<n$ (see \autoref{thm:conservationhausdim} below for a precise statement).
	Just at a speculative level, let us point out that, in the case of $\dist$-Lusin Lipschitz maps, this conclusion would have been a direct consequence of standard geometric measure theory arguments. However, a priori, it is not clear how to build non trivial maps from the space into itself with $\dist$-Lusin Lipschitz regularity, while in \autoref{sec:Gregularity} above we were able to obtain $\dist_{G}$-Lusin Lipschitz regularity for a very rich family of maps\footnote{A posteriori, one of the consequences of \autoref{thm:constancyofdimension} will be that flow maps are also $\dist$-Lusin Lipschitz, see \autoref{thm:LusinLip}}.   
	
	We begin with a Euclidean result. It can be considered in some sense as a much simplified version of Sard's lemma.
	
	\begin{proposition}\label{prop:sard}
		Fix $k,n\in\setN$ such that $1\le k<n$. Let $A\subset\setR^n$, $\Phi:A\to\setR^k$ be such that
		\begin{equation}\label{eq:improvedregularity}
		\lim_{r\to 0^+}\sup_{y\in A\cap B(x,r)}\frac{\abs{\Phi(y)-\Phi(x)}}{\abs{y-x}^{\frac{n}{k}}}=0,\quad\text{for any $x\in A$.}
		\end{equation}
		Then $\mathcal{H}^{k}(\Phi(A))=0$.
	\end{proposition}
	
	\begin{proof}
		We wish to prove that $\mathcal{H}^{k}_{\delta}(\Phi(A))= 0$ for any $\delta>0$. Let us assume without loss of generality that $A\subset P$ for some compact $P\subset\setR^n$. 
		Fix now $\epsilon>0$. It follows from \eqref{eq:improvedregularity} that, for any $x\in A$, we can find $r_x<\delta/10$ such that, for any $y\in B(x, 5r_x)\cap A$, it holds
		\begin{equation}\label{eq:selection}
		\abs{\Phi(y)-\Phi(x)}\le \epsilon\abs{x-y}^{\frac{n}{k}}.
		\end{equation}
		Moreover, if $\epsilon,\delta<1$, then \eqref{eq:selection} grants that $\Phi(B(x,5r_x)\cap A)$ has diameter less than $\delta$, for any $x\in A$. 
		
		Applying Vitali's covering theorem we can find a subfamily $\mathcal{F}:=\set{B(x_i,r_i):i\in\setN}$ such that the balls $B(x_i,r_i)$ are disjoint and $A\subset\cup_{i\in\setN} B(x_i,5r_i)$. Hence $\set{\Phi(A\cap B(x_i,5r_i)):i\in\setN}$ is an admissible covering of $\Phi(A)$ in the definition of $\mathcal{H}^k_{\delta}(\Phi(A))$. Therefore
		\begin{equation*}
		\mathcal{H}^{k}_{\delta}(\Phi(A))\le \omega_k5^{n}\epsilon^{k}\sum_{i=0}^{\infty}r_i^n,
		\end{equation*}
		since it follows from \eqref{eq:selection} that $\Phi(B(x_i,5r_i)\cap A)\subset B(\Phi(x_i),\epsilon 5^{\frac{n}{k}}r_i^{\frac{n}{k}})$ for any $i\in\setN$. Observing now that $\sum_{i=0}^{\infty}\omega_nr_i^n\le \mathcal{H}^n(P^1)<+\infty$, where $P^1$ is the $1$-enlargement of the compact $P$, we conclude that $\mathcal{H}^k_{\delta}(\Phi(A))=0$ for any $\delta>0$, as we claimed.
	\end{proof}

	\begin{remark}
		The proof of \autoref{prop:sard} above resembles the part of the proof of Sard's lemma where it is shown that the image of the set of points where all the derivatives vanish up to a certain order is negligible (see for instance \cite{Figalli08}).
		Recall that the classical Sard's lemma requires some regularity of the map and that the highest is the difference between the dimension of the domain an the dimension of the codomain the highest is the regularity to be required. Actually, even if we do not explicitly require any sort of regularity for $\Phi$, \eqref{eq:improvedregularity} is essentially telling us that the map is differentiable with vanishing derivatives up to the order $n/k$.  
	\end{remark}

	It is a rather classical fact in Riemannian geometry that on an $n$-dimensional compact Riemannian manifold with $n>2$ the Green function behaves locally like the distance raised to the power $2-n$ (see \cite[Chapter 4]{Aubin98}). The comparison is also global on a non compact manifold with nonnegative Ricci curvature and Euclidean volume growth (see \cite{Ding02}) and, in our previous work \cite{BrueSemola18}, we extended these results to Ahlfors regular $\RCD(K,N)$ metric measure spaces.\newline 
	The aim of \autoref{lemma:asymptotics} below is to prove that the weak Ahlfors regularity result of \autoref{thm:weakAhlforsregularity}	is enough to obtain an asymptotic version of the comparison above on any $\RCD(K,N)$ m.m.s. satisfying assumption \autoref{assumption: good definition of G}.
	
	\begin{lemma}\label{lemma:asymptotics}
		Let $(X,\dist,\meas)$ be an $\RCD(K,N)$ m.m.s. satisfying assumption \autoref{assumption: good definition of G}.
		Suppose that $x\in \mathcal{R}_k^{*}$ for some $k\ge 3$ and denote by $\theta_k(x)\in(0,+\infty)$ the value of the limit appearing in \eqref{eq:kregsub}. Then
		\begin{equation*}
		\lim_{r\to 0^+}\frac{F(x,r)}{\frac{1}{r^{k-2}}}=\frac{k-2}{\omega_k\theta_k(x)}.
		\end{equation*}
	\end{lemma}
	
	\begin{proof}
		Let us observe that
		\begin{equation*}
		\frac{F(x,r)}{\frac{1}{r^{k-2}}}=(k-2)\frac{\int_{r}^{+\infty}\frac{s}{\meas(B(x,s))}\di s}{\int_r^{+\infty}\frac{1}{s^{k-1}}\di s}.
		\end{equation*}
		An application of De L'Hopital's rule yields now
		\begin{equation*}
		\lim_{r\to 0^+} \frac{F(x,r)}{\frac{1}{r^{k-2}}} =\lim_{r\to 0^+}(k-2) \frac{\frac{r}{\meas(B(x,r))}}{\frac{1}{r^{k-1}}} =\frac{k-2}{\omega_k\theta_k(x)},
		\end{equation*}
		since, by the very definition of $\theta_k(x)$, it holds $\lim_{r\to 0^+}\frac{\meas(B(x,r))}{\omega_kr^{k}}=\theta_k(x)$.
	\end{proof}
	
	Let us assume, up to the end of this section that $(X,\dist,\meas)$ is an $\RCD(K,N)$ m.m.s. satisfying assumption \autoref{assumptionnegative}. 
	It is not difficult to check that, under this assumption, the regular sets $\mathcal{R}_k$  of $(X,\dist,\meas)$ associated to $k=0,1$ and $2$ are empty.\newline 
	Below we state and prove the main result of this part of this subsection. Roughly speaking it can be rephrased by saying that a $\dist_{G}/\dist_{\bar{G}}$-Lusin Lipschitz map with bounded compressibility cannot map a portion of the space of a certain dimension into a lower dimensional one.
	
	\begin{theorem}\label{thm:conservationhausdim}
		Let $(X,\dist,\meas)$ be as in the discussion above. Let $\Phi:X\to X$ be either a $\dist_G$-Lusin Lipschitz or a $\dist_{\bar{G}}$-Lusin Lipschitz map (see \autoref{def:dGLusinLip} above). Fix $\mu\in\Prob(X)$ absolutely continuous w.r.t. $\meas$ and assume that $\nu:=\Phi_{\sharp}\mu\ll\meas$. If $\mu$ is concentrated on $\mathcal{R}_n$ for some $n\ge 3$, then $\nu$ is concentrated on $\cup_{k\ge n}\mathcal{R}_k$.
	\end{theorem}
	
	\begin{proof}
		Before entering into the details of the proof, that we will divide into two steps, we briefly outline its strategy.\newline
		The first step consists into proving that, if we have a $\dist_G/\dist_{\bar{G}}$-Lusin Lipschitz map which maps a subset of $\mathcal{R}_n^*$ into $\mathcal{R}_{k}^*$ for some $n>k\ge3$ and we read it after the composition with bi-Lipschtiz charts, then we essentially end up with a map from a subset of $\setR^n$ into $\setR^k$ which satisfies the assumptions of \autoref{prop:sard}.\newline 	
		In the second step we show how this information can be used to prove that $\nu=\Phi_{\sharp}\mu$ is concentrated over $\cup_{k\ge n}\mathcal{R}_k$, a formal argument being the following one: suppose that $\meas(\Phi(\mathcal{R}_n^*)\cap\mathcal{R}_k^*)=0$, then, neglecting the measurability issues, we could compute
		\begin{equation*}
		\Phi_{\sharp}\mu(\mathcal{R}_k^*)=\mu\left(\Phi^{-1}(\mathcal{R}_k^*)\right)=\mu\left(\Phi^{-1}(\mathcal{R}_k^*)\cap\mathcal{R}_n^*\right)\le \mu\left(\Phi^{-1}(\mathcal{R}_k^*\cap\Phi(\mathcal{R}_n^*))\right)=\Phi_{\sharp}\mu(\mathcal{R}_k^*\cap\mathcal{R}_n^*)=0.
		\end{equation*} 
		
		\textbf{Step 1.}
		Recall from \autoref{thm:rectifiability} that, for any $3\le l\le N$, we can find $\mathcal{S}_l^*\subset\mathcal{R}_l^*$ such that $\meas(\mathcal{R}_l^*\setminus\mathcal{S}_l^*)=0$ and $\mathcal{S}_l^*$ is a countable union of Borel sets which are $2$-biLipschitz equivalent to subsets of $\setR^l$.\newline 
		We want to prove that, if $P\subset\mathcal{S}_n^{*}$ is such that $\Phi$ is $\dist_G/\dist_{\bar{G}}$-Lipschitz over $P$, then $\mathcal{H}^k(\Phi(P)\cap\mathcal{R}_{k}^{*})=0$ for any $3\le k<n$. Since $\mathcal{H}^k\llcorner\mathcal{R}_k^*$ and $\meas\llcorner\mathcal{R}_k^*$ are mutually absolutely continuous (see \autoref{thm:weakAhlforsregularity}) and $\meas(\mathcal{R}_k^*\setminus\mathcal{S}_k^*)=0$, it suffices to prove that $\mathcal{H}^k(\Phi(P)\cap\mathcal{S}_k^*)=0$.	 
		Therefore, to prove the claimed conclusion, we can reduce ourselves to the case when $P$ is contained into the domain of an $n$-dimensional $2$-biLipschitz chart, that we shall call $\alpha$, and $\Phi(P)$ is contained in the domain of a $k$-dimensional $2$-biLipschitz chart, that we shall call $\beta$.\newline
		Next, with the aid of \autoref{lemma:continuitycontinuity} and \autoref{lemma:asymptotics} above,
		we wish to prove that, for any $x\in P$, it holds
		\begin{equation}\label{eqref:eqasymptdG}
		\lim_{r\to 0^+}\sup_{y\in B(x,r)\cap P}\frac{\dist(\Phi(x),\Phi(y))}{\dist(x,y)^{\frac{n-2}{k-2}}}<+\infty.
		\end{equation} 
		To this aim we observe that, by the very definition of $\dist_G$ and thanks to the two-sided bounds we obtained in \autoref{prop:estimateG}, the $\dist_{G}$-Lipschitz regularity assumption can be turned into
		\begin{equation*}
		\lim_{r\to 0^+}\sup_{y\in B(x,r)\cap P}\frac{F(x,\dist(x,y))}{F(\Phi(x),\dist(\Phi(x),\Phi(y)))}<+\infty
		\end{equation*}
		and the same holds true in case we are working with $\dist_{\bar{G}}$, thanks to \eqref{eq:estimatebarG}.
		Observe now that \autoref{lemma:continuitycontinuity} grants that, as $\dist(x,y)\to0$, also $\dist_G(x,y)\to 0$ (and an analogous result holds for $\dist_{\bar{G}}$, as we observed after \eqref{eq:barGreen distance}). Hence we can apply \autoref{lemma:asymptotics} to obtain, taking into account the fact that $x\in\mathcal{R}_n^{*}$ and $\Phi(x)\in\mathcal{R}_{k}^{*}$,
		\begin{equation*}
		\lim_{r\to 0^+}\sup_{y\in B(x,r)\cap P}\frac{\dist(\Phi(x),\Phi(y))^{k-2}}{\dist(x,y)^{n-2}}<\infty,
		\end{equation*}
		which easily yields \eqref{eqref:eqasymptdG}. 
		
		This being said, observe that, denoting by $\phi:=\beta\circ\Phi\circ\alpha^{-1}:\alpha(P)\to\beta(\Phi(P))$ (where we remark that $\alpha(P)\subset\setR^n$ and $\beta(\Phi(P))\subset\setR^k$), the map $\phi$ satisfies the assumptions of \autoref{prop:sard}, since $(n-2)/(k-2)>n/k$. Therefore $\mathcal{H}^{k}(\beta(\Phi(P)))=0$. Hence $\mathcal{H}^{k}(\Phi(P))=0$, since $\beta$ is $2$-bi-Lipschitz.
		
		It easily follows that, if $Q\subset\mathcal{S}_{n}^{*}$, $Q=\cup_{i\in\setN}Q_i$ where $\Phi|_{Q_i}$ is $\dist_G/\dist_{\bar{G}}$-Lipschitz for any $i\in\setN$, then $\mathcal{H}^{k}(\mathcal{R}_k^{*}\cap\Phi(Q))=0$ for any $3\le k<n$.  
		
		\textbf{Step 2.}
		Suppose by contradiction that
		\begin{equation*}
		\nu\left(\bigcup_{k<n}\mathcal{R}_k\right)>0.
		\end{equation*}	
		Then we can find $k<n$ such that $\nu(\mathcal{R}_k)>0$. Moreover, thanks to \autoref{thm:weakAhlforsregularity} and to the assumption $\nu\ll\meas$, we can also say that $\nu(\mathcal{R}_{k}^*)>0$.\newline
		We want to prove that, if this is the case, we can find a compact $P\subset\mathcal{R}_{n}^{*}$ such that $P=\cup_{i\in\setN} P_i$, where $\Phi\restr_{P_i}$ is $\dist_G/\dist_{\bar{G}}$-Lipschitz for any $i\in\setN$, $\Phi(P)\subset\mathcal{R}_{k}^{*}$ and $\meas(\Phi(P))>0$. This would contradict what we obtained in step 1 above, since by \autoref{thm:rectifiability} we know that $\meas\res\mathcal{R}_k^{*}$ is absolutely continuous w.r.t. $\mathcal{H}^k$ and $\mathcal{H}^k(\Phi(P)\cap\mathcal{R}_{k}^{*})=0$.
		
		We are assuming that $\nu(\mathcal{R}_{k}^{*})=\Phi_{\sharp}\mu(\mathcal{R}_k^*)>0$, hence $\mu(\Phi^{-1}(\mathcal{R}_k^*))=\mu(\Phi^{-1}(\mathcal{R}_k^{*})\cap\mathcal{S}_{n}^{*})>0$, since $\mu$ is concentrated on $\mathcal{R}_{n}$ and therefore it is concentrated on $\mathcal{S}_n^{*}$. Thus, the inner regularity of $\mu$ and the assumption on $\Phi$ grant that we can find a compact $P\subset \Phi^{-1}(\mathcal{R}_k^*)\cap\mathcal{S}_n^{*}$ such that $\mu(P)>0$ and $P$ is the union of countably many subsets where $\Phi$ is $\dist_G/\dist_{\bar{G}}$-Lipschitz. It remains to prove that $\meas(\Phi(P))>0$.
		To this aim, observe that
		\begin{equation*}
		0<\mu(P\cap\Phi^{-1}(\mathcal{R}_k^*))\le\mu(\Phi^{-1}(\Phi(P)\cap\mathcal{R}_k^*))=\Phi_{\sharp}\mu(\Phi(P)\cap\mathcal{R}_{k}^*).
		\end{equation*}
		The claimed conclusion $\meas(\Phi(P)\cap\mathcal{R}_k^*)>0$ follows recalling that $\nu=\Phi_{\sharp}\mu\ll\meas$.
	\end{proof}
	
	\subsection{Regularity of vector fields drifting $W_2$-geodesics}\label{subsec:regularityintheinterior}
	
	In \autoref{thm:LewyStampacchia} below, which is \cite[Theorem 3.13]{GigliMosconi15}, we state a version of the so-called Lewy-Stampacchia inequality. It will be the key tool in order to apply the regularity theory of Lagrangian Flows we developed in \autoref{sec:Gregularity} to vector fields drifting $W_2$-geodesics. 
	
	Below we will indicate by $l_{K,N}:[0,+\infty)\to[0,+\infty)$ the continuous function, whose explicit expression will be of no importance for our purposes, appearing in the Laplacian comparison theorem (see \cite{Gigli15} and \cite[Theorem 3.5]{GigliMosconi15}).\\ 
	We will denote by $\mathcal{Q}_t$ the Hopf-Lax semigroup defined by
		\begin{equation*}
		\mathcal{Q}_tf(x):=\inf_{y\in X}\left\lbrace f(y)+\frac{\dist^2(x,y)}{2t}\right\rbrace \quad\text{for any }(x,t)\in X\times(0,+\infty), 
		\end{equation*}
		for any $f:X\to\setR\cup\set{+\infty}$,
		referring to \cite{AmbrosioGigliSavare13,AmbrosioColomboDiMarino15} for a detailed discussion about its properties.
	\begin{theorem}\label{thm:LewyStampacchia}
		Let $(X,\dist,\meas)$ be an $\RCD(K,N)$ metric measure space for some $K\in\setR$ and $1<N<+\infty$. Let $\mu_0, \mu_1\in\Prob(X)$ be absolutely continuous w.r.t. $\meas$ and with bounded supports, $(\mu_t)_{t\in[0,1]}$ be the $W_2$-geodesic connecting them and  $\varphi:X\to\setR$ be a Kantorovich potential inducing it (which we can assume to be Lipschitz and with compact support).
		
		Then, for every $t\in(0,1)$, there exists $\eta_t\in\Lip(X)$ with compact support, uniformly w.r.t. time, and such that
		\begin{equation}\label{eq:unchanged}
		-\mathcal{Q}_t(-\varphi)\le\eta_t\le\mathcal{Q}_{(1-t)}(-\varphi^c),
		\end{equation}
		\begin{equation*}
		(t\eta_t)^{cc}(x)=t\eta_t(x)\quad\text{and}\quad (-(1-t)\eta_t)^{cc}(x)=-(1-t)\eta_t(x) \quad\text{for any $x\in\supp\mu_t$}
		\end{equation*}
		and $\eta_t\in D(\Delta)$ with
		\begin{equation}\label{eq:mainLewyStampacchia}
		\norm{\Delta\eta_t}_{L^{\infty}}\le\max\left\lbrace \frac{l_{K,N}(2\sqrt{t\norm{\varphi}_{L^ {\infty}}})}{t},\frac{l_{K,N}(\sqrt{2(1-t)\norm{\varphi}_{L^{\infty}}})}{1-t}\right\rbrace. 
		\end{equation}
	\end{theorem}
	
	\begin{remark}\label{rm:regularoutside}
		We remark that, passing from the starting potentials to the regularized potentials $\eta_t$, we gain global regularity without modifying the potential in the support of $\mu_t$, as it follows from \eqref{eq:unchanged} recalling that $-\mathcal{Q}_t(-\varphi)=\mathcal{Q}_{(1-t)}(-\varphi^c)$ on $\supp\mu_t$ (see \cite[Proposition 3.6]{GigliMosconi15}).
	\end{remark}

	In view of the applications of the forthcoming \autoref{subsec:conclusion}, in \autoref{prop:consequencesimproved} below we collect some consequences of the improved regularity of Kantorovich potentials.
	
	\begin{proposition}\label{prop:consequencesimproved}
		Let $(X,\dist,\meas)$ be an $\RCD(K,N)$ m.m.s. for some $K\in\setR$ and $1<N<+\infty$. Let $\mu_0,\mu_1\in\Prob(X)$ be absolutely continuous w.r.t. $\meas$ with bounded densities and bounded supports. Then there exists a time dependent vector field $(b_t)_{t\in(0,1)}$ such that the following conditions are satisfied:
		\begin{itemize}
			\item[(i)] for any $t\in(0,1)$ it holds $b_t\in W^{1,2}_{C}(TX)$ and
			\begin{equation}\label{eq:integralbounds}
			\int_{\epsilon}^{1-\epsilon}\left\lbrace \norm{\nabla_{\sym}b_s}_{L^2(X,\meas)}+\norm{\div b_s}_{L^2(X,\meas)}\right\rbrace \di s<+\infty\quad\text{for any $0<\epsilon<1$};
			\end{equation}  
			\item[(ii)] for any $0<s<1$, denoting by $(\XX_s^t)_{t\in[s,1)}$ the Regular Lagrangian flow of $(b_t)_{t\in(s,1)}$, it holds that $\left(\XX_s^t\right)_{\sharp}\mu_s=\mu_t$ for any $s\le t<1$.
		\end{itemize}  
		
	\end{proposition}
	
	\begin{proof}
		We claim that the vector field $(\nabla\eta_s)_{s\in(0,1)}$ (where $\eta_s$ are the regularized Kantorovich potentials we introduced in \autoref{thm:LewyStampacchia}) does the right job.
		
		Observe that, for any $s\in(0,1)$, it holds that $\nabla\eta_t$ is bounded with bounded support, as it was stated in \autoref{thm:LewyStampacchia}, Moreover, since $\eta_s\in D(\Delta)$, \cite[Corollary 3.3.9]{Gigli14} implies that $\eta_s\in W^{2,2}(X,\dist,\meas)$ which yields, in turn, $\nabla\eta_s\in W^{1,2}_C(TX)$.\newline   
		Let us check \eqref{eq:integralbounds}. To this aim we observe that the construction described in the proof of \cite[Theorem 3.13]{GigliMosconi15} grants that the regularized potentials can be chosen to have all support contained in the same compact set $C\subset X$.
		Hence 
		\begin{equation*}
		\int_{\epsilon}^{1-\epsilon}\norm{\div b_s}_{L^2}\di s\le \int_{\epsilon}^{1-\epsilon}\max\left\lbrace \frac{l_{K,N}(2\sqrt{s\norm{\varphi}_{L^ {\infty}}})}{s},\frac{l_{K,N}(\sqrt{2(1-s)\norm{\varphi}_{L^{\infty}}})}{1-s}\right\rbrace\meas(C)\di s<+\infty.
		\end{equation*}
		Dealing with the bound of the Sobolev norm we recall that \cite[Corollary 3.3.9]{Gigli14} provides the quantitative bound
		\begin{equation}\label{eq:laptosob}
		\int_X\abs{\Hess f}^2_{\HS}\di\meas\le\int_X\left\lbrace (\Delta f)^2-K\abs{\nabla f}^2\right\rbrace \di\meas
		\end{equation}
		for any $f\in D(\Delta)$. Recalling that the regularized potentials can also be chosen uniformly Lipschitz on $(0,1)$, the sought bound for $\int_{\epsilon}^{1-\epsilon}\norm{\nabla_{\sym}b_s}_{L^2}\di s$ follows applying \eqref{eq:laptosob} to the functions $\eta_s$, taking into account the $L^{\infty}$-bound for the laplacian \eqref{eq:mainLewyStampacchia} and the uniform boundedness of the supports.

		Passing to the proof of (ii), observe that the very construction of the regularized Kantorovich potentials (see \autoref{rm:regularoutside}) $\eta_s$ grants that $(\mu_s,b_s)_{s\in(0,1)}$ is a solution to the continuity equation with uniformly bounded density (the uniform bound for the densities is a consequence of \autoref{prop:finepropertiesgeodesics}).
		Moreover, \eqref{eq:integralbounds} grants, via the results of \cite{AmbrosioTrevisan14}, that, for any $0<s<t<1$, there exists a unique Regular Lagrangian flow $(\XX_s^r)_{r\in[s,t]}$ of $(b_r)_{r\in(s,t)}$. Observe that, by the very definition of RLF, also $r\mapsto (\XX_s^r)_{\sharp}\mu_s$ is a solution, with uniformly bounded density and initial datum $\mu_s$, to the continuity equation induced by $(b_r)_{r\in(s,t)}$. Hence $(\XX_s^t)_{\sharp}\mu_s=\mu_t$ for any $s\le t<1$, since the conclusion in (i) implies that the continuity equation induced by $(b_r)_{r\in(s,t)}$ has a unique solution with uniformly bounded density (again by the results of \cite{AmbrosioTrevisan14}). 
	\end{proof}
	
	\subsection{Conclusion}\label{subsec:conclusion}
	
	Below we state and prove our main theorem concerning the constancy of the dimension for finite dimensional $\RCD$ metric measure spaces.
	Its proof involves almost all the ingredients we developed so far in the note.
	
	\begin{theorem}[Constancy of the dimension]\label{thm:constancyofdimension}
		Let $(X,\dist,\meas)$ be an $\RCD(K,N)$ m.m.s. for some $K\in\setR$ and $1\le N<+\infty$. Then there is exactly one regular set $\mathcal{R}_n$ having positive $\meas$-measure in the Mondino-Naber decomposition of $(X,\dist,\meas)$.
	\end{theorem}
	
	\begin{proof}
		As we already observed, the statement is not affected by tensorization with Euclidean factors. Thus we assume without loss of generality that $(X,\dist,\meas)$ satisfies either assumption \autoref{ass:doublingG} or assumption \autoref{assumptionnegative}.
		
		Suppose by contradiction that there exist $3\le k<n$ such that $\meas(\mathcal{R}_k),\meas(\mathcal{R}_n)>0$. Then we can find $\eta_0,\eta_1\in\Prob(X)$, absolutely continuous w.r.t. $\meas$, with bounded densities and bounded supports, such that $\eta_0(\mathcal{R}_n)=1$ and $\eta_1(\mathcal{R}_k)=1$.\newline
		Let $(\eta_r)_{r\in[0,1]}$ be the $W_2$-geodesic joining them and recall from \autoref{prop:finepropertiesgeodesics} that the measures $\eta_r$ are absolutely continuous w.r.t. $\meas$, with uniformly bounded densities and uniformly bounded supports. Applying the second conclusion in \autoref{prop:finepropertiesgeodesics}, we can also conclude that there exist $0<s<t<1$ such that $\eta_s(\mathcal{R}_n)>1/2$ and $\eta_t(\mathcal{R}_k)>1/2$. Calling $\Pi\in\Prob(\Geo(X))$ the unique geodesic plan lifting $(\eta_r)_{r\in[0,1]}$, it follows from what we just observed that
		\begin{equation*}
		\Pi(\set{\gamma\in\Geo(X): \gamma(s)\in\mathcal{R}_n\text{ and }\gamma(t)\in\mathcal{R}_k})>0.
		\end{equation*}
		Hence, setting
		\begin{equation*}
		A:=\set{\gamma\in\Geo(X): \gamma(s)\in\mathcal{R}_n\text{ and }\gamma(t)\in\mathcal{R}_k}, \quad \bar{\Pi}:=\frac{1}{\Pi(A)}\Pi\llcorner A\quad\text{and } \mu_r:=(e_r)_{\sharp}\bar{\Pi},
		\end{equation*}
		for any $r\in[0,1]$, we obtain a $W_2$-geodesic $(\mu_r)_{r\in[0,1]}$ which joins probabilities with bounded support and bounded densities w.r.t. $\meas$ and such that $\mu_s$ is concentrated on $\mathcal{R}_n$ and $\mu_t$ is concentrated on $\mathcal{R}_k$. \newline
		Next we apply \autoref{prop:consequencesimproved} to the $W_2$-geodesic $(\mu_r)_{r\in[0,1]}$ to obtain that, with the notation therein introduced, $\XX_t^s$ is the RLF of a Sobolev time dependent vector field satisfying the assumptions of \autoref{thm: sobolev vectorfield, G-regularity} (or \autoref{thm: sobolev vectorfield, barG-regularity}). Hence $\XX_t^s$ is a $\dist_{G}/\dist_{\bar{G}}$-Lusin Lipschitz map such that $\left(\XX_s^{t}\right)_{\sharp}\mu_s=\mu_t$ and, applying \autoref{thm:conservationhausdim}, we eventually reach a contradiction. 
	\end{proof}
	
	Let us conclude this section stating and proving some corollaries of \autoref{thm:constancyofdimension} that might be useful for the future investigation in this framework.
	
	As a consequence of \autoref{thm:constancyofdimension}, the quantitative $\dist_G/\dist_{\bar{G}}$-Lusin Lipschitz regularity results for Lagrangian flows \autoref{thm: sobolev vectorfield, G-regularity} and \autoref{thm: sobolev vectorfield, barG-regularity} can be turned into the following qualitative $\dist$-Lusin Lipschitz regularity statement.
	
	\begin{theorem}\label{thm:LusinLip}
		Let $(X,\dist,\meas)$ be an $\RCD(K,N)$ m.m.s for some $K\in\setR$ and $1<N<+\infty$. Let $b$ be a compactly supported bounded vector field such that $b_t\in W^{1,2}_{C,s}(TX)$ for a.e. $t\in(0,T)$ and that $|\nabla_{\sym} b_t|\in L^{1}((0,T);L^2(X,\meas))$ and $\div b_t\in L^1((0,T); L^{\infty}(X,\meas))$. Let $\XX_t$ be the regular Lagrangian flow associated to $b$. Then, for any $t\in[0,T]$ the map $\XX_t:X\to X$ is $\dist$-Lusin Lispchitz regular (compare with \autoref{def:dGLusinLip}).
	\end{theorem}
	
	\begin{proof}
		To begin we remark that, arguing as in \cite[Subsection 3.4]{BrueSemola18}, we can reduce ourselves to the case of an $\RCD(K,N)$ space satisfying assumption \autoref{assumptionnegative}.
		
		Let us fix $t\in[0,T]$ and call $\Phi:=\XX_t$. Applying \autoref{thm: sobolev vectorfield, barG-regularity}, exploiting \autoref{lemma:asymptotics}, the bounded compression property of the flow and \autoref{thm:constancyofdimension} we get
		\begin{equation}
		\lim_{r\to 0}\esssup_{y\in B(x,r)}\frac{\dist(\Phi(x),\Phi(y))}{\dist(x,y)}<+\infty
		\end{equation}
		for $\meas$-a.e. $x\in X$. 
		
		The claimed conclusion follows applying \cite[Theorem 3.1.8]{Federer69}, taking into account the rectifiability of $(X,\dist,\meas)$ and the bounded compression property of $\Phi$ once more. 
	\end{proof}
	
	In view of the results of \cite{GigliPasqualetto16b}, the constancy of the dimension, that we stated and proved in \autoref{thm:constancyofdimension} at the level of the Mondino-Naber decomposition, can be equivalently rephrased at the level of the dimensional decomposition of the tangent module $L^2(TX)$. 
	We refer to \cite[Theorem 1.4.11]{Gigli14} for the basic terminology and results about this topic.
	
	\begin{corollary}\label{thm:notionsofdimension}
		Let $(X,\dist,\meas)$ be an $\RCD(K,N)$ m.m.s. for some $K\in\setR$ and $1<N<+\infty$. Let $n\in\setN$, $1\le n\le N$ be such that $\mathcal{R}_n$ is the unique regular set with positive measure in the Mondino-Naber decomposition of $(X,\dist,\meas)$. Then the tangent module $L^2(TX)$ has constant dimension equal to $n$.
	\end{corollary}
	
	\begin{proof}
		The result directly follows from \cite[Theorem 3.3]{GigliPasqualetto16b} and \autoref{thm:constancyofdimension}.
	\end{proof}
	
	\begin{remark}
		With the notation introduced in \autoref{thm:notionsofdimension}, one has that $n$ is the analytic dimension of $(X,\dist,\meas)$ (see \cite[Definition 2.10]{Han18}). 
		
		To let the picture about the different notions of dimension introduced in the literature so far be more complete, we also point out that $n$ is also the dimension of $(X,\dist,\meas)$ according to \cite[Definition 4.1]{Kitabeppu18}. Indeed, as it is observed in \cite[Remark 4.14]{Kitabeppu18}, if $\mathcal{R}_n$ is the unique regular set of positive measure, the results of \cite{Kitabeppu18} grant that it is also the non empty regular set of maximal dimension.  
		
		Up to our knowledge, the problem of whether $n$ is the Hausdorff dimension of $(X,\dist)$ or not is still open also in the case of collapsed Ricci limit spaces (see \cite[Remark 1.3]{ColdingNaber12}) essentially due to the lack of knowledge about the Hausdorff dimension of the singular set.
	\end{remark}
	
	Eventually we give a positive answer to a conjecture raised in \cite[Remark 1.13]{DePhilippisGigli17}. As it is therein observed, its validity follows from the fact that the tangent module has constant dimension exploiting the results of \cite{Han18}. We wish to thank Nicola Gigli for pointing this out to us.

	\begin{theorem}
		Let $(X,\dist,\meas)$ be an $\RCD(K,N)$ m.m.s. for some $K\in\setR$ and $1\le N<+\infty$. Assume that (referring to \cite{Gigli14} for the terminology) $H^{2,2}(X,\dist,\meas)=D(\Delta)$ and 
		\begin{equation}\label{eq:tracelapl}
		\tr\Hess f=\Delta f,\quad\text{for any $f\in H^{2,2}(X,\dist,\meas)$}.
		\end{equation} 
		Then, there exists $n\in\setN$, $1\le n\le N$ such that, with the terminology introduced in \cite[Definition 1.10]{DePhilippisGigli17}, $(X,\dist,\meas)$ is a weakly non collapsed $\RCD(K,n)$ m.m.s..
	\end{theorem}
	
	\begin{proof}
		We wish to prove that the statement holds true with $n$ equal to the dimension of the unique regular set with positive measure in the Mondino-Naber decomposition of $(X,\dist,\meas)$.
		
		By the very definition of weakly non collapsed $\RCD(K,n)$ m.m.s., we just need to prove that $(X,\dist,\meas)$ is an $\RCD(K,n)$ m.m.s.. The sought conclusion will follow after proving that we are under the assumptions of \cite[Theorem 4.3]{Han18}. To this aim, observe that the first and the second assumption in the statement of \cite[Theorem 4.3]{Han18} are fulfilled thanks to our choice of $n$ and the validity of \eqref{eq:tracelapl}. To see that also the third one is satisfied, it suffices to observe that
		\begin{equation*}
		\boldmath{\Ric}_n(\nabla f,\nabla f)=\boldsymbol{\Gamma}_2(\nabla f,\nabla f)-\abs{\Hess f}_{\HS}^2\meas\ge K\abs{\nabla f}^2\meas, 
		\end{equation*}  
		for any $f\in H^{2,2}(X,\dist,\meas)$, where the equality above follows from \eqref{eq:tracelapl} and the inequality from the assumption that $(X,\dist,\meas)$ is an $\RCD(K,N)$ m.m.s..
		
	\end{proof}
	
	\appendix
	\section{Appendix}
	This appendix is dedicated to the proof of a general result about the structure of regular Lagrangian flows associated to vector fields with product structure over product spaces. As a corollary we will obtain that, for $\meas\times\meas$-a.e. $(x,y)\in X\times X$, the map $t\mapsto G(\XX_t(x),\XX_t(y))$ is differentiable $\leb^1$-a.e., with the explicit and expected formula for the derivative we used in the proof of \autoref{th 1}.
	
	Let $(X,\dist_X,\meas_X)$ and $(Y,\dist_Y,\meas_Y)$ be $\RCD(K,\infty)$ m.m. spaces. Let $Z:=X\times Y$ be endowed with the product m.m.s. structure, namely
	\begin{equation*}
	\dist_Z^2((x,y),(x',y')):=\dist_X^2(x,x')+\dist_Y^2(y,y')\quad\text{and}\quad \meas_Z:=\meas_X\times\meas_Y
	\end{equation*}
	and recall from \cite{AmbrosioGigliSavare14,AmbGigliMondRaj12} that $(X,\dist_Z,\meas_Z)$ is an $\RCD(K,\infty)$ m.m.s itself.\newline
	We will denote by $\pi_X$ and $\pi_Y$ the canonical projections from $Z$ onto $X$ and $Y$ respectively. This being said we introduce the so-called algebra of tensor products by
	\begin{equation*}
	\mathcal{A}:=\left\lbrace \sum_{j=1}^ng_j\circ\pi_Xh_j\circ\pi_Y:\ g_j\in W^{1,2}_{\loc}\cap L^{\infty}_{\loc}(X)\text{ and }h_j\in W^{1,2}_{\loc}\cap L^{\infty}_{\loc}(Y)\ \forall j=1,\dots,n\right\rbrace. 
	\end{equation*}
	
	Below we state and prove a useful density result concerning the density of the algebra of tensors.
	
	\begin{theorem}\label{thm:densityofproductalgebra}
		Let $X,Y$ and $Z$ be as above. Then, for any $f\in W^{1,2}_{\loc}(Z,\dist_Z,\meas_Z)\cap L^{\infty}_{\loc}(Z,\meas)$ and for any compact $P\subset Z$, there exists a sequence $(f)_{n\in\setN}$ in $\mathcal{A}$ with $\norm{f_n}_{L^{\infty}(P)}$ uniformly bounded and such that $\norm{f_n-f}_{L^2(P,\meas_Z)}+\norm{\abs{\nabla (f_n-f)}}_{L^2(P,\meas_Z)}\to 0$ as $n\to\infty$.
	\end{theorem}
	
	\begin{proof}
		Let us denote by $\bar{\mathcal{A}}$ the set of functions $f\in W^{1,2}_{\loc}(Z,\dist_Z,\meas_Z)\cap L^{\infty}_{\loc}(Z,\meas)$ for which the statement of the theorem holds true.	
		Let $\mathcal{A}_{\dist}$ be the smallest subset of $\Lip_b(X)$ containing truncated distances from points of $Z$ and closed with respect to sum, product and lattice operations, let $\mathcal{A}_{\dist\bs}\subset W^{1,2}(Z,\dist_Z,\meas_Z)\cap L^{\infty}(Z,\meas_Z)$ be the subalgebra of $\mathcal{A}_{\dist}$ made by functions with bounded support.
		In \cite[Theorem B.1]{AmbrosioStraTrevisan17} it is proved that $\mathcal{A}_{\dist\bs}$ is dense in $W^{1,2}(Z,\dist_Z,\meas_Z)$ and it is straightforward to check that one can approximate any bounded function in $W^{1,2}(Z,\dist_Z,\meas_Z)$ with a sequence of uniformly bounded functions in $\mathcal{A}_{\dist\bs}$. 
		Hence, to get the stated conclusion, it is sufficient to
		prove that $\dist_Z(z,\cdot)\wedge k\in \bar{\mathcal{A}}$ for any $z\in Z$, for any $k\ge 0$, and the implication $f,g\in \bar{\mathcal{A}}\implies f\wedge g\in \bar{\mathcal{A}}$.
		
			Let us first prove that $\dist_Z(z,\cdot)\in \bar{\mathcal{A}}$ for any $z\in Z$.
		For any natural $n\ge 1$ let $(h_n^k)_{k\in\setN}$ be a sequence of polynomials converging to $t\mapsto\sqrt{1/n+t}$ in $C^{1}_{\loc}([0,+\infty))$ as $k\to\infty$. Let us fix $z\in Z$, it is simple to see that $h_n^k(\dist_Z(z,\cdot)^2)$ converges in $W^{1,2}_{\loc}(Z,\dist_Z,\meas_Z)\cap L^{\infty}_{\loc}(Z,\meas_Z)$ to $\sqrt{1/n+\dist_Z^2(z,\cdot)}$ when $k\to \infty$ and that $\sqrt{1/n+\dist_Z^2(z,\cdot)}\to \dist_Z(z,\cdot)$, in the same topology, when $n\to\infty$.
		Observe that the very definition of $\dist_Z$ yields $\dist_Z(z,w)^2=\dist_X(\pi_X(z),\pi_X(w))^2+\dist_Y(\pi_Y(z),\pi_Y(w))^2$ for any $w\in Z$, therefore $h_n^k(\dist_Z(z,\cdot)^2)\in \mathcal{A}$.
		
		Let us now prove the implication $g\in \bar{\mathcal{A}}\implies |g|\in \bar{\mathcal{A}}$.
		With this aim, let us fix $g\in\bar{\mathcal{A}}$ and a sequence $g_m\in \mathcal{A}$ converging to $g$ in $W^{1,2}_{\loc}(Z,\dist_Z,\meas_Z)\cap L^{\infty}_{\loc}(Z,\meas_Z)$ when $m\to\infty$.
		Set $g^k_{n,m}:=h_n^k\circ g_m^2$, we have $g^k_{n,m}\in \mathcal{A}$ and it is easy to check that it converges to $\sqrt{1/n + g_m^2}$ in $W^{1,2}_{\loc}(Z,\dist_Z,\meas_Z)\cap L^{\infty}_{\loc}(Z,\meas_Z)$ as $k\to\infty$. Moreover $\sqrt{1/n + g_m^2}\to |g_m|$ in $W^{1,2}_{\loc}(Z,\dist_Z,\meas_Z)\cap L^{\infty}_{\loc}(Z,\meas_Z)$ when $n\to\infty$ and eventually $|g_m|\to |g|$, in the same topology, when $m\to \infty$. By a diagonal argument, we recover the sought approximating sequence. 
		
			Finally we exploit the identity
			\begin{equation*}
			a\wedge b=\frac{|a+b|-|a-b|}{2},\qquad\forall   a,b\in[0,\infty),
			\end{equation*} 
			to deduce that $\dist_Z(z,\cdot)\wedge k\in \bar{\mathcal{A}}$ for any $z\in Z$, for any $k\ge 0$ and the implication $f,g\in \bar{\mathcal{A}}\implies f\wedge g\in \bar{\mathcal{A}}$.
	\end{proof}
	
	Let us consider now $b^X_t\in L^1((0,T);L^2(TX))$ and $b^Y_t\in L^1((0,T);L^2(TY))$. We introduce the ``product'' vector field $b^Z_t$ by saying that, for every $f\in W^{1,2}(Z,\dist_Z,\meas_Z)$,
	\begin{equation}\label{z11}
	b^Z_t\cdot\nabla f(x,y):= b_t^X\cdot \nabla f_y(x)+b_t^Y\cdot \nabla f_x(y),
	\end{equation}
	for $\meas_Z$-a.e. $(x,y)\in Z$, where $f_x(y):=f(x,y)$, $f_y(x):=f(x,y)$ and we are implicitly exploiting the tensorization of the Cheeger energy (see \cite{AmbrosioGigliSavare14,AmbGigliMondRaj12}). It is simple to check that $b^Z_t \in L^1((0,T);L^2_{\loc}(Z,\meas_Z))$ and 
	\begin{equation*}
	|b_t^Z|^2(x,y)\leq |b_t^X|^2(x)+|b_t^Y|^2(y),\qquad \text{for}\ \meas\times\meas\text{-a.e.}\ (x,y)\in X\times Y.
	\end{equation*}
	
	\begin{proposition}\label{prop: product of vector fields}
		Let $b^X_t$ and $b^Y_t$ be as above and $\XX^X_t$ and $\XX^Y_t$ be regular Lagrangian flows associated to $b_t^X$ and $b_t^Y$, respectively. Then 
		\begin{equation*}
		\XX^Z_t(x,y):=(\XX^X_t(x),\XX^Y_t(y))
		\end{equation*}
		is a regular Lagrangian flow associated to $b^{Z}_t$.
	\end{proposition}
	
	\begin{proof}
		We need to check the validity of the three defining conditions in \autoref{def:Regularlagrangianflow}.\newline 
		The first one is trivial and the bounded compressibility property of $\XX^Z_t$ is a direct consequence of the bounded compressibility property of $\XX^X_t$ and $\XX^Y_t$.\newline
		Dealing with the third one, we observe that, thanks to \autoref{thm:densityofproductalgebra} and \autoref{remark: dense class definition RLF}, it is sufficient to check its validity testing it for any $f\in\mathcal{A}$. 
		Moreover, by the linearity of \eqref{eq: RLF condition 3} w.r.t. the test function, we can assume without loss of generality that $f=g\circ\pi_Xh\circ\pi_Y$, with $g\in W^{1,2}_{\loc}(X,\dist_X,\meas_X)\cap L^{\infty}_{\loc}(X,\meas_X)$ and $h\in W^{1,2}_{\loc}(Y,\dist_Y,\meas_Y)\cap L^{\infty}_{\loc}(Y,\meas_Y)$. 
		We need to prove that for $\meas_Z$-a.e. $(x,y)\in X\times Y$ the map $z\mapsto f(\XX^Z_t(z))$ belongs to $W^{1,1}((0,T))$ and has derivative given by
		\begin{equation}\label{eq:flowproduct}
		\frac{\di}{\di t}f(\XX_t^Z(z))=b^Z_t\cdot\nabla^Zf(\XX_t(z))\quad \text{for $\leb^1$-a.e. $t\in(0,T)$.}
		\end{equation}	
		To this aim we observe that, since $\XX^X_t$ and $\XX^Y_t$ are regular Lagrangian flows of $b^X_t$ and $b_t^Y$ respectively, it holds that the maps $t\mapsto g(\XX_t^X(x))$ and $t\mapsto h(\XX^Y_t(y))$ are bounded and belong to $W^{1,1}((0,T))$ for $\meas_X$-a.e. $x\in X$ and $\meas_Y$-a.e. $y\in Y$ respectively. Moreover
		\begin{equation*}
		\frac{\di}{\di t}g(\XX^X_t(x))=b^X_t\cdot\nabla g(\XX_t^X(x))\quad\text{and }\ \frac{\di}{\di t}h(\XX^Y_t(y))=b^Y_t\cdot\nabla h(\XX_t^Y(y))\quad\text{for $\leb^1$-a.e. $t\in(0,T)$,}
		\end{equation*}
		for $\meas_X$-a.e. $x\in X$ and $\meas_Y$-a.e. $y\in Y$, respectively.
		Applying Fubini's theorem and the Leibniz rule we obtain that, for $\meas_X\times\meas_Y$-a.e. $(x,y)\in X\times Y$, the map $t\mapsto g(\XX_t^X(x))h(\XX^Y_t(y))$ belongs to $W^{1,1}((0,T))$, moreover
		\begin{align*}
		\frac{\di}{\di t}\left(g(\XX^X_t(x))h(\XX^Y_t(y))\right) = & \left(\frac{\di}{\di t}g(\XX^X_t(x))\right) h(\XX^Y_t(y)) +g(\XX^X_t(x))\left(\frac{\di}{\di t}  h(\XX^Y_t(y))\right)\\
		= &   h(\XX^Y_t(y)) b^X_t\cdot \nabla g(\XX^X_t(x))+ g(\XX^X_t(x))b^Y_t\cdot \nabla h(\XX^Y_t(y))\\
		= & b_t^Z\cdot\nabla f(\XX^Z_t(x,y)),
		\end{align*}
		for $\leb^1$-a.e. $t\in(0,T)$, which implies \eqref{eq:flowproduct}.
	\end{proof}
	
	The following corollary of \autoref{prop: product of vector fields} plays an important role in the proof of \autoref{th 1}.

	\begin{corollary}\label{cor: coupled derivative with G} Let $(X,\dist,\meas)$ be an $\RCD(0,N)$ m.m.s. satisfying assumption \autoref{assumption: good definition of G}. Let moreover $b\in L^{1}((0,T);L^{2}(TX))$ and $\XX_t$ be a regular Lagrangian flow associated to $b$. Then, the map
		\begin{equation*}
		t\mapsto G(\XX_t(x),\XX_t(y))
		\end{equation*}
		belongs to $W^{1,1}((0,T))$ for $\meas\times \meas$-a.e. $(x,y)\in X\times X$ and its derivative is given by
		\begin{equation*}
		\frac{\di}{\di t}  G(\XX_t(x),\XX_t(y))= b_t\cdot \nabla G_{\XX_t(x)}(\XX_t(y))+b_t\cdot \nabla G_{\XX_t(y)}(\XX_t(x)),
		\end{equation*}
		for $\leb^1$-a.e. $t\in (0,T)$.
	\end{corollary}
	
	\begin{proof}
		Let us start observing that $G^{\eps}\in W^{1,2}_{\loc}(X\times X)$ for any $\epsilon>0$ (actually it has locally bounded weak upper gradient as one can prove with the same techniques introduced in the proof of \autoref{prop:estimateG}, taking into account also \autoref{rm:continuityFH}).\newline
		It follows from \autoref{prop: product of vector fields}, applied with $X=Y$ and $b^X=b^Y=:b$, that
		\begin{equation}\label{eq:z15}
		G^{\eps}(\XX_t(x),\XX_t(y))-G^{\eps}(x,y)= \int_0^t \left\lbrace b_s\cdot \nabla G^{\eps}_{\XX_s(x)}(\XX_s)(y)+b_s\cdot \nabla G^{\eps}_{\XX_s(y)}(\XX_s(x))\right\rbrace  \di s,
		\end{equation}
		for $\meas\times\meas$-a.e. $(x,y)\in X\times X$ for every $t\in[0,T]$.
		We wish to pass to the limit as $\epsilon\downarrow 0$ in \eqref{eq:z15} to obtain that for any $t\in[0,T]$ it holds
		\begin{equation}\label{eq:z16}
		G(\XX_t(x),\XX_t(y))-G(x,y)= \int_0^t\left\lbrace  b_s\cdot \nabla G_{\XX_s(x)}(\XX_s(y))+b_s\cdot \nabla G_{\XX_s(y)}(\XX_s(x))\right\rbrace  \di s,
		\end{equation}	
		for $\meas\times\meas$-a.e. $(x,y)\in X\times X$. The sought conclusion would easily follow.
		To this aim let us observe that the left hand side in \eqref{eq:z15} converges to $ G(\XX_t(y),\XX_t(x))-G(x,y)$ in $L^{1}_{\loc}(X\times X,\meas\times\meas)$. 
		Thus, it suffices to prove that the right hand side in \eqref{eq:z15} converges to
		\begin{equation*}
		\int_0^t \left\lbrace b_s\cdot\nabla G_{\XX_s(x)}(\XX_s(y))+b_s\cdot \nabla G_{\XX_s(y)}(\XX_s(x))\right\rbrace \di s \qquad \text{in}\ L^1_{\loc}(X\times X,\meas\times\meas).
		\end{equation*}
		To this aim we fix $z\in X$ such that  $\dist(\XX_s(z),z)\leq \norm{b}_{L^{\infty}}t$ for every $s\in [0,t]$ (observe that this property holds true for $\meas$-a.e. point). The triangle inequality yields
		\begin{equation}\label{z5}
		\dist(\XX_s(z),\XX_s(y))\leq 2t\norm{b}_{L^{\infty}}+\dist(z,y), \qquad \text{for $\meas$-a.e.}\ y\in X.
		\end{equation}
		Thus, setting $B:=B(z,R)$, for some $R>0$, and $\bar{B}:= B(z,R+2t\norm{b}_{L^{\infty}})$, we have
		\begin{equation}
		(\XX_s)_{\#} (\mathds{1}_B\meas)\leq L \mathds{1}_{\bar{B}}\meas.
		\end{equation}
		The bounded compressibility property of the RLF allows us to estimate
		\begin{align*}
		\int_{B\times B} & \abs{\int_0^t b_s\cdot \nabla G^{\eps}_{\XX_s(x)}(\XX_s(y))\di s-\int_0^t b_s\cdot \nabla G_{\XX_s(x)}(\XX_s(y))\di s}  \di \meas(x)\di \meas(y)\\
		\leq & \int_0^t\int_{B} \int_{B} |b_s|(\XX_s(y))\cdot |\nabla ( G^{\eps}_{\XX_s(x)}- G_{\XX_s(x)})|(\XX_s(y)) \di \meas(y)\di \meas(x)\di s\\
		\leq & L^2t \norm{b}_{L^{\infty}} \int_{\bar{B}}\norm{\nabla (G^{\eps}_x-G_x)}_{L^1(\bar{B})} \di \meas(x).
		\end{align*}
		The last term goes to zero, as a simple application of dominated convergence theorem
		shows (for more details about this step we refer to the proof of \autoref{prop: G-maximal estimate}, where we dealt with a similar term). Arguing similarly for the term $\int_0^t b_s\cdot \nabla G^{\eps}_{\XX_s(y)}(\XX_s(x)) \di s$ we obtain the thesis.
	\end{proof}
	
		\begin{remark}\label{remark: corollaryA3}
			A conclusion analogous to the one stated in \autoref{cor: coupled derivative with G} holds true with $\bar{G}$ in place of $G$ assuming that $(X,\dist,\meas)$ is an $\RCD(K,N)$ m.m.s. satisfying assumption \autoref{assumptionnegative}. To get this result it suffices to argue as in the proof of \autoref{cor: coupled derivative with G} using \autoref{prop:estimatebarG} instead of \autoref{prop:estimateG}.
		\end{remark}

\end{document}